\documentclass[envcountsect,oribibl,orivec]{llncs}
\usepackage{amssymb}
\usepackage{bussproofs}
\def\fCenter{\hspace{1pt}\Rightarrow\hspace{1pt}}
\EnableBpAbbreviations

\usepackage[right=3cm,left=3cm]{geometry}

 \usepackage{graphics}
 \usepackage{graphicx}

\newcommand{\qe }{\hfill $\dashv$ \\}

  \newcommand{\MLJL}{{\sf MLJL}}
 \newcommand{\4}{{\sf S4LP}}
\newcommand{\N}{{\sf S4LPN}}
\newcommand{\s}{{\sf S4}}
 \newcommand{\LP}{{\sf LP}}
 
 \newcommand{\M}{{\mathcal M}}
 \newcommand{\W}{{\mathcal W}}
  \newcommand{\V}{{\mathcal V}}
\newcommand{\RR}{{\mathcal R}}
\newcommand{\CS}{{\sf CS}}
\newcommand{\E}{{\mathcal E}}
\newcommand{\J}{${\sf G3JL}$}
\newcommand{\JJ}{${\sf G3JL^-}$}
\newcommand{\JL}{{\sf JL}}
\newcommand{\ML}{{\sf ML}}
\newcommand{\ww}{{\sf w}}
\newcommand{\vv}{{\sf v}}
\newcommand{\uu}{{\sf u}}

 \def\f{\frac}
 \def\g{\Gamma}
 \def\d{\Delta}
 \def\R{\Rightarrow}
 \def\di{\displaystyle}
 \def\r{\rightarrow}
 \def\b{\Box}

\begin{document}

\title{Labeled Sequent Calculus and Countermodel Construction for Justification Logics }
\author{Meghdad Ghari}
\institute{School of Mathematics,
Institute for Research in Fundamental Sciences (IPM),\\ P.O.Box: 19395-5746, Tehran, Iran\\
\email{ghari@ipm.ir}}

\maketitle
\begin{abstract}
Justification logics are modal-like logics that provide a framework for reasoning about justifications. This paper introduces labeled sequent
calculi for justification logics, as well as for hybrid
modal-justification logics. Using the method due to Sara
Negri, we internalize the Kripke-style semantics of justification
logics, known as Fitting models, within the syntax of the sequent
calculus to produce labeled sequent calculus. We show that our labeled
sequent calculi enjoy a weak subformula property, all of the rules are
invertible and the structural rules (weakening and contraction) and cut
are admissible.  Finally soundness and completeness are
established, and
termination of proof search for some of the labeled systems are shown. We describe a procedure, for some of the labeled systems, which produces a
derivation for valid sequents and a countermodel for non-valid sequents. We also show a model correspondence for justification logics in the context of labeled sequent calculus.\\

\end{abstract}

\section{Introduction}
 Artemov in \cite{A1995,A2001} proposed the \emph{Logic of Proofs},
\LP, to present a provability interpretation for the modal logic
\s~and the intuitionistic propositional logic. \LP~extends the
language of propositional logic by proof terms and expressions of
the form $t:A$, with the intended meaning ``term $t$ is a proof of
$A$''. Terms are constructed from variables and constants by means
of operations on proofs. \LP~can be also viewed as a refinement of
the epistemic logic \s, in which knowability operator $\Box A$
($A$ is known) is replaced by explicit knowledge operators $t:A$
($t$ is a justification for $A$). The exact correspondence between
\LP~and \s~is given by the \textit{Realization Theorem}: all
occurrences of $\b$ in a theorem of \s~can be replaced by suitable
terms to produce a theorem of \LP, and vice versa. Regarding this
theorem, \LP~is called the justification (or explicit) counterpart
of \s.

The justification counterpart of other modal logics were developed
in \cite{AKS1999,Brezhnev2000,KaiRemoRoman2010,Ghari2012-Thesis,GoetschiKuznets2012,Pacut2005,Rubtsova2006}. Various proof methods for the realization theorem are known, such as the syntactic and constructive proofs (see e.g. \cite{A1995,A2001}), semantic and non-constructive proofs
(see e.g. \cite{Fitting2005}), and indirect proofs using embedding (see e.g. \cite{Ghari2012-Thesis,GoetschiKuznets2012}). We give a proof of realization for modal logic {\sf KB} and its extensions using embedding of justification logics.

Combination of modal and justification logics, aka \textit{logics of justifications and belief}, were  introduced in \cite{A2012,AN2004,Fitting2012,KuznetsStuder2012}. In this paper, we introduce \textit{modal-justification logics} which include the previous logics of justifications and belief from \cite{KuznetsStuder2012}, and some new combinations. A modal-justification logic \MLJL~is a combination of a modal logic {\sf ML} and a justification logic \JL~such that \JL~is the justification counterpart of {\sf ML}.

Various proof systems have been developed for \LP~(see
\cite{A2001,Finger2010,Fitting2005,Ghari2011-JLC,Renne2004,Renne2006}), for
intuitionisitc logic of proofs \cite{A2002a,Poggiolesi2010,Poggiolesi2012},
for \4~and \N~(see \cite{Fitting2004,Ghari2011-JLC,Kurokawa2009,Renne2006}), and for justification logics of belief (see \cite{Goetschi2012,Milnikel2012}).
All aforementioned proof systems are cut-free. However, the only
known analytic proof method is Finger's tableau system for
\LP~\cite{Finger2010}.
Moreover, most justification logics still lack cut-free proof
systems. The aim of this paper
will be to present labeled sequent calculus for justification
logics, which enjoy the subformula property and cut elimination.

In a labeled sequent calculus some additional information, such as
possible worlds and accessibility relation of Kripke models, from
semantics of the logic are internalized into the syntax. Thus
sequents in these systems are expressions about semantics of the
logic. We employ Kripke-style models of justification logics,
called Fitting models (cf. \cite{A2008,Fitting2005}), and a method due to
Negri \cite{Negri2005} to present G3-style labeled sequent
calculi for justification logics. Thus the syntax of the labeled
systems of justification logics also contains atoms for
representing evidence function of Fitting models.

Further, we present Fitting models and labeled sequent calculi based on Fitting models for modal-justification logics \MLJL. For \N~(and \4), we  present two labeled sequent calculi one based on Fitting models (Fitting models for \N~will be presented in Section 10.3) and the other based on Artemov-Fitting models \cite{AN2005b}. The latter system has the subformula and subterm properties, whereas these properties does not hold in the former.

In all labeled systems the rules of weakening, contraction, and cut are
admissible, and all rules are invertible. Soundness and completeness of the
labeled systems with respect to Fitting models are also shown.
The method used in the proof of completeness theorem (Theorem
\ref{thm:reduction tree}) gives a procedure to produce
countermodels for non-valid sequents, and helps us to prove the
termination of proof search for some of the labeled systems. Termination of proof search is shown using the analyticity of the labeled systems, i.e. the subformula, sublabel, and subterm properties. Thus decidability
results for some justification and modal-justification logics are achieved, in the case that
finite constant specifications are used.

The paper is organized as follows. In Section 2, we introduce the
axiomatic formulation of modal and justification logics, and show the correspondence between them. We also generalize the Fitting's embedding theorem to all justification logics and show how it can be used to prove the realization theorem. In Section 3, we describe Fitting
models for justification logics and show how a possible evidence function can be extended to an admissible one. In Section 4, we present labeled sequent
calculi for justification logics, and  in Section \ref{sec:Basic properties} we establish the basic properties of these systems. In Section \ref{sec:Analyticity} we show the analyticity of some of the labeled systems. In Section \ref{sec:
Admissibility of structural rules} we prove the admissibility of
structural rules. Then, in Section \ref{sec: Soundness
Completeness} we prove soundness and completeness of the labeled
systems and give a procedure to construct a proof tree or a countermodel for a given sequent. We also show a model correspondence for justification logics in the context of labeled sequent calculus. In Section \ref{section: termination proof search} we
establish the termination of proof search for some of the labeled
systems. Finally, in Section \ref{sec:Other labeled systems} we present Fitting models and labeled sequent calculus for modal-justification logics and for \N, and also a labeled sequent calculus based on Artemov-Fitting models for \N.
\section{Modal and justification logics}
In this section, we recall the axiomatic formulation of modal and
justification logics, and explain the correspondence between them.
\subsection{Modal logics}
Modal formulas are constructed by the following grammar:
\[ A::= P~|~\bot~|~\neg A~|~A\wedge A~|~A\vee A~|~A\rightarrow A~|~\b A,\]
where $P$ is a propositional variable,  $\bot$ is a propositional
constant for falsity. 
The basic modal logic {\sf K} has the following axiom schemes and rules:
\begin{description}
\item[Taut.] All propositional tautologies,
  \item[K.] $\b(A\r B)\r(\b A\r \b B)$,
\end{description}
The rules of inference are \textit{Modus Ponens} and \textit{Necessitation rule}:
\begin{description}
\item[MP.]  from $\vdash A$ and $\vdash A\r B$, infer $\vdash B$.
\item[Nec.]  from $\vdash A$, infer $\vdash \b A$.
\end{description}
Other modal logics are obtained by adding the following axiom schemes to {\sf K} in various combinations:
\begin{description}
\item[T.] $\b A\r A$.
\item[D.] $\b \bot\r\bot$.
\item[4.] $\b A\r\b\b A$.
\item[B.] $\neg A\r\b\neg\b A$.
\item[5.] $\neg \b A\r\b\neg\b A$.
\end{description}
In this paper we consider the following 15 normal modal logics: {\sf K}, {\sf T}, {\sf D}, {\sf K4}, {\sf KB}, {\sf K5}, {\sf KB5}, {\sf K45}, {\sf D5}, {\sf DB}, {\sf D4}, {\sf D45}, {\sf TB}, {\sf S4}, {\sf S5}. The name of each modal logic indicates the list of its axioms, except {\sf S4} and {\sf S5} which can be named {\sf KT4} and {\sf KT45}, respectively.
\subsection{Justification logics}
 The language of justification logics is an
extension of the language of propositional logic by the formulas
of the form $t:F$, where $F$ is a formula and $t$ is a
justification term. \textit{Justification terms} (or
\textit{terms} for short) are built up from (justification)
variables $x, y, z, \ldots$ (possibly with sub- or superscript) and (justification) constants $a,b,c,
\ldots$ (possibly with subscript) using several operations
depending on the logic: application `$\cdot$', sum `$+$', verifier
`$!$', negative verifier `$?$', and weak negative verifier `$\bar{?}$'.
Although the axioms of \JL~reflect the meaning of operation symbols on terms we briefly explain them here. The binary operation $+$ combines two justifications: $s+ t$ is a justification for everything justified by $s$ or by $t$. The binary operation $\cdot$ is used to internalize \textit{modus ponens}: if $s$ is a justification for $A\r B$ and $t$ is a justification for $A$, then $s\cdot t$ is a justification for $B$. The unary operation $!$ is a positive verifier: if $t$ is a justification for $A$, then this fact can be verified by the justification $! t$. The unary operation $?$ is a negative verifier: if $t$ is not a justification for $A$, then this fact can be verified by the justification $? t$. The unary operation $\bar{?}$ is a weak negative verifier: if $A$ is not true, then for every justification $t$ the justification $\bar{?} t$ verifies that $t$ is not a justification for $A$.
A term is called \textit{ground} if
it does not contain any justification variable. The
definition of subterm is in the usual way: $s$ is a subterm of $s,
s+t, t+s, s\cdot t$, $!s$, $\bar{?}s$, and  $?s$.

Justification
formulas (\JL-formulas) are constructed by the following grammar:
\[ A::= P~|~\bot~|~\neg A~|~A\wedge A~|~A\vee A~|~A\rightarrow A~|~t:A,\]
where $P$ is a propositional variable,  $\bot$ is a propositional
constant for falsity, and $t$ is a justification term.

For a \JL-formula $A$, the
set of all subformulas of $A$, denoted by $Sub(A)$, is defined inductively as follows:
$ Sub(P)=\{P\}$, for propositional variable $P$;
$Sub(\bot)=\{\bot\}$; $Sub(A\r B)=\{A\r B\}\cup Sub(A)\cup
Sub(B)$; $Sub(t:A)=\{t:A\}\cup Sub(A)$. For a set $S$ of \JL-formulas, $Sub(S)$ denotes the set of all subformulas of the formulas from $S$.

We now begin with describing the axiom schemes and rules of the basic
justification logic {\sf J}, and continue with other justification
logics. The basic justification logic {\sf J} is the weakest
justification logic we shall be discussing. Other
justification logics are obtained by adding certain axiom schemes
to {\sf J}.
\begin{definition}\label{def: justification logics}
The language of basic justification logic {\sf J} contains the binary operations $\cdot$ and $+$ on terms. Axioms schemes of {\sf J} are:
\begin{description}
\item[Taut.] All propositional tautologies,
 \item[Sum.] Sum axiom, $s:A\r
(s+t):A~,~s:A\r (t+s):A$,
 \item[jK.] \textit{Application axiom},
$s:(A\r B)\r(t:A\r (s\cdot t):B)$,
\end{description}
Justification logic {\sf JT} and its extensions \LP, ${\sf JTB}$,
${\sf JT5}$, ${\sf JTB5}$, ${\sf JT45}$, ${\sf JTB4}$, ${\sf JTB45}$  contain in addition the following axiom scheme:
\begin{description}
\item[jT.] \textit{Factivity axiom}, $t:A\r A$.
 \end{description}
Justification logic {\sf JD} and its extensions ${\sf JD4}$, ${\sf
JD5}$, ${\sf JDB}$, ${\sf JDB4}$, ${\sf JDB5}$, ${\sf JD45}$, ${\sf JDB45}$  contain in addition the following axiom scheme:
\begin{description}
\item[jD.] \textit{Consistency}, $t:\perp \r \perp$.
 \end{description}
The language of justification logic {\sf J4} and its extensions \LP, ${\sf JD4}$,
${\sf JT45}$, ${\sf J45}$, ${\sf JB4}$, ${\sf JTB4}$, ${\sf JDB4}$, ${\sf JB45}$, ${\sf JD45}$, ${\sf JTB45}$, ${\sf JDB45}$  contains in addition the unary operation $!$ on terms, and these logics contain
the following axiom scheme:
\begin{description}
\item[j4.] \textit{Positive introspection axiom}, $t:A\r !t:t:A$,
 \end{description}
 The language of justification logic {\sf J5} and its extensions ${\sf JD5}$, ${\sf
JT5}$, ${\sf J45}$, ${\sf JB5}$, ${\sf JD45}$, ${\sf JB45}$, ${\sf JT45}$, ${\sf JTB5}$, ${\sf JDB5}$, ${\sf JTB45}$, ${\sf JDB45}$  contains in addition the unary operation $?$ on terms, and these logics contain the following axiom scheme:

\begin{description}
\item[j5.] \textit{Negative introspection axiom}, $\neg t:A\r
?t:\neg t:A$.
 \end{description}
The language of justification logic {\sf JB} and its extensions ${\sf JDB}$, ${\sf
JTB}$, ${\sf JB4}$, ${\sf JB5}$, ${\sf JDB4}$, ${\sf JTB4}$, ${\sf JTB5}$, ${\sf JDB5}$, ${\sf JB45}$, ${\sf JTB45}$, ${\sf JDB45}$  contains in addition the unary operation $\bar{?}$ on terms, and these logics contain the following axiom scheme:
\begin{description}
\item[jB.] \textit{Weak negative introspection axiom}, $\neg A\r
\bar{?} t:\neg t: A$.
 \end{description}
 All justification logics have the inference rule Modus Ponens. Moreover, if {\bf j4} is not an axiom of the justification logic it has the \textit{Iterated Axiom Necessitation}  rule:
\begin{description}
\item[IAN.]
  $\vdash c_{i_n}:c_{i_{n-1}}:\ldots:c_{i_1}:A$, where $A$ is an axiom instance of the logic, $c_{i_j}$'s
are arbitrary justification constants and $n\geq 1$.
\end{description}
If {\bf j4} is an axiom of the justification logic it has the \textit{Axiom Necessitation}  rule:
\begin{description}
\item [AN.] $\vdash c:A$, where $A$ is
an axiom instance of the logic and $c$ is an arbitrary justification constant.
\end{description}
 \end{definition}
As it is clear from the above definition, the name of each justification logic is indicated by the list of its axioms. For example, ${\sf JT45}$ is the extension of ${\sf J}$ by axioms ${\bf jT}$, ${\bf j4}$, and ${\bf j5}$ (and moreover, it contains the rule {\bf AN}). An exception is the logic of proofs \LP~which can be named as {\sf JT4}.
\begin{remark}\label{Remark: Remo-Roman JB}
Goetschi and Kuznets in \cite{GoetschiKuznets2012}
considered $A\r \bar{?} t:\neg t:\neg A$ (which was called axiom {\bf jb}
there) instead of axiom {\bf jB} in justification logic {\sf JB} and its extensions. Moreover, they assign levels to justification constants, and consider the following iterated axiom necessitation rule (instead of {\bf IAN}) for all justification logics:
\[\di{\frac{A~\textrm{is an axiom instance}}{c^n_{i_n}:c^{n-1}_{i_{n-1}}:\ldots:c^1_{i_1}:A}~{\bf iAN}}\]
where $c^j_i$ is a constant of level $j$.
Let us denote Goetschi and Kuznets' justification logics containing axiom {\bf jb}
in the above mentioned formalization by ${\sf JB'}$, ${\sf JDB'}$, ${\sf
JTB'}$, ${\sf JB4'}$, ${\sf JB5'}$, ${\sf JDB4'}$, ${\sf JTB4'}$, ${\sf JTB5'}$, ${\sf JDB5'}$, ${\sf JB45'}$, ${\sf JTB45'}$, ${\sf JDB45'}$. Note that axiom {\bf jb} is provable in {\sf JB} and its extensions.
\end{remark}

In what follows, {\sf JL} (possibly with subscript) denotes any of the justification logics
defined in Definition \ref{def: justification logics}, unless stated otherwise.\footnote{Here by a logic (particularly a justification logic) we mean the set of its theorems.} $Tm_\JL$ and $Fm_\JL$ denote the set of all terms and all formulas of \JL, respectively. Note that for justification logics $\JL_1$ and $\JL_2$, if $Tm_{\JL_1} \subseteq Tm_{\JL_2}$ then $Fm_{\JL_1} \subseteq Fm_{\JL_2}$.

We now proceed to the definition of \textit{Constant
Specifications}.
\begin{definition}
Let \JL~be a justification logic which does not contain axiom {\bf
j4} in its axiomatization.
\begin{enumerate}
\item A \textit{constant specification} $\CS$
for \JL~is a set of formulas of the form
$c_{i_n}:c_{i_{n-1}}:\ldots:c_{i_1}:A$, where $n\geq 1$, $c_{i_j}$'s are
justification constants and $A$ is an axiom instance of \JL, such that it is downward closed: if $c_{i_n}:c_{i_{n-1}}:\ldots:c_{i_1}:A\in\CS$, then $c_{i_{n-1}}:\ldots:c_{i_1}:A\in\CS$.
\item A constant specification $\CS$ is \textit{axiomatically appropriate} if for each axiom instance $A$ of \JL~
there is a constant $c$ such that $c:A\in\CS$, and if $F\in\CS$ then $c:F\in\CS$ for some constant $c$.
\item A constant specification $\CS$ is
\textit{schematic} if whenever $c_{i_n}:c_{i_{n-1}}:\ldots:c_{i_1}:A\in\CS$ for some axiom instance $A$, then for every instance $B$ of the same axiom scheme $c_{i_n}:c_{i_{n-1}}:\ldots:c_{i_1}:B\in\CS$.
\end{enumerate}
\end{definition}

\begin{definition}\label{def: CS for extensions of J4}
Let \JL~be a justification logic which contain axiom {\bf
j4} in its axiomatization.
\begin{enumerate}
\item A \textit{constant specification} $\CS$ for \JL~is a set of formulas
of the form $c:A$, such that $c$ is a justification constant and
$A$ is an axiom instance of \JL.
\item  A constant specification $\CS$ is
\textit{axiomatically appropriate} if for each axiom instance $A$ there is
a constant $c$ such that $c:A\in\CS$.
\item A constant specification $\CS$ is
\textit{schematic} if whenever $c:A\in\CS$ for some axiom instance $A$, then for every instance $B$ of the same axiom scheme $c:B\in\CS$.
\end{enumerate}
\end{definition}

The typical form of a formula in a constant specification for \JL~is $c:F$, where $c$ is a justification constant, and $F$ is an axiom instance of \JL~or $F=c_{i_m}:c_{i_{m-1}}:\ldots:c_{i_1}:A$, where $m\geq 1$, $c_{i_j}$'s are justification constants and $A$ is an axiom instance of \JL.

\begin{definition}
Let $\CS$ be a constant specification. If $c_{i_n}:c_{i_{n-1}}:\ldots:c_{i_1}:A\in\CS$ ($n\geq 1$), then the sequent  $c_{i_n},c_{i_{n-1}},\ldots,c_{i_1}$ of justification constants are called \CS-associated with the axiom instance $A$. A constant is \CS-associated if it occurs in  a sequence of constants which is \CS-associated with some axiom instance. A constant specification $\CS$ is injective if each justification constant occurs at most once in the formulas of $\CS$ as a \CS-associated constant.
\end{definition}

Let ${\sf JL}_\CS$ (or ${\sf JL}(\CS)$) be the fragment of ${\sf JL}$ where the (Iterated) Axiom Necessitation rule only produces formulas from the given $\CS$. In the rest of the paper whenever we use $\JL_\CS$ it is assumed that $\CS$ is a constant specification for $\JL$. By
${\sf JL}_\CS,S\vdash F$ we mean that formula $F$ is derivable in ${\sf JL}_\CS$ from the set of formulas $S$. Every proof in ${\sf JL}$ generates a finite (injective) constant specification $\CS$, which contains those
formulas which are introduced by {\bf IAN}/{\bf AN}. The \textit{total constant specification for \JL}, denoted ${\sf TCS}_\JL$, is the largest constant specification  which is defined as follows:
\begin{itemize}
\item if \JL~contains axiom {\bf j4}:
$${\sf TCS}_\JL=\{ c:A~|~c~\textrm{is a justification constant,~$A$~is an axiom instance of}~ \JL\},$$
\item if \JL~does not contain axiom {\bf j4}:
$${\sf TCS}_\JL=\{ c_{i_n}:c_{i_{n-1}}:\ldots:c_{i_1}:A~|~c_{i_n}, c_{i_{n-1}}, \ldots, c_{i_1}~\textrm{are justification constants},$$ $$\textrm{~$A$~is an axiom instance of}~ \JL\},$$
\end{itemize}
By $\JL$ we mean $\JL_{{\sf TCS}_\JL}$. It is immediate that ${\sf TCS}_\JL$ is axiomatically appropriate and schematic, but it is not injective.

\begin{remark}
In some literature on justification logics (e.g. \cite{Kuznets2008,KuznetsStuder2012}) constant specifications for all justification logics  are defined as in Definition \ref{def: CS for extensions of J4}. In this case all justification logics contain the term operation $!$, and the Axiom Necessitation Rule for $\JL_\CS$ is formulated as follows:
\[\di{\frac{c:A\in\CS}{!\cdots !c:\cdots:!!c:!c:A}~{\bf AN!}}\]
All the results obtained in this formulation can be obtained in our formulation with small modifications in the proofs.
\end{remark}

\begin{theorem}[Deduction Theorem]
 $\JL_\CS, S, A\vdash B$ if and only if $\JL_\CS, S\vdash A\r B$.
\end{theorem}

For a given justification formula $F$, we write $F[t/x]$ for the
result of simultaneously replacing all occurrences of variable $x$
in $F$ by term $t$. Substitution lemma holds in all justification
logics \JL.

\begin{lemma}[Substitution Lemma]\label{lemma: substitution lemma JL}
\begin{enumerate}
\item  Given a schematic constant specification $\CS$ for \JL, if $$\JL_\CS,S\vdash F,$$ then for every justification variable $x$ and justification term $t$, we
have $$\JL_\CS,S[t/x]\vdash F[t/x].$$
\item Given an axiomatically appropriate constant specification $\CS$ for \JL, if $$\JL_\CS,S\vdash F,$$ then for every justification variable $x$ and justification term $t$, we
have $$\JL_\CS,S[t/x]\vdash \bar{F}[t/x],$$ where $\bar{F}$ is obtained from
formula $F$ by (possibly) replacing some justification constants with other constants.
\end{enumerate}
\end{lemma}

 The following lemma is standard in justification logics.
\begin{lemma}[Internalization Lemma]\label{lemma:Internalization Lemma}
Given axiomatically appropriate constant specification $\CS$ for \JL, if $${\sf JL}_\CS,A_1,\ldots,A_n \vdash F,$$ then there is a justification term
$t(x_1,\ldots,x_n)$, for variables $x_1,\ldots,x_n$, such that $${\sf JL}_\CS, x_1:A_1,\ldots,x_n:A_n \vdash t(x_1,\ldots,x_n):F.$$ In
particular, if ${\sf JL}_\CS \vdash F$, then there is a ground
justification term $t$ such that ${\sf
JL}_\CS \vdash t:F$.
\end{lemma}

\subsection{Correspondence theorem}
In order to show the correspondence between justification logics (particularly {\sf JB} and its extensions) and their modal counterparts, it is helpful to recall the definition of embedding. There are various kinds of embedding of justification logics defined in the literature (cf. \cite{Fitting2007,Fitting2008,Goetschi2012,GoetschiKuznets2012}).
We first recall the definition of embedding of Goetschi and Kuznet \cite{Goetschi2012,GoetschiKuznets2012} which is more general than the others.
\begin{definition}\label{def: Roman-Remo embed, equivalent}
Let $\JL_1$ and $\JL_2$ be two justification logics over
languages ${\mathcal L}_1$ and ${\mathcal L}_2$ respectively.
\begin{enumerate}
\item An operation translation $\omega$ from ${\mathcal L}_1$ to ${\mathcal L}_2$ is a total function that for each $n \geq 0$, maps every $n$-ary operation $*$ of ${\mathcal L}_1$ to an ${\mathcal L}_2$-term $\omega(*)=\omega_*(x_1,\ldots,x_n)$, which do not contain variables other than $x_1,\ldots,x_n$.
\item Justification logic $\JL_1$ embeds in $\JL_2$, denoted by
$\JL_1\tilde{\subseteq} \JL_2$, if there is an operation translation $\omega$ from ${\mathcal L}_1$ to ${\mathcal L}_2$ such that
$\JL_1\vdash F$ implies $\JL_2\vdash F\omega$ for any ${\mathcal L}_1$-formula $F$, where $F\omega$ results from $F$ by replacing each
$n$-ary operation $*$ by the ${\mathcal L}_2$-term $\omega(*)$.
\item Two justification
logics $\JL_1$ and $\JL_2$ are equivalent, denoted by
$\JL_1\equiv\JL_2$, if each embeds in the other.
\end{enumerate}
 \end{definition}
Recall that our justification logics axiomatized using axiom schemes. Following \cite{Goetschi2012,GoetschiKuznets2012}, the formula representation of a scheme $X$ is defined as a formula of the form $A(x_1,\ldots,x_n,P_1,\ldots,P_m)$, with $n,m\geq 0$, in which all terms in $X$ replaced by variables $x_1,\ldots,x_n$ and all formulas in $X$ replaced by propositional variables $P_1,\ldots,P_m$. The following theorem from \cite{Goetschi2012,GoetschiKuznets2012} is a useful tool for showing the embedding of justification logics.
\begin{theorem}[Embedding $\tilde{\subseteq}$]\label{thm: Roman-Remo JL embedding}
Let $\JL_1$ and $\JL_2$ be two justification logics as defined in Definition
\ref{def: justification logics} or Remark \ref{Remark: Remo-Roman JB} over
languages ${\mathcal L}_1$ and ${\mathcal L}_2$ respectively. Let ${\mathcal L}_1$ and $\JL_1$ have the following conditions:
\begin{enumerate}
\item the set of constants of ${\mathcal L}_1$ is divided into levels,
\item {\bf MP} and {\bf iAN} are the only rules of $\JL_1$,
\item $\JL_1$ is axiomatized using axiom schemes,
\item the formula representations of the axioms of $\JL_1$ do not contain constants,
\end{enumerate}
and ${\mathcal L}_2$ and $\JL_2$ have the following conditions:
\begin{enumerate}
\item $\JL_2$ satisfies the substitution closure, i.e. if $\JL_2\vdash F$, then $\JL_2\vdash F[t/x]$ for every variable $x$ and term $t\in Tm_{\JL_2}$.
\item $\JL_2$ satisfies the internalization property, i.e. if ${\sf JL}_2 \vdash F$, then there is a ground
justification term $t\in Tm_{\JL_2}$ such that $\JL_2 \vdash t:F$.
\end{enumerate}
If there exists an operation translation $\omega$ from ${\mathcal L}_1$ to ${\mathcal L}_2$ such that for every axiom $X$  of $\JL_1$ with formula representation $A(x_1,\ldots,x_n,P_1,\ldots,P_m)$, the ${\mathcal L}_2$-formula $A(x_1,\ldots,x_n,P_1,\ldots,P_m)\omega$ is provable in $\JL_2$, then $\JL_1\tilde{\subseteq} \JL_2$.
\end{theorem}
\begin{example}\label{example: Roman-Remo JB equivalent JB'}
Let $\JL_1$ be one of the following justification logics
\[\textrm{${\sf JB'}$, ${\sf JDB'}$, ${\sf
JTB'}$, ${\sf JB4'}$, ${\sf JB5'}$, ${\sf JDB4'}$, ${\sf JTB4'}$, ${\sf JTB5'}$, ${\sf JDB5'}$, ${\sf JB45'}$, ${\sf JTB45'}$, ${\sf JDB45'}$}\]
 from Remark \ref{Remark: Remo-Roman JB}, and $\JL_2$ be the corresponding justification logic
 \[\textrm{${\sf JB}$,${\sf JDB}$, ${\sf
JTB}$, ${\sf JB4}$, ${\sf JB5}$, ${\sf JDB4}$, ${\sf JTB4}$, ${\sf JTB5}$, ${\sf JDB5}$, ${\sf JB45}$, ${\sf JTB45}$, ${\sf JDB45}$}\]
from Definition \ref{def: justification logics}. Note that $\JL_1$ and $\JL_2$ satisfy all the conditions of Theorem \ref{thm: JL embedding}, and moreover ${\sf JL}_2\vdash P\r \bar{?}x:\neg x:\neg P$. Consider the identity operation translation $\omega^{id}$ which is defined as $\omega^{id}(*) := *(x_1,\ldots,x_n)$ for each $n$-ary operation $*$. Now it is easy to show that for every axiom $X$  of $\JL_1$ with formula representation $A(x_1,\ldots,x_n,P_1,\ldots,P_m)$, the formula $A(x_1,\ldots,x_n,P_1,\ldots,P_m)\omega^{id}$ is provable in $\JL_2$, and hence $\JL_1\tilde{\subseteq} \JL_2$.
\end{example}
In the following we will state the connection between
justification logics and modal logics.
\begin{definition}
For a justification formula $F$, the forgetful projection of $F$,
denoted by $F^\circ$, is a modal formula defined inductively as
follows: $P^\circ :=P$ ($P$ is a propositional variable),
$\bot^\circ := \bot$, $(\neg A)^\circ := \neg A^\circ$, $(A \star
B)^\circ :=A^\circ \star B^\circ$ (where $\star$ is a propositional
connective), and $(t:A)^\circ := \Box A^\circ$. For a set $S$ of
justification formulas let $S^\circ=\{ F^\circ~|~F\in S\}$.
\end{definition}
\begin{definition}
A \JL-realization $r$ of a modal formula $F$ is a formula $F^r$ in
the language of {\sf JL} that is obtained by replacing all
occurrences of $\b$ with justification terms from $Tm_\JL$. A realization $r$ of
$F$ is called \textit{normal} if all negative occurrences of $\b$
in $F$ are replaced by distinct justification variables.
\end{definition}
The realization theorem for modal logic \ML~and justification logic \JL~states that theorems of \ML~can be realized into the theorems of \JL; in other words $\ML\subseteq \JL^\circ$.

It is proved in \cite{Goetschi2012,GoetschiKuznets2012} that if $\JL_1\tilde{\subseteq} \JL_2$ then $\JL_1^\circ\subseteq \JL_2^\circ$. Thus, if $\JL_1\tilde{\subseteq} \JL_2$ and \ML~is a modal logic such that $\ML\subseteq \JL_1^\circ$, then $\ML\subseteq \JL_2^\circ$. This shows that how embedding of justification logics can be used to prove the realization theorem.

As axiomatic formulation of justification logics suggests, modal
logics {\sf ML} are forgetful projections of their corresponding
justification logics {\sf JL}.
\begin{theorem} [Correspondence Theorem] \label{thm: Realization Theorem}
The following correspondences hold:
\begin{equation}\label{number equation: realization 1}
\begin{array}{ccccc}
  {\sf J}^\circ={\sf K}, & {\sf JT}^\circ={\sf T}, &  {\sf JD}^\circ={\sf D}, & {\sf J4}^\circ={\sf K4}, & {\sf J5}^\circ={\sf K5},\\
    ~~{\sf JD5}^\circ={\sf D5}, & ~~{\sf JD4}^\circ={\sf D4},&{\sf J45}^\circ={\sf K45},  &~~{\sf JD45}^\circ={\sf D45}, &{\sf LP}^\circ={\sf S4}, \\
{\sf JT5}^\circ={\sf JT45}^\circ={\sf S5}.
\end{array}
 \end{equation}
 \begin{equation}\label{number equation: realization 2}
\begin{array}{cccc}
{\sf JB}^\circ={\sf KB}, &~~ {\sf JTB}^\circ={\sf TB}, &{\sf JDB}^\circ={\sf DB},& {\sf
JB4}^\circ={\sf
JB5}^\circ={\sf JB45}^\circ={\sf KB5}, \\
\multicolumn{4}{l}{{\sf JTB5}^\circ={\sf JDB5}^\circ={\sf JTB45}^\circ={\sf JDB45}^\circ={\sf JTB4}^\circ={\sf JDB4}^\circ={\sf
 S5}.}
\end{array}
 \end{equation}
 Moreover, all the realization results of (1) and (2) are normal realizations.
\end{theorem}
\begin{proof}  The proof of the correspondences in (\ref{number
equation: realization 1}) can be found in \cite{A2001,A2008,Brezhnev2000,KaiRemoRoman2010,GoetschiKuznets2012,Rubtsova2006}. For the correspondences in (\ref{number equation: realization 2}), we use Example
\ref{example: Roman-Remo JB equivalent JB'} and the following correspondences from
\cite{GoetschiKuznets2012}:
 \begin{equation}\label{number equation: realization 3}
\begin{array}{cccc}
({\sf JB'})^\circ={\sf KB}, &~~ ({\sf JTB'})^\circ={\sf TB}, &({\sf JDB'})^\circ={\sf DB},& ({\sf JB4'})^\circ=({\sf JB5'})^\circ=({\sf JB45'})^\circ={\sf KB5}, \\
\multicolumn{4}{l}{({\sf JTB5'})^\circ=({\sf JDB5'})^\circ=({\sf JTB45'})^\circ=({\sf JDB45'})^\circ=({\sf JTB4'})^\circ=({\sf JDB4'})^\circ={\sf S5}.}
\end{array}
 \end{equation}
 We only prove that ${\sf JB}^\circ={\sf
KB}$, the proof of the remaining cases is similar.

If ${\sf JB}\vdash F$, then by induction on the proof of $F$ it is
easy to show that ${\sf KB}\vdash F^\circ$. Hence ${\sf JB}^\circ\subseteq{\sf
KB}$. For the opposite
direction, that is the realization theorem, suppose that ${\sf KB}\vdash F$.  By Example \ref{example: Roman-Remo JB equivalent JB'}, since ${\sf JB'}\tilde{\subseteq}{\sf JB}$, we have $({\sf JB'})^\circ\subseteq{\sf JB}^\circ$. Then, by (3) we have $({\sf JB'})^\circ={\sf
KB}$, and thus ${\sf KB}\subseteq{\sf JB}^\circ$.\qe \end{proof}

Finally, we recall the definition of Fitting's embeddings and show how it can be used to prove the realization theorem.

\begin{definition}\label{def: embed}
Let $\JL_1(\CS_1)$ and $\JL_2(\CS_2)$ be two justification logics such that $Tm_{\JL_1} \subseteq Tm_{\JL_2}$.
Justification logic $\JL_1(\CS_1)$ embeds in $\JL_2(\CS_2)$, denoted by $\JL_1(\CS_1)\hookrightarrow \JL_2(\CS_2)$, if there is a mapping $m$ from constants of $\JL_1$ to ground terms of $\JL_2$ such that  $\JL_1(\CS_1)\vdash F$ implies $\JL_2(\CS_2)\vdash m(F)$ for any formula $F\in Fm_{\JL_1}$, where $m(F)$ results from $F$ by replacing each constant $c$ by the justification term $m(c)$.
 \end{definition}
 \begin{definition}\label{def: equivalent}
Two justification logics $\JL_1(\CS_1)$ and $\JL_2(\CS_2)$, in the same language, are equivalent, denoted by
$\JL_1(\CS_1)\rightleftharpoons\JL_2(\CS_2)$, if $\JL_1(\CS_1)\hookrightarrow \JL_2(\CS_2)$ and $\JL_2(\CS_2)\hookrightarrow \JL_1(\CS_1)$.
 \end{definition}
 Since justification constants are 0-ary operations, operation translations in  Definition \ref{def: Roman-Remo embed, equivalent} are extensions of mappings considered by Fitting. Therefore, if $\JL_1\tilde{\subseteq} \JL_2$ then $\JL_1\hookrightarrow\JL_2$.

The following theorem is a generalization of Fitting's result (Theorem 9.2 in \cite{Fitting2008}) on the embedding of justification logics.
\begin{theorem}[Embedding $\hookrightarrow$]\label{thm: JL embedding}
Let $\JL_1(\CS_1)$ and $\JL_2(\CS_2)$ be two justification logics as defined in Definition \ref{def: justification logics} or Remark \ref{Remark: Remo-Roman JB} such that $Tm_{\JL_1} \subseteq Tm_{\JL_2}$. Assume
\begin{enumerate}
\item $\CS_1$ is a finite constant specification for $\JL_1$.
    \item $\JL_1(\CS_1)$ is axiomatized using axiom schemes.
    \item $\JL_2(\CS_2)$ satisfies the substitution closure, i.e. if $\JL_2(\CS_2)\vdash F$, then $\JL_2(\CS_2)\vdash F[t/x]$ for every variable $x$ and term $t\in Tm_{\JL_2}$.
    \item $\JL_2(\CS_2)$ satisfies the internalization property,  i.e. if ${\sf JL}_2(\CS_2) \vdash F$, then there is a ground
justification term $t\in Tm_{\JL_2}$ such that $\JL_2 (\CS_2)\vdash t:F$.
        \item Every axiom instance of $\JL_1(\CS_1)$ is a theorem of $\JL_2(\CS_2)$.
\end{enumerate}
Then $\JL_1(\CS)\hookrightarrow\JL_2(\CS_2)$.
\end{theorem}
\begin{proof} Suppose $\JL_1$ and $\JL_2$ are two justification logics such that $Tm_{\JL_1} \subseteq Tm_{\JL_2}$. If $\CS_1$ is empty, then let $m$ be the identity function. Obviously $\JL_1(\emptyset)\hookrightarrow\JL_2(\CS_2)$. Now suppose $\CS_1$ is not empty. In the rest of the proof it is
helpful to consider a fix list of all terms of $Tm_{\JL_1}$ and $Tm_{\JL_2}$, e.g. $Tm_{\JL_1}=\{s^1_1,s^1_2,\ldots\}$ and $Tm_{\JL_2}=\{s^2_1,s^2_2,\ldots\}$. In the rest of the proof we write $A(d_1,\ldots,d_m)$ to denote that $d_1,\ldots,d_m$ are all the constants occurring in the $\JL_1$-formula $A$ in the order of the list $Tm_{\JL_1}$. We detail the proof for justification logics defined in Definition \ref{def: justification logics}. The case of justification logics defined in Remark \ref{Remark: Remo-Roman JB} is similar. We first define mapping $m$ from constants to ground terms of $Tm_{\JL_2}$.

If $c$ is not a $\CS_1$-associated constant, then define $m(c)=c$. Next suppose $c$ is a $\CS_1$-associated constant.

Suppose $\JL_1$ contains the Axiom Necessitation rule {\bf AN}. We distinguish two cases:
\begin{description}
\item[$(a)$] If constant $c$ occurs only once as a $\CS_1$-associated constant, then the proof is similar
to the proof of Theorem 9.2 in \cite{Fitting2008}, however we repeat it here for convenience. Assume $c:A(d_1,\ldots,d_m)\in\CS_1$. We shall define the mapping $m$ on $c$.  Let $x_1,x_2, \ldots , x_m$ be distinct justification variables not occurring in $A$ (in the order of the list $Tm_{\JL_1}$). Since $A$ is an axiom instance of $\JL_1(\CS_1)$, which is axiomatized using axiom schemes, $A(x_1, x_2, \ldots , x_m)$ will also be an axiom instance. Then by hypothesis 5, $A(x_1, x_2, \ldots ,x_m)$ is a theorem of $\JL_2(\CS_2)$. Since $\JL_2(\CS_2)$ satisfies the internalization property, there is a ground justification term $t\in Tm_{\JL_2}$ such that $\JL_2(\CS_2)\vdash t:A(x_1, x_2, \ldots , x_m)$ (if there is more than one such term, we choose the first in the order of the list $Tm_{\JL_2}$). Now, let $m(c)=t$.
\item[$(b)$] If constant $c$ occurs more than once as a $\CS_1$-associated constant, then suppose
$$c:A_1(d_{1_1},\ldots,d_{1_{m_1}}), \ldots,c:A_k(d_{k_1},\ldots,d_{k_{m_k}}),$$
for some $k\geq 2$, are all formulas of $\CS_1$ with $c$ as a $\CS_1$-associated constant. Similar to the case $(a)$, we find ground justification terms $t_1,\ldots,t_k\in Tm_{\JL_2}$
such that $$\JL_2(\CS_2)\vdash t_j:A_j(x_1^j, x_2^j, \ldots , x_{m_j}^j),$$ for $j=1,2,\ldots,k$, where $x_1^j, x_2^j, \ldots , x_{m_j}^j$ are distinct justification variables not occurring in $A_j$. Letting $t=t_1 + \ldots + t_k$, by axiom {\bf Sum} and {\bf MP}, we have $\JL_2(\CS_2)\vdash t:A_j(x_{1}^j, x_{2}^j, \ldots , x_{m_j}^j)$, for $j=1,2,\ldots,k$. Now, let $m(c)=t$.
\end{description}
Suppose $\JL_1$ contains the Iterated Axiom Necessitation rule {\bf IAN}. Similar to the previous argument, We distinguish two cases:
\begin{description}
\item[$(a')$] If constant $c$ occurs only once as a $\CS_1$-associated constant, then suppose $c:c_{i_n}:c_{i_{n-1}}:\ldots:c_{i_1}:A(d_1,\ldots,d_m)\in\CS_1$. We shall define the mapping $m$ on $c$. Similar to case $(a)$, $A(x_1, x_2, \ldots , x_m)$ is a theorem of $\JL_2(\CS_2)$, where $x_1,
x_2, \ldots , x_m$ are distinct justification variables not occurring in A (again
in the order of the list $Tm_{\JL_1}$). Since $\JL_2(\CS_2)$ satisfies the
internalization property, by applying the internalization lemma $i_n+1$ times, we obtain
ground justification terms $t,t_{i_1},t_{i_2},\ldots,t_{i_n}\in Tm_{\JL_2}$ such that $\JL_2(\CS_2)\vdash t:t_{i_n}:\ldots:t_{i_1}:A(x_1, x_2, \ldots , x_m)$. Now, let
$m(c)=t$.
\item[$(b')$] If constant $c$ occurs more than once as a $\CS_1$-associated constant in sequences of constants, then suppose
  $$c:c_{1_{n_1}}:\ldots:c_{1_1}:A_1(d_{1_1},\ldots,d_{1_{m_1}}), \ldots,c:c_{k_{n_k}}:\ldots:c_{k_1}:A_k(d_{k_1},\ldots,d_{k_{m_k}}),$$ for some $k\geq 2$, are all formulas of $\CS_1$ where $c$ occurs  in the sequences of constants as a $\CS_1$-associated constant (since $\CS_1$ is downward closed, we can consider only those formulas of $\CS_1$ which begins with $c$). Similar to the case $(a')$, we find ground justification terms $t_j,t^j_1,\ldots,t^j_{n_j}\in Tm_{\JL_2}$, for $j=1,2,\ldots,k$,
such that
$$\JL_2(\CS_2)\vdash t_j:t^j_{n_j}:\ldots:t^j_{1}:A_j(x_1^j, x_2^j, \ldots , x_{m_j}^j),$$
 where $x_1^j, x_2^j, \ldots , x_{m_j}^j$ are distinct justification variables not occurring in $A_j$. Letting $t=t_1 + \ldots + t_k$, by axiom {\bf Sum} and {\bf MP}, we have $\JL_2(\CS_2)\vdash t:A_j(x_1^j, x_2^j, \ldots , x_{m_j}^j)$, for $j=1,2,\ldots,k$. Now, let $m(c)=t$.
\end{description}
 Now suppose $F$ is any theorem of $\JL_1(\CS_1)$. By
induction on the proof of $F$, we show that $m(F)$ is a theorem of
$\JL_2(\CS_2)$.

If $F$ is an axiom instance of $\JL_1(\CS_1)$, since $\JL_1(\CS_1)$ is axiomatized by schemes, $m(F)$ will also be an axiom instance of $\JL_1$, and hence a theorem of $\JL_2(\CS_2)$ by hypothesis 5.

If $F$ follows from $G$ and
$G\r F$ by Modus Ponens, then by the induction hypothesis $m(G)$
and $m(G\r F)=m(G)\r m(F)$ are theorems of $\JL_2(\CS_2)$. Then by Modus
Ponens $m(F)$ is a theorem of $\JL_2(\CS_2)$.

Finally suppose
$F=c_{i_n}:c_{i_{n-1}}:\ldots:c_{i_1}:A(d_1,\ldots,d_m)$ is obtained by {\bf IAN} (when $i_n \geq 1$) or by {\bf AN} (when $i_n=1$), i.e., $c_{i_n}:c_{i_{n-1}}:\ldots:c_{i_1}:A\in\CS_1$, where
$A=A(d_1,\ldots,d_m)$ is a $\JL_1(\CS_1)$ axiom. If $x_1, x_2,\ldots ,
x_m$ are justification variables not occurring in $A$, as above, there are
ground terms $t_{i_1},t_{i_2},\ldots,t_{i_n}\in Tm_{\JL_2}$ such that $$t_{i_1}=m(c_{i_1}),t_{i_2}=m(c_{i_2}),\ldots,t_{i_n}=m(c_{i_n})$$
 and
  $$\JL_2(\CS_2)\vdash
t_{i_n}:\ldots:t_{i_2}:t_{i_1}:A(x_1, x_2, \ldots , x_m).$$ Since $\JL_2(\CS_2)$
satisfies the substitution closure, we obtain
$$\JL_2(\CS_2)\vdash
t_{i_n}:\ldots:t_{i_2}:t_{i_1}:A(m(d_1), m(d_2), \ldots ,m(d_m)),$$
but we have
\[
t_{i_n}:\ldots:t_{i_1}:A(m(d_1), \ldots
,m(d_m))=m(c_{i_n}):\ldots:m(c_1):m(A)=m(c_{i_n}:\ldots:c_1:A)=m(F).
\]
Thus $\JL_2(\CS_2)\vdash m(F)$. \qe \end{proof}
\begin{theorem}\label{thm:embedding injective}
Suppose $\JL_1(\CS_1)$ and $\JL_2(\CS_2)$ satisfy all the conditions of Theorem \ref{thm: JL embedding} except that $\CS_1$ is not required to be finite, instead it  is an injective constant specification for $\JL_1$. Then $\JL_1(\CS_1)\hookrightarrow \JL_2(\CS_2)$
\end{theorem}
\begin{proof} Proceed just as in the proof of Theorem \ref{thm: JL embedding}, except that now cases $(b)$ and $(b')$ cannot happen.\qe \end{proof}

The above theorems give criterions to showing $\JL_1(\CS_1)\hookrightarrow \JL_2(\CS_2)$, where $\CS_1$ is finite or injective and $\CS_2$ is the total constant specification ${\sf TCS}_{\JL_2}$ or an arbitrary schematic constant specification. Note that all justification logics \JL~of
Definition \ref{def: justification logics} and Remark \ref{Remark: Remo-Roman JB} satisfy hypothesis 2 of Theorem \ref{thm: JL embedding}. Thus we have
\begin{corollary}
Let $\JL_1$ and $\JL_2$ be two justification logics as defined in Definition
\ref{def: justification logics} or Remark \ref{Remark: Remo-Roman JB} such that $Tm_{\JL_1} \subseteq Tm_{\JL_2}$, $\CS_1$ be a finite (or an injective) constant specification for $\JL_1$, and $\CS_2$ be an axiomatically appropriate schematic constant specification for $\JL_2$, and every axiom of $\JL_1(\CS_1)$ is a theorem of $\JL_2(\CS_2)$. Then
$\JL_1(\CS_1)\hookrightarrow\JL_2(\CS_2)$.
\end{corollary}
Using the above corollary we give some examples here.
\begin{example}
Define the justification logic ${\sf JT45'}$ in the language of ${\sf JT45}$ by replacing the axiom scheme {\bf j5} by $!t:t:A~\vee~?t:\neg
t:A$. Let $\CS_1$ and $\CS'_1$ be finite (or injective) constant specifications for ${\sf JT45}$ and ${\sf JT45'}$, respectively, and $\CS_2$ and $\CS'_2$ be axiomatically appropriate schematic constant specifications for ${\sf JT45}$ and ${\sf JT45'}$, respectively. It is not difficult to show that the substitution closure and internalization property hold for ${\sf JT45'}(\CS_2')$, and hence ${\sf JT45}(\CS_1)\hookrightarrow {\sf JT45'}(\CS_2')$ and ${\sf JT45'}(\CS'_1)\hookrightarrow {\sf JT45}(\CS_2)$.
\end{example}
\begin{example}\label{example: JB equivalent JB'}
Let $\JL_1$ be one of the following justification logics
\[\textrm{${\sf JB'}$, ${\sf JDB'}$, ${\sf
JTB'}$, ${\sf JB4'}$, ${\sf JB5'}$, ${\sf JDB4'}$, ${\sf JTB4'}$, ${\sf JTB5'}$, ${\sf JDB5'}$, ${\sf JB45'}$, ${\sf JTB45'}$, ${\sf JDB45'}$}\]
 from Remark \ref{Remark: Remo-Roman JB}, and $\JL_2$ be the corresponding justification logic
 \[\textrm{${\sf JB}$,${\sf JDB}$, ${\sf
JTB}$, ${\sf JB4}$, ${\sf JB5}$, ${\sf JDB4}$, ${\sf JTB4}$, ${\sf JTB5}$, ${\sf JDB5}$, ${\sf JB45}$, ${\sf JTB45}$, ${\sf JDB45}$}\]
from Definition \ref{def: justification logics}.
If $\CS$ is a finite (or an injective) constant specifications for $\JL_1$, and $\CS_2$ be an axiomatically appropriate schematic constant specification for $\JL_2$, then it is
easy to prove that ${\sf JL}_1(\CS)\hookrightarrow{\sf JL}_2(\CS_2)$.
\end{example}
The following lemma shows how Fitting's embedding can be used to reduce the realization theorem from one justification logic to another.
\begin{lemma}\label{lem:embedding realization}
Let $\JL_1(\CS_1)$ and $\JL_2(\CS_2)$ be two justification logics such that $Tm_{\JL_1} \subseteq Tm_{\JL_2}$.
\begin{enumerate}
\item if $\JL_1(\CS_1)\hookrightarrow \JL_2(\CS_2)$ then $\JL_1(\CS_1)^\circ\subseteq\JL_2(\CS_2)^\circ$.
\item If $\JL_1(\CS_1)\hookrightarrow \JL_2(\CS_2)$ and $\ML\subseteq \JL_1(\CS_1)^\circ$, then $\ML\subseteq \JL_2(\CS_2)^\circ$.
\end{enumerate}
\end{lemma}
\begin{proof} For (1), suppose $\JL_1(\CS_1)\hookrightarrow \JL_2(\CS_2)$ and $F^\circ\in\JL_1(\CS_1)^\circ$. Hence $\JL_1(\CS_1)\vdash F$. Thus, for the embedding $m$ witnesses $\JL_1(\CS_1)\hookrightarrow \JL_2(\CS_2)$, we have $\JL_2(\CS_2)\vdash m(F)$. By induction on the complexity of the formula $F$, it is easy to prove that $(m(F))^\circ=F^\circ$. Therefore, $F^\circ\in\JL_2(\CS_2)^\circ$. Clause 2 follows easily  from 1.\qe \end{proof}
\begin{remark}
In \cite{KaiRemoRoman2010} and \cite{Ghari2012-Thesis}, authors use operation $?$ instead of $\bar{?}$ in the language of {\sf JB} and its extensions, and axioms $A\r ? t:\neg t:\neg A$ and $\neg A\r ?t:\neg t: A$ are used instead of {\bf jb} and {\bf jB}, respectively (moreover in \cite{KaiRemoRoman2010} rule {\bf AN!} is used for all justification logics). For what follows, let ${\sf JB}_?$, ${\sf JTB}_?$, ${\sf JDB}_?$, ${\sf JB5}_?$, ${\sf JB45}_?$ denote justification logics with axiom $\neg A\r ?t:\neg t: A$ instead of {\bf jB}, and ${\sf JB}'_?$, ${\sf JTB}'_?$, ${\sf JDB}'_?$, ${\sf JB5}'_?$, ${\sf JB45}'_?$ denote justification logics with axiom $A\r ? t:\neg t:\neg A$ instead of {\bf jb}. Similar to Example \ref{example: JB equivalent JB'}, one can show that
\begin{equation}\label{number equation: embedding JB'? to LB?}
 {\sf JB}'_?\hookrightarrow{\sf JB}_?,{\sf JTB}'_?\hookrightarrow{\sf JTB}_?,{\sf JDB}'_?\hookrightarrow{\sf JDB}_?,{\sf JB5}'_?\hookrightarrow{\sf JB5}_?,{\sf JB45}'_?\hookrightarrow{\sf JB45}_?,
 \end{equation}
where for simplicity we omit the constant specifications in (\ref{number equation: embedding JB'? to LB?}).
The correspondence theorems
$$\textrm{${\sf JB}_?^\circ={\sf KB}$, ${\sf JTB}_?^\circ={\sf TB}$, ${\sf JDB}_?^\circ={\sf DB}$, ${\sf JB5}_?^\circ={\sf JB45}_?^\circ={\sf KB5}$}$$
 are proved in \cite{Ghari2012-Thesis} using Theorem \ref{thm:embedding injective}, Lemma \ref{lem:embedding realization}, embeddings in (\ref{number equation: embedding JB'? to LB?}), and the fact that the correspondences
 \[\textrm{$({\sf JB}_?')^\circ={\sf KB}$, $({\sf JTB}_?')^\circ={\sf TB}$, $({\sf JDB}_?')^\circ={\sf DB}$, $({\sf JB5}_?')^\circ=({\sf JB45}_?')^\circ={\sf KB5}$},\]
which are proved in \cite{KaiRemoRoman2010} uses injective constant specifications (this assumption is claimed in Remark 5.10 in \cite{KaiRemoRoman2010}).
\end{remark}
\section{Semantics of justification logics}
In this section, we recall the definitions of Fitting models for justification logics. Fitting models (or F-models for short) are
Kripke-style models, were first developed by Fitting in
\cite{Fitting2005} for \LP. In Section 4, we internalize these models within the syntax of sequent calculus to
produce labeled sequent calculi for justification logics. First we recall a definition from \cite{Kuznets2008}.

\begin{definition}\label{def:possible evidence function}
Let \JL~be a justification logic, $\CS$ be a constant specification for \JL.
\begin{enumerate}
\item A possible evidence function on a nonempty set $\W$ for $\JL_\CS$ is any function $\mathcal{A}:Tm_\JL \times Fm_\JL \r 2^\W$ such that if $c:F\in\CS$, then $\mathcal{A}(c,F)=\W$.
\item For two possible evidence functions $\mathcal{A}_1$, $\mathcal{A}_2$ on $\W$ for $\JL_\CS$, we say that $\mathcal{A}_2$ is based on $\mathcal{A}_1$, and write $\mathcal{A}_1\subseteq\mathcal{A}_2$, if $\mathcal{A}_1(t,F)\subseteq\mathcal{A}_2(t,F)$ for any $t\in Tm_\JL$ and any $F\in Fm_\JL$.
\item A Kripke frame is a pair $(\W,\RR)$, where $\W$ is a non-empty set (of possible worlds) and accessibility relation $\RR$ is a binary relation on $\W$.
\item A Fitting frame for $\JL_\CS$ is a triple $(\W,\RR,\mathcal{A})$, where $(\W,\RR)$ is a Kripke frame and $\mathcal{A}$ is a possible evidence function on $\W$ for $\JL_\CS$.
\item A possible Fitting model for $\JL_\CS$ is a quadruple $(\W,\RR,\mathcal{A},\V)$ where $(\W,\RR,\mathcal{A})$ is a Fitting frame for $\JL_\CS$ and $\V$ is a valuation, that is a function from the set of propositional variables to subsets of $\W$.
\end{enumerate}
\end{definition}

\begin{definition}\label{def:forcing relation}
For a possible Fitting model $\M=(\W, \RR, \mathcal{A}, \V)$ the forcing relation $\Vdash$ defined as follows:
\begin{enumerate}
\item $(\M,w)\Vdash P$ if{f} $w\in\V(P)$, for propositional variable $P$,
\item ~$\Vdash$ respects classical Boolean connectives,
 \item $(\M,w)\Vdash t:F$ if{f} $w\in\mathcal{A}(t,F)$ and for every $v\in \W$ with $w \RR v$, $(\M, v)\Vdash F$.
\end{enumerate}
\end{definition}
\begin{definition}\label{Kripke-Fitting models J}
 A Fitting model $\M=(\W, \RR, \E, \V)$ for justification
logic ${\sf J}_\CS$ is a possible Fitting model in which  $\E$ is a possible evidence function on $\W$ for ${\sf J}_\CS$, meeting the following conditions:
\begin{description}
 \item[$\E 1.$] Application: $\E(s,A\r B)\cap\E(t,A)\subseteq\E(s\cdot t,B)$,
  \item[$\E 2.$] Sum: $\E(s,A)\cup \E(t,A)\subseteq\E(s+t,A)$.
  \end{description}
$\E$ is called  \textit{admissible evidence function} for ${\sf J}_\CS$.
  \end{definition}

In order to define F-models for other justification logics of Definition \ref{def: justification logics} certain
additional conditions should be imposed on the accessibility
relation $\RR$ and possible evidence function $\E$.
\begin{definition}\label{Kripke-Fitting models JL}
 A Fitting model $\M=(\W, \RR, \E, \V)$ for justification logic ${\sf JL}_\CS$ is a Fitting model for ${\sf J}_\CS$, in which the admissible evidence function $\E$ is now a possible evidence function on $\W$ for $\JL_\CS$, and $\M$ satisfies the following conditions:
\begin{itemize}
\item if $\JL_\CS$ contains axiom {\bf jT}, then $\RR$ is reflexive.\vspace{0.05cm}
\item if $\JL_\CS$ contains axiom {\bf jD}, then $\RR$ is serial.\vspace{0.05cm}
\item if $\JL_\CS$ contains axiom {\bf j4}, then $\RR$ is transitive, and $\E$ satisfies:
\begin{description}
 \item[$\E 3.$] \textit{Monotonicity}: If $w\in\E(t,A)$ and $w\RR v$, then
 $v\in\E(t,A)$.
  \item[$\E 4.$] \textit{Positive introspection}:
   $\E(t,A)\subseteq\E(!t,t:A)$.\vspace{0.05cm}
     \end{description}
\item if $\JL_\CS$ contains axiom {\bf jB}, then $\RR$ is symmetric, and $\E$ satisfies:
\begin{description}
     \item[$\E 5.$]  \textit{Weak negative introspection}: If $(\M,w)\Vdash\neg A$, then $w\in\E(\bar{?}t,\neg t:A)$, for all terms $t$.\vspace{0.05cm}
    \end{description}
\item if $\JL_\CS$ contains axiom {\bf j5}, then $\E$ satisfies:
\begin{description}
       \item [$\E 6.$] \textit{Negative introspection}: $[\E(t,A)]^c\subseteq\E(?t,\neg t:A)$, where $S^c$ means the complement of the set $S$ relative to the
set of worlds $\W$.
      \item[$\E 7.$]  \textit{Strong evidence}: if $w\in\E(t,A)$, then $(\M,w)\Vdash
      t:A$.
    \end{description}
\end{itemize}
\end{definition}
We say that a formula $F$ is true in a model $\M$ (denoted by
$\M\Vdash F$), if $(\M,w)\Vdash F$ for all $w\in\W$. For a set $S$ of formulas, $\M\Vdash S$ provided that $\M\Vdash F$ for all formulas $F$ in $S$. Note that given a constant specification $\CS$ for \JL, and a model $\M$ of $\JL_\CS$ we have $\M\Vdash \CS$ (in this case it is said that $\M$ respects $\CS$). By
$\JL_\CS$-model we mean Fitting
model for justification logic $\JL_\CS$.

In modal logic, each modal axiom {\bf T}, {\bf D}, {\bf B}, {\bf
4}, and {\bf 5} characterizes a class of Kripke frames (see columns
1 and 2 of Table 1). Informally, each justification axiom characterizes a class of Fitting frames (see columns 3, 4 and 5 of Table 1).

\begin{table}[ht]
\centering\renewcommand{\arraystretch}{1.5}
\begin{tabular}{|l|l||l|l|c|}
\hline
~ Modal axiom & ~ $\RR$ is ... & ~Justification axiom & ~$\RR$ is ... &
~${\bf\E}$ satisfies ...~ \\\hline
~ {\bf T}: $\b A\r A$ & ~Reflexive & ~{\bf jT}: $t:A\r A$ & ~Reflexive & $-$ \\\hline
~ {\bf D}: $\b\bot\r\bot$ & ~Serial & ~{\bf jD}: $t:\bot\r\bot$ & ~Serial & $-$ \\\hline
~ {\bf 4}: $\b A\r\b\b A$ & ~Transitive & ~{\bf j4}: $t:A\r !t:t:A$ & ~Transitive & $\E 3$, $\E 4$\\\hline
~ {\bf B}: $\neg A\r\b\neg\b A$ & ~Symmetric & ~{\bf jB}: $\neg A\r \bar{?}t:\neg t:
A$ & ~Symmetric & $\E 5$ \\\hline
 ~ {\bf 5}: $\neg\b A\r\b\neg\b A$ & ~Euclidean & ~{\bf j5}: $\neg t:A\r ?t:\neg t:A$ & ~~~~~$-$ & $\E 6$, $\E 7$\\
  \hline
\end{tabular}\vspace{0.3cm}
\caption{Modal and justification axioms with corresponding
frame properties.}\label{table: frame properties of models}
\end{table}
\begin{theorem}[Completeness] \label{Sound Compl JL}
Let \JL~be one of the justification logics of Definition \ref{def: justification logics}. Let $\CS$ be a constant specification for \JL, with the requirement that if \JL~contains axiom scheme {\bf jD} then \CS~should be axiomatically appropriate. Then justification logics
$\JL_\CS$ are sound and complete with respect to their $\JL_\CS$-models.
\end{theorem}
\begin{proof} The first detailed presentation of the proof of completeness
theorem was presented in \cite{Fitting2005} by Fitting for \LP.
The proof of completeness for other justification logics is similar to those given in
\cite{A2008,Ghari2012-Thesis,Kuznets2008,KuznetsStuder2012,Rubtsova2006}. Completeness of justification logics ${\sf JB}$, ${\sf JTB}$, ${\sf JDB}$, ${\sf JB5}$, ${\sf JB45}$ are proved in \cite{Ghari2012-Thesis}, and of justification logic ${\sf JB}_\CS$ and all its extensions are proved in \cite{KuznetsStuder2012}.

Soundness is straightforward. Let us only check the truth of
axiom $\neg A\r \bar{?}t:\neg t:A$ in a ${\sf JB}_\CS$-model $\M=(\W, \RR,
\E, \V)$. Suppose $(\M,w)\Vdash \neg A$. By ($\E 5$) we have
$w\in\E(\bar{?}t,\neg t:A)$, for any term $t$. Now suppose $v$ is an
arbitrary world such that $w\RR v$. Since $\RR$ is symmetric we
have $v\RR w$. We claim that $(\M,v)\Vdash \neg t:A$. Indeed,
otherwise from $(\M,v)\Vdash t:A$ and $v\RR
w$, we get $(\M,w)\Vdash A$, which is a contradiction. Hence
$(\M,w)\Vdash \bar{?}t:\neg t:A$.

For completeness, we construct a canonical model $\M=(\W, \RR,
\E, \V)$ for each \JL. Let $\W$ be the set of all maximal
consistent sets in $\JL_\CS$, and define other components of $\M$
as follows. For all $\g,\d$ in $\W$:
\[ \g\RR\d \Longleftrightarrow \g^\flat \subseteq\d,  \]
\[ \E(t,A)=\{\g\in\W~|~t:A\in\g\},\]
\[\g\in\V(P) \Longleftrightarrow P\in\g\]
where $\g^\flat=\{A~|~t:A\in\g, ~for~ some~ term~ t\}$ and $P$ is a
propositional variable. We define $\Vdash$ on arbitrary formulas as
in Definition \ref{def:forcing relation}. Since $\CS\subseteq\g$ for each $\g\in\W$, $\E$ is a possible evidence function on $\W$ for $\JL_\CS$. The Truth Lemma can
be shown for $\JL_\CS$: for every formula
$F$ and every $\g\in\W$,
\[(\M,\g) \Vdash F \Longleftrightarrow F\in\g.\]
The base case follows from the definition of $\V$ in the canonical model, and the cases of Boolean connectives are standard. We only verify the case in which $F$ is of the form $t:A$. If $(\M,\g) \Vdash t:A$, then $\g\in\E(t,A)$, and therefore $t:A\in\g$. Conversely, suppose $t:A\in\g$. Then $\g\in\E(t,A)$, and by the definition of $\RR$, $A\in\d$ for each $\d\in\W$ such that $\g\RR\d$. By the induction hypothesis, $(\M,\d) \Vdash A$. Therefore, $(\M,\g) \Vdash t:A$.

 For each
justification logic \JL, it suffices to show that the canonical
model $\M$ is an \JL-model~or, equivalently, $\E$ is an admissible evidence function for $\JL_\CS$.  We only verify it for justification logic ${\sf JB}$ and its extensions. For the canonical model of ${\sf JB}$ we
establish that the accessibility relation $\RR$ is symmetric and
the possible evidence function $\E$ satisfies $\E 5$.

\textit{$\RR$ is symmetric:} Assume that $\g\RR\d$ for
$\g,\d\in\W$, and further $t:A\in\d$. It suffices to prove that
$A\in\g$. Suppose, in contrary, that $A$ is not in $\g$. Since
$\g$ is a ${\sf JB}_\CS$-maximally consistent set, $\neg A$ and
also $\neg A\r \bar{?}t:\neg t:A$ are in $\g$. Hence $\bar{?}t:\neg t:A$ is in
$\g$. Now, by $\g\RR\d$, we conclude that $\neg t:A$ is in
$\d$, which is a contradiction. Thus $A$ is in $\g$, and therefore
$\d\RR\g$.

\textit{$\E$ satisfies $\E 5$:} Suppose that $(\M,\g)\Vdash \neg
A$. Then by the Truth Lemma $\neg A\in\g$. On the other hand,
$\neg A\r \bar{?}t:\neg t: A$ is in $\g$, for every term $t$. Hence $
\bar{?}t:\neg t: A\in\g$, and therefore by the definition of evidence
function $\E$ we have $\g\in\E(\bar{?}t,\neg t: A)$.

Finally, suppose $\JL_\CS\not\vdash F$. Then the set $\{\neg F\}$
is $\JL_{\CS}$-consistent, and it can be extended to a maximal
consistent set $\g$. Hence $ F\not\in \g$, and by the Truth Lemma
$(\M,\g)\not\Vdash F$. Thus $\M\not\Vdash F$.  \qe \end{proof}

\begin{remark}
One can verify that the canonical model of all justification
logics defined in the proof of Theorems \ref{Sound Compl JL} is fully explanatory. A Fitting model $\M=(\W, \RR,
\E, \V)$ is fully explanatory if for any world $w\in\W$ and any
formula $F$, whenever $v\Vdash F$ for every $v\in\W$ such
that $w\RR v$, then for some justification term $t$ we have
$w\Vdash t:F$. This fact was first proved by Fitting in
\cite{Fitting2005} for the logic of proofs \LP. Artemov
\cite{A2008} established it for {\sf J}, and his proof can be
adapted for other justification logics.
\end{remark}

\begin{remark}\label{Remark: Strong Evidence}
It is easy to show that the canonical model of all justification
logics defined in the proof of Theorems \ref{Sound Compl JL} enjoys the strong evidence property $(\E 7)$.
 To this end, suppose $\g\in\E(t,A)$. By the definition of $\E$ in the canonical model $\M$, we have
$t:A\in\g$. Then by the Truth Lemma, $(\M,\g)\Vdash t:A$, as
desired.
\end{remark}
\begin{remark}
It is useful  to show that in the canonical model of {\sf J5} and its extensions the accessibility relation is Euclidean and also we have: if $\g\RR \d$ and $\d\in\E(t,A)$, then  $\g\in\E(t,A)$. The latter is called  anti-monotonicity condition. The proof is as follows:

$\E8$. \textit{Anti-monotonicity.} Suppose $\g\RR\d$ and
$\d\in\E(t,A)$. By the definition of $\E$, we have $t:A\in\d$.
Let us suppose $\g\not\in\E(t,A)$, or equivalently
$t:A\not\in\g$, and derive a contradiction. Since $t:A$ is not in
$\g$, and $\g$ is a maximal consistent set, we have $\neg
t:A\in\g$. Since $\neg t:A\r ?t:\neg t:A\in\g$, we have $?t:\neg t:A\in\g$. Now $\g\RR\d$ yields $\neg t:A\in\d$, which is a contradiction.

\textit{Euclideanness of $\RR$.} Suppose $\g\RR\d$, $\g\RR\Sigma$, and $t:A\in\d$. We need to show that $A\in\Sigma$. To this end it suffices to show that $t:A\in\g$ (this together with $\g\RR\Sigma$ implies that $A\in\Sigma$). Suppose towards a contradiction that $t:A\not\in\g$. Thus $\neg t:A\in\g$, and hence $?t:\neg t:A\in\g$. Now $\g\RR\d$ yields $\neg t:A\in\d$, which is a contradiction.

Thus {\sf J5} and its extensions are complete with respect to their F-models which have Euclidean accessibility relation and satisfy the anti-monotonicity condition.
\end{remark}

In the rest of this section we show how possible evidence functions can be extended to admissible evidence functions.\footnote{Another procedure for constructing (minimal) admissible evidence functions based on possible evidence functions is described by Kuznets in \cite{Kuznets2008}. In contrast to the work of Kuznets, our generated admissible evidence functions are not necessarily minimal.} First suppose that $\JL_\CS$ is a justification logic that does not contain axioms {\bf jB} and {\bf j5}, i.e. $\JL\in\{ {\sf J}, {\sf J4}, {\sf JD}, {\sf JD4}, {\sf JT}, {\sf LP}\}$.

\begin{definition}\label{def:generated evidence function}
Let $(\W, \RR)$ be a Kripke frame, and $\mathcal{A}$ be a possible evidence function on $\W$ for $\JL_\CS$. We construct possible evidence functions $\E_i$ ($i\geq 0$) inductively as follows:
 \begin{enumerate}
 \item  $\E_0(t,F)=\mathcal{A}(t,F)$.
 \item $\E_{i+1} (x,F)=\E_{i} (x,F)$, for any justification variable $x$.
  \item $\E_{i+1} (c,F)=\E_{i} (c,F)$, for any justification constant $c$.
   \item $\E_{i+1} (s+t,F)=\E_{i} (s+t,F)\cup\E_i(t,F)\cup\E_i(s,F)$.
  \item $\E_{i+1} (s\cdot t,F)=\E_{i} (s\cdot t,F)\cup\{w\in\W~|~w\in\E_i(s,G\r F)\cap\E_i(t,G)$, for some formula $G\}$.
  \item $\E_{i+1} (!t,F)=\E_{i} (! t,F)\cup\{w\in\W~|~w\in\E_i(t,G)$, where $F=t:G\}$, if {\bf j4} is an axiom of $\JL_\CS$.
  \item $\E_{i+1} (t,F)=\E_{i} (t,F)\cup\{w\in\W~|~v\in\E_i(t,F), v\RR w\}$, if {\bf j4} is an axiom of $\JL_\CS$.
\end{enumerate}
  The evidence function $\E_\mathcal{A}$ based on $\mathcal{A}$ is defined as follows: $\E_\mathcal{A}(t,F)= \bigcup_{i=0}^\infty \E_i(t,F)$, for any term $t$ and any formula $F$.
\end{definition}

Note that $\E_i(t,F)\subseteq \E_{i+1}(t,F)$, for any term $t$, any formula $F$, and any $i\geq 0$.

  \begin{lemma}\label{lem: generated evidence function is admissible}
Let $\JL_\CS$ a justification logic that does not contain axioms {\bf jB} and {\bf j5}, $(\W, \RR)$ be a Kripke frame for  $\JL_\CS$, $\mathcal{A}$ be a possible evidence function on $\W$ for  $\JL_\CS$, and $\E_\mathcal{A}$ be the evidence function based on $\mathcal{A}$ as defined in Definition \ref{def:generated evidence function}. Then $\E_\mathcal{A}$ satisfies the application $(\E1)$ and sum $(\E2)$ conditions. If $\JL_\CS$ contains axiom {\bf j4}, then $\E_\mathcal{A}$ also satisfies the monotonicity $(\E3)$ and positive introspection $(\E4)$ conditions.
\end{lemma}

Since $\mathcal{A}$ is a possible evidence function on $\W$ for $\JL_\CS$, we have $\mathcal{A}(c,F)=\W$ for every $c:F\in\CS$. By the definition of $\E_\mathcal{A}$ we obtain $\E_\mathcal{A}(c,F)=\W$, for every $c:F\in\CS$. Thus, the evidence function $\E_\mathcal{A}$ is an admissible evidence function for $\JL_\CS$ based on $\mathcal{A}$.

\begin{remark}
Definition \ref{def:generated evidence function} can be extended to justification logics $\JL$ that contain axiom {\bf jB} but not axiom {\bf j5}, i.e. $\JL \in\{{\sf JB}, {\sf JDB}, {\sf JTB} , {\sf JB4}, {\sf JDB4}, {\sf JTB4} \}$. In the definition of $\E_i$, change the base case $i=0$ as follows:
\begin{enumerate}
\item $\E_0(t,F)= \W$, if $t=\bar{?}r$ and $F=\neg r:A$, for some term $r$ and formula $A$; and $\E_0(t,F)=\mathcal{A}(t,F)$, otherwise.
\end{enumerate}
Furthermore, add the following clause:
\begin{description}
\item[$8.$]  $\E_{i+1} (\bar{?}t,F)=\E_{i} (\bar{?}t,F)$.
\end{description}
Now it is easy to show that  $\E_\mathcal{A}$ also satisfies the weak negative introspection $(\E5)$ condition.
\end{remark}

Next we show how Definition \ref{def:generated evidence function} can be extended to all justification logics. First we need some definitions.

\begin{definition}
The rank of a justification term is defined as follows:
\begin{enumerate}
\item $rk(x)=rk(c)=0$, for any justification variable $x$ and any justification constant $c$,
\item $rk(s+t)= rk(s\cdot t)=\max(rk(s),rk(t)) + 1$,
\item $rk(!t)= 
rk(?t)=rk(t) + 1$.
\end{enumerate}
\end{definition}

In the following definition and lemmas the accessibility relation of Kripke frames for {\sf J5} and its extensions is required to be Euclidean. The properties of the accessibility relation of Kripke frames for other justification logics are as before (see Definition \ref{Kripke-Fitting models JL} or Table \ref{table: frame properties of models}).

\begin{definition}
Let $(\W,\RR)$ be a Kripke frame, and $S \subseteq \W$.
\begin{itemize}
\item $S$ is called monotone with respect to $\RR$ if for all $w,v\in\W$, $w\in S$ and $w \RR v$ imply $v\in S$.
\item $S$ is called anti-monotone with respect to $\RR$ if for all $w,v\in\W$, $w\in S$ and $v \RR w$ imply $v\in S$.
\item $S$ is called stable with respect to $\RR$ if for all $w,v\in\W$, $w \RR v$ implies $w\in S$ iff $v\in S$.
\item The monotone (anti-monotone, stable) closure of $S$ is the smallest subset of $\W$ containing $S$ which is monotone (respectively, anti-monotone, stable) with respect to $\RR$.
\end{itemize}
\end{definition}

Note that $S \subseteq \W$ is stable with respect to $\RR$ if and only if it is both monotone and anti-monotone with respect to $\RR$. It is easy to show the following
\begin{itemize}
\item $w$ is in the monotone closure of $S$ if and only if there is a sequence of related worlds $w_0 \RR w_1 \RR$ $\ldots \RR w_n$ ($n\geq 0$) such that $w_n=w$ and $w_0\in S$.
\item $w$ is in the anti-monotone closure of $S$ if and only if there is a sequence of related worlds $w_0 \RR w_1 \RR$ $\ldots \RR w_n$ ($n\geq 0$) such that $w=w_0$ and $w_n\in S$.
\item $w$ is in the stable closure of $S$ if and only if $w$ is in the monotone closure or in the anti-monotone closure of $S$.
\end{itemize}

 The following definition is inspired by the work of Studer \cite{Studer2013} which introduced \textit{inductively generated evidence functions} for singleton Fitting models (although our definition of inductively generated evidence function for {\sf J5} and its extension is not quite the same as Studer's).

\begin{definition}\label{def:Inductively generated evidence function j5}
Let $\JL_\CS$ be a justification logic, $(\W, \RR)$ be a Kripke frame for  $\JL_\CS$, and $\mathcal{A}$ be a possible evidence function on $\W$ for $\JL_\CS$. We construct possible evidence functions $\E_i$ ($i\geq 0$) inductively as follows:
 \begin{enumerate}
\item $\E_0(t,F)= \W$, if $t=\bar{?}r$ and $F=\neg r:A$, for some term $r$ and formula $A$; and $\E_0(t,F)=\mathcal{A}(t,F)$, otherwise.

Note that $\E_0$ is based on $\mathcal{A}$. Moreover, if $\JL$ does not contain axiom {\bf jB} (and hence the term construction $\bar{?}$), then $\E_0(t,F)= \mathcal{A}(t,F)$, for any justification term $t$ and any formula $F$.
\item $\E_{i+1} (x,F)=\E_{i} (x,F)$, for any justification variable $x$.
  \item $\E_{i+1} (c,F)=\E_{i} (c,F)$, for any justification constant $c$.
     \item $\E_{i+1} (s\cdot t,F)=\E_{i} (s\cdot t,F)\cup\{w\in\W~|~w\in\E_i(s,G\r F)\cap\E_i(t,G)$, for some formula $G$, and $rk(s\cdot t)=i+1\}$.
\item $\E_{i+1} (s+t,F)=\E_{i} (s+t,F)\cup\{w\in\W~|~w\in\E_i(t,F)$ and $rk(s+t)=i+1\}\cup\{w\in\W~|~w\in\E_i(s,F)$ and $rk(s+t)=i+1\}$.
  \item $\E_{i+1} (!t,F)=\E_{i} (! t,F)\cup\{w\in\W~|~w\in\E_i(t,G)$, for some formula $G$ such that $F=t:G$, and $rk(! t)=i+1\}$, if {\bf j4} is an axiom of $\JL_\CS$.
\item $\E_{i+1} (\bar{?}t,F)=\E_{i} (\bar{?}t,F)$, if {\bf jB} is an axiom of $\JL_\CS$.
  \item $\E_{i+1} (?t,F)=\E_{i} (? t,F)\cup\{w\in\W~|~w\not\in\E_i(t,G)$, for some formula $G$ such that $F=\neg t:G$, and $rk(? t)=i+1\}$, if {\bf j5} is an axiom of $\JL_\CS$.
     \end{enumerate}
Let $\bar{\E}(t,F)= \bigcup_{i=0}^\infty \E_i(t,F)$, for any term $t$ and any formula $F$. Inductively generated evidence function $\E$ based on $\mathcal{A}$ is defined as follows: for any term $t$ and any formula $F$
\begin{itemize}
  \item  $\E(t,F)=\bar{\E}(t,F)$, if $\JL \in\{ {\sf J}, {\sf JB}, {\sf JD},{\sf JT}, {\sf JDB}, {\sf JTB} \}$.
  \item $\E(t,F)$ is the monotone closure of $\bar{\E}(t,F)$, if $\JL \in\{ {\sf J4}, {\sf JD4}, {\sf JB4}, {\sf JDB4}, {\sf JTB4}, {\sf LP} \}$.
  \item $\E(t,F)$ is the anti-monotone closure of $\bar{\E}(t,F)$, if $\JL \in\{ {\sf J5}, {\sf JD5}, {\sf JT5}, {\sf JB5}, {\sf JTB5}, {\sf JDB5} \}$.
  \item $\E(t,F)$ is the stable closure of $\bar{\E}(t,F)$, if $\JL \in\{ {\sf J45}, {\sf JD45},{\sf JT45}, {\sf JB45}, {\sf JTB45}, {\sf JDB45}\}$.
\end{itemize}
\end{definition}

Note that $\E_i(t,F)\subseteq \E_{i+1}(t,F)$, for any term $t$, any formula $F$, and any $i\geq 0$. Thus, $\E_k(t,F) \subseteq \E_l(t,F)$, for any term $t$, any formula $F$, and any $k\leq l$. Using this fact it is easy to see that if $w\in\bar{\E}(t,F)$ and $rk(t)\leq i$, then $w\in\E_i(t,F)$.

\begin{lemma}\label{lem:Inductively generated evidence function is admissible}
Let $\JL_\CS$ a justification logic, $(\W, \RR)$ be a Kripke frame for  $\JL_\CS$, $\mathcal{A}$ be a possible evidence function on $\W$ for  $\JL_\CS$, and $\E$ be the inductively generated evidence function based on $\mathcal{A}$. Then $\E$ satisfies the application $(\E1)$ and sum $(\E2)$ conditions. If $\JL_\CS$ contains axiom {\bf j4}, then $\E$ also satisfies the monotonicity $(\E3)$ and positive introspection $(\E4)$ conditions. If $\JL_\CS$ contains axiom {\bf jB}, then $\E$ also satisfies the weak negative introspection $(\E5)$ condition. If $\JL_\CS$ contains axiom {\bf j5}, then $\E$ also satisfies the negative introspection $(\E6)$ and anti-monotonicity $(\E8)$ conditions.
\end{lemma}

\begin{lemma}\label{lem:Inductively generated evidence function j5 is strong}
Let $\JL_\CS$ be one of the justification logics that contains  axiom {\bf j5}, $(\W, \RR)$ be a Kripke frame for  $\JL_\CS$, $\mathcal{A}$ be a possible evidence function on $\W$ for  $\JL_\CS$,  $\E$ be the inductively generated evidence function based on $\mathcal{A}$, $\E_0$ be the possible evidence function defined in Definition \ref{def:Inductively generated evidence function j5}, and $\V$ be a valuation. Consider the possible Fitting model $\M=(\W,\RR,\E,\V)$. If for all terms $r\in Tm_\JL$, all formulas $F\in Fm_\JL$, and all $w\in\W$ we have
\[ w\in \E_0(r,F)\Rightarrow (\M,w)\Vdash r:F,\]
then $\E$ satisfies the strong evidence condition $(\E7)$.
\end{lemma}
\begin{proof} By induction on $j$ we show that for all $w\in\W$

\begin{equation}\label{eq:Inductively generated evidence function j5 is strong}
w\in \E_j(r,F)\Rightarrow (\M,w)\Vdash r:F.
\end{equation}

 The  base case, $j=0$, follows from the assumption. For the induction hypothesis, suppose that (\ref{eq:Inductively generated evidence function j5 is strong}) is true for all $w\in\W$ and all $0 \leq j \leq i$.
For any $w\in\W$, if $w\in\E_{i+1}(r,F)$ for $i\geq 0$, and $w\in\E_j(r,F)$ for $j< i+1$, then by the induction hypothesis $(\M,w)\Vdash r:F$. Assume now that for an arbitrary $w\in\W$ we have $w\in\E_{i+1}(r,F)$ for $i\geq 0$, and $w\not\in\E_j(r,F)$ for any $j< i+1$. We have the following cases:
 \begin{enumerate}
  \item $r=s\cdot t$ and $w\in\E_{i+1}(s\cdot t,F)$. Then $w\in\E_{i}(s,G\r F)\cap\E_{i}(t,G)$, for some formula $G$, and $rk(s\cdot t)=i+1$. Thus, by the induction hypothesis, $(\M,w)\Vdash s:(G\r F)$ and $(\M,w)\Vdash t:G$. Therefore, $w\in \E(s,G\r F)$ and $w\in \E(t,G)$. By Lemma \ref{lem:Inductively generated evidence function is admissible}, $w \in\E(s\cdot t,F)$. Now it is easy to show that for arbitrary world $v\in\W$ such that $w\RR v$ we have $(\M,v)\Vdash F$. Hence $(\M,w)\Vdash s\cdot t:F$, which is what we wished to show.
   \item $r=s+t$ and $w\in\E_{i+1}(s+t,F)$. Thus $w\in\E_{i}(t,F)$ or $w\in\E_{i}(s,F)$, and $rk(s+t)=i+1$. If $w\in\E_{i}(t,F)$, then by the induction hypothesis, $(\M,w)\Vdash t:F$. Therefore, $w \in\E(t,F)$. By Lemma \ref{lem:Inductively generated evidence function is admissible},  $w \in\E(s+t,F)$. On the other hand, for arbitrary world $v\in\W$ such that $w\RR v$ we have $(\M,v)\Vdash F$. Hence $(\M,w)\Vdash s+ t:F$. Proceed similarly if  $w\in\E_{i}(s,F)$.
   \item {\bf j4} is an axiom of $\JL_\CS$, $r=!t$, and $w\in\E_{i+1}(!t,F)$. Then $w\in\E_{i}(t,G)$, for some formula $G$ such that $F=t:G$, and $rk(!t)=i+1$. Thus, by the induction hypothesis, $(\M,w)\Vdash t:G$. Therefore, $w \in\E(t,G)$. By Lemma \ref{lem:Inductively generated evidence function is admissible},  $w \in\E(!t,F)$. Since $\RR$ is transitive and $\E$ satisfies the monotonicity condition, it is easy to show that for any arbitrary world $v\in\W$ such that $w\RR v$ we have $(\M,v)\Vdash F$. Hence $(\M,w)\Vdash !t:F$.
  \item {\bf j5} is an axiom of $\JL_\CS$, $r=?t$, and $w\in\E_{i+1}(?t,F)$. Then $w\not\in\E_{i}(t,G)$, for some formula $G$ such that $F=\neg t:G$, and $rk(?t)=i+1$. For arbitrary $v\in\W$ such that $w\RR v$, we need to show that $(\M,v)\Vdash \neg t:G$. Suppose towards a contradiction that $v\in\E(t,G)$. By anti-monotonicity, $w\in\E(t,G)$. This, together with $rk(t)=i$, implies that $w\in\E_i(t,G)$, which is a contradiction. Thus $v\not\in\E(t,G)$, and hence $(\M,v)\Vdash \neg t:G$. Therefore, $(\M,w)\Vdash ?t:\neg t:G$.
\end{enumerate}
Finally we show that $\E$ satisfies the strong evidence condition $(\E7)$. Assume that $w\in\E(r,F)$. We will show that $(\M,w)\Vdash r:F$.

First suppose that $\JL \in\{ {\sf J5}, {\sf JD5}, {\sf JT5}, {\sf JB5}, {\sf JTB5}, {\sf JDB5} \}$. Since $\E(r,F)$ is the anti-monotone closure of $\bar{\E}(r,F)$, there is a sequence of related worlds $w_0 \RR w_1 \RR$ $\ldots \RR w_n$ ($n\geq 0$) such that $w=w_0$ and $w_n\in \bar{\E}(r,F)$. Thus, $w_n\in \E_i(r,F)$ for some $i\geq 0$. By (\ref{eq:Inductively generated evidence function j5 is strong}), $(\M,w_n)\Vdash r:F$. Since $\RR$ is Euclidean, it is easy to show that for any arbitrary world $v\in\W$ such that $w\RR v$ we have $w_n \RR v$. Hence $(\M,v)\Vdash F$, and therefore $(\M,w)\Vdash r:F$.

Now suppose that $\JL \in\{ {\sf J45}, {\sf JD45},{\sf JT45}, {\sf JB45}, {\sf JTB45}, {\sf JDB45}\}$. Since $\E(r,F)$ is the stable closure of $\bar{\E}(r,F)$, $w$ is in the monotone closure or in the anti-monotone closure of $\bar{\E}(r,F)$. The latter case  is treated similar to the above argument. For the former case suppose that $w$ is in the monotone closure of $\bar{\E}(r,F)$. Thus there is a sequence of related worlds $w_0 \RR w_1 \RR$ $\ldots \RR w_n$ ($n\geq 0$) such that $w=w_n$ and $w_0\in \bar{\E}(r,F)$. Thus, $w_0\in \E_i(r,F)$ for some $i\geq 0$. By (\ref{eq:Inductively generated evidence function j5 is strong}), $(\M,w_0)\Vdash r:F$. Since $\RR$ is transitive, it is easy to show that for any arbitrary world $v\in\W$ such that $w\RR v$ we have $w_0 \RR v$. Hence $(\M,v)\Vdash F$, and therefore $(\M,w)\Vdash r:F$. \qe
\end{proof}

Lemmas \ref{lem:Inductively generated evidence function is admissible} and  \ref{lem:Inductively generated evidence function j5 is strong} informally states that: every possible evidence function for $\JL$, which does not contain axiom {\bf j5}, can be extended to an admissible evidence function; and every possible evidence function for $\JL$, which does  contain axiom {\bf j5} but not axiom {\bf jB}, that satisfies the strong evidence condition can be extended to an admissible evidence function.

\section{Labeled sequent calculus}\label{sec:Labeled sequent calculus}
Negri and von Plato in \cite{Negri Plato1998} proposed a method to
transform universal axioms into rules of Gentzen system. 
 Universal
axioms are first transformed into conjunctions of formulas of the
form $P_1\wedge\ldots\wedge P_m\r Q_1\vee \ldots\vee Q_n$, where
$P_i$ and $Q_j$ are atomic formulas. Then each conjunct is
converted into the \emph{regular rule scheme}:
\[\di{\f{Q_1,\overline{P},\g\R\d \quad\cdots\quad Q_n,\overline{P},\g\R\d}{\overline{P},\g\R\d}Reg}\]
in which $\overline{P}$ abbreviates $P_1,\ldots,P_m$. In \cite{Negri2003} the method was extended to transform geometric axioms of the form $\forall \bar{z}(P_1\wedge\ldots\wedge P_m\r\exists x_1 M_1\vee \ldots\vee \exists x_n M_n)$, where each $M_i$ is a conjunction of atomic formulas $Q_{i_1},\ldots,Q_{i_k}$, to rules by \emph{geometric rule scheme}:
\[\di{\f{\overline{Q_1}(y_1/x_1),\overline{P},\g\R\d \quad\cdots\quad \overline{Q_n}(y_n/x_n),\overline{P},\g\R\d}{\overline{P},\g\R\d}GRS}\]
where $\overline{Q_i}$ abbreviates $Q_{i_1},\ldots,Q_{i_k}$, and the eigenvariables $y_1,\ldots ,y_n$ of the premises are not free in the conclusion. $A(y/x)$ indicates $A$ after the substitution of the variable $y$ for the variable $x$.

Using $Reg$ and $GRS$, Sara Negri in \cite{Negri2005}  proposed cut-free labeled sequent
calculi for a wide variety of modal logics characterized by Kripke
models. By adopting a labeled language including possible worlds
as labels and the accessibility relation, she presented G3-style  sequent calculi (i.e. sequent calculi without structural rules) for normal modal logics. In this section, we
extend this method to present labeled sequent calculi for
justification logics.

Let us first define the language of labeled sequent calculus,
called \textit{labeled language}, for propositional logic. This
language will be extended later in sections 4.1, 4.2 and 4.3. We
need a countably infinite set $L$ of labels $\ww,\vv,\uu, \ldots$
which are used in the labeled systems as possible worlds of Kripke
style models. The labeled language for propositional logic
consists of \textit{labeled formulas} (or \textit{forcing
formulas}) $\ww\Vvdash A$, where $A$ is a formula of propositional
logic. Sequents are expressions of the form $\g\R\d$, where $\g$
and $\d$ are multisets of formulas in the labeled language.
\begin{table}[t]
\centering\renewcommand{\arraystretch}{2.3}
\begin{tabular}{|ll|}
\hline
~\noindent{\bf Initial sequents:}&\\
~$\ww \Vvdash P,\Gamma\R\d,\ww\Vvdash P$ ~~~~~$(Ax)$&
$\ww\Vvdash\perp,\g\R\d$~~~~~$(Ax\bot)$\\
~$P$ is a propositional variable. &\\
~\noindent {\bf Propositional rules:}&\\
~$\di{\f{\g\R\d,\ww\Vvdash A}{\ww\Vvdash \neg A,\g\R\d}}(L\neg)$ & $\di{\f{\ww\Vvdash A,\g\R\d}{\g\R\d,\ww\Vvdash \neg A}}(R\neg)$ \\
~$\di{\f{\ww\Vvdash A,\ww\Vvdash B,\g\R\d}{\ww\Vvdash A\wedge B,\g\R\d}}(L\wedge)$ & $\di{\f{\g\R\d,\ww\Vvdash A ~~~~\g\R\d,\ww\Vvdash B}{\g\R\d,\ww\Vvdash A\wedge B}}(R\wedge)$ \\
~$\di{\f{\ww\Vvdash A,\g\R\d~~~~\ww\Vvdash B,\g\R\d}{\ww\Vvdash A\vee B,\g\R\d}}(L\vee)$ & $\di{\f{\g\R\d,\ww\Vvdash A,\ww\Vvdash B}{\g\R\d,\ww\Vvdash A\vee B}}(R\vee)$ \\
~$\di{\f{\g\R\d,\ww\Vvdash A~~~~\ww\Vvdash B,\g\R\d}{\ww\Vvdash A\r
B,\g\R\d}}(L\r)$ & $\di{\f{\ww\Vvdash A,\g\R\d,\ww\Vvdash
B}{\g\R\d,\ww\Vvdash A\r B}}(R\r)$\\
\hline
\end{tabular}\vspace{0.3cm}
\caption{Labeled sequent calculus {\sf G3c} for propositional
logic.}\label{table: rules of labeled G3c}
\end{table}
Initial sequents (or axioms) and rules for the labeled sequent calculus  {\sf G3c} of propositional logic are similar to the
ordinary sequent calculus which are augmented by labels, see Table
\ref{table: rules of labeled G3c}.

In Sections 4.1-4.3, in order to construct labeled sequent calculi for modal and justification logics, we will extend the underlying language of the formula $A$ in $\ww\Vvdash A$ to formulas in the language of modal and justification logics. Moreover, we will extend the labeled language of {\sf G3c} to an \textit{extended labeled language} with new atoms (e.g. relational, evidence). When we extend a labeled sequent calculus ${\sf G3L}$ of logic $L$ to ${\sf G3L'}$ of logic $L'$ (i.e. ${\sf G3L'}$ contains all initial sequents and rules of ${\sf G3L}$), all formulas $A$ in labeled formulas $\ww\Vvdash A$, in initial sequents and rules of ${\sf G3L'}$, should be considered in the language of $L'$, and also $\g,\d$ may now contain the new atoms of the extended labeled language.

 In all of the rules and initial sequents of this paper, formulas in $\g$ and $\d$ are called \textit{side formulas}. The formula(s) in the premise(s) and conclusion of a rule not in the side formulas are called \textit{active} and \textit{principal}, respectively. Also, the formula(s) in an initial sequent not in the side formulas are called \textit{principal}.

Negri \cite{Negri2005} considered a zero-premise rule $L\bot$
instead of initial sequent $(Ax\bot)$. She used
$w:A$ to denote labeled formulas in the labeled language. To
avoid confusion with justification formulas of the form $t:A$, we
replace it by $\ww\Vvdash A$. In addition, compare to
\cite{Negri2005} we use $\r$ in place of $\supset$ for
implication, $\wedge$ in place of $\&$ for conjunction, and for
simplicity, we deal only with $\b$ and take $\Diamond$ as a
definable modality in the modal language. We also add the rules $(L\neg)$ and $(R\neg)$
to {\sf G3c}. The name of rules given here is also different from
Negri's.
 \subsection{Labeled sequent calculi of modal logics}
Extended labeled language for modal logics consists of labeled formulas $\ww\Vvdash A$,
in which $A$ is a modal formula, and \textit{relational atoms} (or
\textit{accessibility atoms}) $\ww R\vv $. Negri in \cite{Negri2005}
presented the modal rules $(L\b)$ and $(R\b)$ for modal logics
(see Table \ref{table: G3K}). Rule $(R\b)$ has the restriction
that the label $v$ (called \emph{eigenlabel}) must not occur in
the conclusion. To derive the properties of the accessibility
relation (such as reflexivity, symmetric, transitivity, etc.)
initial sequents for $R$ (denoted by $(AxR)$) are added.
\begin{table}[t]
\centering\renewcommand{\arraystretch}{1.6}
\begin{tabular}{|p{8cm}c|}
\hline
~\noindent{\bf Initial sequent:} & \\
 \multicolumn{2}{|c|}{$\ww R\vv,\g\R\d, \ww R\vv $~~~~~$(AxR)$}\\
~\noindent{\bf Modal rules:} & \\ ~$\di{\f{\vv\Vvdash
A,\ww\Vvdash \b A,\ww R\vv,\g\R\d}{\ww\Vvdash \b A,\ww R\vv,\g\R\d}}(L\b)$ &
$\di{\f{\ww R\vv,\g\R\d,\vv\Vvdash A}{\g\R\d,\ww\Vvdash \b A}}(R\b)$~ \\
\multicolumn{2}{|c|}{In $(R\b)$ the eigenlabel $\vv$ must not occur in
the conclusion of rule.}\\
\hline
\end{tabular}\vspace{0.3cm}
\caption{Initial sequent and rules for labeled sequent calculus {\sf G3K} to be added to {\sf G3c}.}\label{table: G3K}
\end{table}
All the axioms and rules in Table \ref{table: rules of labeled
G3c} and \ref{table: G3K} constitute a G3-style labeled system
for the basic modal logic {\sf K}, denoted by {\sf G3K}. Labeled systems for other modal logics
are obtained by adding rules, that correspond to the properties of
the accessibility relation in their Kripke models from Table
\ref{table: rules for accessibility relation}, according to Table \ref{table: G3ML}. For example,
\begin{equation}
\begin{array}{cll}
{\sf G3DB}&=& {\sf G3K} + (Ser) + (Sym).\\
{\sf G3S4}&=& {\sf G3K} + (Ref) + (Trans).\\
{\sf G3S5}&=& {\sf G3K} + (Ref) + (Trans) + (Eucl) + (Eucl_*).
\end{array}
\end{equation}
Note that contracted instances of rule $(Eucl)$:
$$\di{\f{\vv R\vv,\ww R\vv,\g\R\d}{\ww R\vv,\g\R\d}}(Eucl_*)$$
should be added to those labeled systems that contain rule $(Eucl)$. In the seriality rule $(Ser)$ the eigenlabel $\vv$ must not occur in
the conclusion (this rule is obtained by the geometric rule scheme
$GRS$). All fifteen modal logics of the modal cube have labeled sequent calculus.
Negri in \cite{Negri2005,Negri2009} shows that all these
labeled systems are sound and complete with respect to their
Hilbert systems, all of the rules in the modal labeled systems
are invertible and structural rules (weakening and contraction rules) and
cut are admissible. The termination of proof search was proved for {\sf G3K}, {\sf G3T}, {\sf G3KB}, {\sf G3TB}, {\sf G3S4}, and {\sf G3S5}.
\begin{table}[t]
\centering\renewcommand{\arraystretch}{2.5}
\begin{tabular}{|cc|}
\hline
~ $\di{\f{\ww R\ww,\g\R\d}{\g\R\d}}(Ref)$\hspace{2.7cm}& $\di{\f{\vv R\ww,\ww R\vv,\g\R\d}{\ww R\vv,\g\R\d}}(Sym)$~~~ \\
~ $\di{\f{\ww R\vv,\g\R\d}{\g\R\d}}(Ser)$\hspace{2.7cm}& $\di{\f{\ww R\uu,\ww R\vv,\vv R\uu,\g\R\d}{\ww R\vv,\vv R\uu,\g\R\d}}(Trans)$\\
 \multicolumn{2}{|c|}{$\di{\f{\vv R\uu,\ww R\vv,\ww R\uu,\g\R\d}{\ww R\vv,\ww R\uu,\g\R\d}}(Eucl)$}\\
  \multicolumn{2}{|c|}{~In $(Ser)$ the eigenlabel $\vv$ must not occur in the conclusion of the rule.~}  \\
  \hline
\end{tabular}\vspace{0.3cm}
\caption{Rules for relational atoms.}\label{table: rules for accessibility relation}
\end{table}
\begin{table}[t]
\centering
\renewcommand{\arraystretch}{1.2}
\begin{tabular}{|c|c|}
\hline
~Modal axiom & ~Corresponding rule  \\\hline
{\bf T} & $(Ref)$\\
{\bf D} & $(Ser)$\\
{\bf B} & $(Sym)$\\
{\bf 4} &  $(Trans)$\\
{\bf 5} & $(Eucl)$, $(Eucl_*)$\\
\hline
\end{tabular}\vspace{0.3cm}
\caption{Corresponding rules to be added to {\sf G3K} for labeled sequent calculi of modal logics.}\label{table: G3ML}
\end{table}
\subsection{Labeled sequent calculus based on F-models}
Extended labeled language for
justification logics consist of labeled formulas $\ww\Vvdash A$,
relational atoms $\ww R\vv$, and \textit{evidence atoms} $\ww  E(t,A)$,
where $\ww$ and $\vv$ are labels from $L$, $t$ is a justification term
and $A$ is a \JL-formula. These atoms
respectively denote the statements $(\M,w)\Vdash A$, $w\RR v$ and
$w\in \E(t,A)$ in Fitting models. Thus in this language we are
able to give a symbolic presentation of semantical elements of
Fitting models. For example, application $(\E1)$ and negative introspection $(\E6)$ conditions could be expressed in the extended labeled language as the following sequents:

\begin{eqnarray*}
\ww E(s,A\r B), \ww E(t,A) \R \ww E(s\cdot t,B),\\
 \R \ww E(t,A) , \ww E(? t,\neg t:A).
\end{eqnarray*}

If we allow that the propositional connectives are defined on formulas in the extended labeled language, then conditions $\E1$ and  $\E6$ could be expressed as follows:

\begin{eqnarray*}
\ww E(s,A\r B) \wedge \ww E(t,A) \r \ww E(s\cdot t,B),\\
 \ww E(t,A) \vee \ww E(? t,\neg t:A).
\end{eqnarray*}

 By the definition of forcing relation for formulas of the form
$t:A$ in F-models:
\[(\M, w)\Vdash t:A~~ \mbox{if{f}}~~ w\in\E(t,A)~~\mbox{and for every}~~v\in \W~~\mbox{with}~~w \RR v, (\M, v)\Vdash A,\]
and the regular rule scheme $(Reg)$, we obtain the rules $(L:)$
and $(E)$ for justification logics, see Table \ref{table: rules
for G3J}. Rule $(R:)$ is obtained similar to rule $(R\b)$ in modal
labeled systems, and likewise the eigenlabel $v$ in the premise
of the rule $(R:)$ must not occur in the conclusion. Again to
derive the properties of admissible  evidence function, we should add
initial sequents for evidence atoms $(AxE)$ from Table \ref{table: rules
for G3J}.

\begin{table}
\centering\renewcommand{\arraystretch}{2.2}
\begin{tabular}{|lc|}
\hline
~\noindent{\bf Initial sequents}:&\\
 \multicolumn{2}{|c|}{$\ww R\vv,\g\R\d, \ww R\vv$~~~~~$(AxR)$}\\
 \multicolumn{2}{|c|}{$\ww E(t,A),\g\R\d, \ww E(t,A)$~~~~~$(AxE)$}\\
~\noindent{\bf JL rules:}&\\
~ $\di{\f{\vv\Vvdash A,\ww\Vvdash
t:A,\ww R\vv,\g\R\d}{\ww\Vvdash t:A,\ww R\vv,\g\R\d}}(L:)$ & $\di{\f{\ww R\vv,\ww E(t,A),\g\R\d,\vv\Vvdash
A}{\ww E(t,A),\g\R\d,\ww\Vvdash t:A}}(R:) $ ~\\
\multicolumn{2}{|c|}{$\di{\f{\ww E(t,A),\ww\Vvdash
t:A,\g\R\d}{\ww\Vvdash
t:A,\g\R\d}}(E)$}\\
\multicolumn{2}{|c|}{In $(R:)$ the eigenlabel $\vv$ must not occur in
the conclusion of rule.}\\
~\noindent {\bf Rules for evidence atoms:}&\\
 ~   $\di{\f{\ww E(s+t,A),\ww E(t,A),\g\R\d}{\ww E(t,A),\g\R\d}}(El+)$ & $\di{\f{\ww E(t+s,A),\ww E(t,A),\g\R\d}{\ww E(t,A),\g\R\d}}(Er+)$\\
     \multicolumn{2}{|c|}{$\di{\f{\ww E(s\cdot t,B),\ww E(s,A\r B),\ww E(t,A),\g\R\d}{\ww E(s,A\r B),\ww E(t,A),\g\R\d}}(E\cdot)$}\\
~  \noindent {\bf Iterated axiom necessitation rule:}&\\
\multicolumn{2}{|c|}{$\di{\f{\ww E(c_{i_n},c_{i_{n-1}}:\ldots:c_{i_1}:A),\g\R\d}{\g\R\d}}(IAN)$}\\
\hline
\end{tabular}\vspace{0.3cm}
\caption{Initial sequents and rules for labeled sequent calculus {\sf G3J} to be added to {\sf G3c}.}\label{table: rules for G3J}
\end{table}

\begin{table}
\centering\renewcommand{\arraystretch}{2.2}
\begin{tabular}{|lc|}
\hline
~\noindent {\bf Rules for evidence atoms:}&\\
~$\di{\f{\ww E(!t,t:A),\ww E(t,A),\g\R\d}{\ww E(t,A),\g\R\d}}(E!)$ &
   $\di{\f{\vv E(t,A),\ww E(t,A),\ww R\vv,\g\R\d}{\ww E(t,A),\ww R\vv,\g\R\d}}(Mon)$ \\
    \multicolumn{2}{|c|}{$\di{\f{\ww\Vvdash A,\g\R\d \quad \ww E(\bar{?}t,\neg t:A),\g\R\d}{\g\R\d}}(E\bar{?})$}\\
~ $\di{\f{\ww\Vvdash t:A,\ww E(t,A),\g\R\d}{\ww E(t,A),\g\R\d}}(SE)$    & $\di{\f{\ww E(t,A),\g\R\d \quad \ww E(?t,\neg t:A),\g\R\d}{\g\R\d}}(E?)$~ \\
~ \noindent {\bf Axiom necessitation rule:}&\\
\multicolumn{2}{|c|}{$\di{\f{\ww E(c,A),\g\R\d}{\g\R\d}}(AN)$}\\
\hline
\end{tabular}\vspace{0.3cm}
\caption{Rules for labeled sequent calculi based
on F-models.}\label{table: rules for G3JL}
\end{table}
Now we define labeled sequent calculi for various justification
logics according to properties of $\RR$ and $\E$ in their
F-models. Labeled system {\sf G3J} is the extension of {\sf G3c}, Table
\ref{table: rules of labeled G3c}, by the initial sequents and
rules from Table \ref{table: rules for G3J}.
 Table \ref{table: G3JL}
specifies which rules must be added to {\sf
G3J} from Table \ref{table: rules for G3JL} to get labeled systems {\sf G3JL} for various justification logics. All rules of Tables \ref{table: rules for G3J}, \ref{table: rules for G3JL} are obtained from regular rule scheme $Reg$.

Labeled system ${\sf G3J4}$ and its extensions possesses rule
$(AN)$ and other labeled systems rule $(IAN)$. In the rules $(AN)$ and $(IAN)$, the formula $A$ is an axiom of \JL, $c$ and $c_i$'s are
justification constants. The labeled sequent calculus $\J_\CS$ is obtained from \J~by restricting rule $(IAN)/(AN)$ as follows: for each evidence atom $\ww E(c_{i_n},c_{i_{n-1}}:\ldots:c_{i_1}:A)$ (or $\ww E(c,A)$) in the premise of the rule $(IAN)$ (or $(AN)$) we shuold have $c_{i_n}:c_{i_{n-1}}:\ldots:c_{i_1}:A\in\CS$ (or $c:A\in\CS$). \J~denotes $\J_{{\sf TCS_\JL}}$. In the rest of the paper whenever we use $\J_\CS$ it is assumed that \CS~is an arbitrary  constant specification for $\JL$, unless stated otherwise.

\begin{table}[t]
\centering
\renewcommand{\arraystretch}{1.2}
\begin{tabular}{|c|c|}
\hline
~Justification axiom &~ Corresponding rule  \\\hline
{\bf jT} & $(Ref)$\\
{\bf jD} & $(Ser)$\\
{\bf jB} & $(E\bar{?})$,$(Sym)$\\
{\bf j4} & ~$(E!), (Mon), (Trans)$~\\
{\bf j5} & $(SE),(E?)$\\
\hline
\end{tabular}\vspace{0.3cm}
\caption{Corresponding rules to be added to {\sf G3J} for labeled sequent calculi of justification logics.}\label{table: G3JL}
\end{table}

\section{Basic properties}\label{sec:Basic properties}
 In the rest of the paper, for a justification logic \JL~defined in Definition \ref{def: justification logics}, let \J~denote its corresponding labeled sequent calculus based on F-models. Let \JJ~denote those labeled sequent calculi which do not contain rules $(SE)$ or $(E\bar{?})$ in its formulation. Thus \JJ~could be each of the labeled sequent calculus {\sf G3J}, {\sf G3J4}, {\sf G3JD}, {\sf G3JD4}, {\sf G3JT}, {\sf G3LP}.

The height of a derivation is the maximum number of successive
applications of rules in it. Given a constant specification $\CS$,
we say that a rule is height-preserving $\CS$-admissible if
whenever an instance of its premise(s) is derivable in $\J_\CS$ with height
$n$, then so is the corresponding instance of its
conclusion. In the rest of the paper, $\varphi$ stands for one of the
formulas $\ww\Vvdash A$, $\ww R\vv$, or $\ww E(t,A)$. We begin
with the following simple observation:
\begin{lemma}\label{lem:Initial axiom for arbitrary A}
 Sequents of the form $\ww\Vvdash A,\g\R\d,\ww\Vvdash A$, with $A$ an arbitrary $\JL$-formula, are derivable in $\J_\CS$.
\end{lemma}
\begin{proof} The proof involves a routine induction on the complexity of
the formula $A$.\qe \end{proof}

\begin{lemma}
The following rule is \CS-admissible in $\J_\CS$:
$$\di{\f{\ww E(s+t,A),\ww E(t+s,A),\ww E(t,A),\g\R\d}{\ww E(t,A),\g\R\d}}(E+)$$
\end{lemma}
\begin{proof} We have the derivation
\begin{prooftree}
\AXC{$\ww E(s+t,A),\ww E(t+s,A),\ww E(t,A),\g\R\d$}
\RightLabel{$(El+)$}
\UIC{$\ww E(t+s,A),\ww E(t,A),\g\R\d$}
\RightLabel{$(Er+)$}
\UIC{$\ww E(t,A),\g\R\d$}
\end{prooftree}
\qe \end{proof}
\begin{lemma}\label{lem:Admissibility of Anti-Mon in G3J5}
The following rule is \CS-admissible in ${\sf J5}_\CS$ and its extensions:
$$\di{\f{\ww E(t,A),\vv E(t,A),\ww R\vv,\g\R\d}{\vv E(t,A),\ww R\vv,\g\R\d}}(Anti\textrm{-}Mon)$$
\end{lemma}
\begin{proof} We have the derivation\\
\scalebox{0.93}{\parbox{\textwidth}{%
\begin{prooftree}
\AXC{$\ww E(t,A),\vv E(t,A),\ww R\vv,\g\R\d$}
\AXC{$\vv\Vvdash t:A,\ww\Vvdash ?t:\neg t:A,\ww E(?t,\neg t:A),\vv E(t,A),\ww R\vv,\g\R\d,\vv\Vvdash t:A$}
\RightLabel{$(L\neg)$}
\UIC{$\vv\Vvdash \neg t:A,\vv\Vvdash t:A,\ww\Vvdash ?t:\neg t:A,\ww E(?t,\neg t:A),\vv E(t,A),\ww R\vv,\g\R\d$}
\RightLabel{$(L:)$}
\UIC{$\vv\Vvdash t:A,\ww\Vvdash ?t:\neg t:A,\ww E(?t,\neg t:A),\vv E(t,A),\ww R\vv,\g\R\d$}
\RightLabel{$(SE)$}
\UIC{$\ww\Vvdash ?t:\neg t:A,\ww E(?t,\neg t:A),\vv E(t,A),\ww R\vv,\g\R\d$}
\RightLabel{$(SE)$}
\UIC{$\ww E(?t,\neg t:A),\vv E(t,A),\ww R\vv,\g\R\d$}
\RightLabel{$(E?)$}
 \insertBetweenHyps{\hskip -25pt}
\BIC{$\vv E(t,A),\ww R\vv,\g\R\d$}
\end{prooftree}}}\\
 with top-sequent in the right branch is derivable by Lemma \ref{lem:Initial axiom for arbitrary A}.\qe \end{proof}

\begin{example}\label{lem:Example 1 in G3J5}
We give a derivation of the sequent $\ww E(t,A),\ww R \vv \R \vv \Vvdash A$ in ${\sf J5}_\CS$:

\begin{prooftree}
\def\extraVskip{3pt}
\AXC{$\vv\Vvdash A,\ww \Vvdash t:A,\ww E(t,A),\ww R \vv \R \vv \Vvdash A$}
   \RightLabel{$(L:)$}
\UIC{$\ww \Vvdash t:A,\ww E(t,A),\ww R \vv \R \vv \Vvdash A$}
  \RightLabel{$(SE)$}
  \UIC{$\ww E(t,A),\ww R \vv \R \vv \Vvdash A$}
\end{prooftree}

where the top-sequent is derivable by Lemma \ref{lem:Initial axiom for arbitrary A}.
\end{example}

\begin{example}\label{lem:Example 2 in G3J5}
We give a derivation of the sequent $\ww \Vvdash \neg t:A \R \ww E(?t,\neg t:A)$ in ${\sf J5}_\CS$:

\noindent
\scalebox{0.92}{\parbox{\textwidth}{%
\begin{prooftree}
\def\extraVskip{3pt}
\AXC{$\ww\Vvdash t:A,\ww E(t,A)\R \ww E(?t,\neg t:A),\ww \Vvdash  t:A $}
  \RightLabel{$(L\neg)$}
\UIC{$\ww\Vvdash t:A,\ww E(t,A),\ww \Vvdash \neg t:A \R \ww E(?t,\neg t:A)$}
  \RightLabel{$(SE)$}
   \UIC{$\ww E(t,A),\ww \Vvdash \neg t:A \R \ww E(?t,\neg t:A)$}
   \AXC{$AxE$}\noLine
\UIC{$\ww E(? t,\neg t:A),\ww \Vvdash \neg t:A \R \ww E(?t,
\neg t:A)$}
  \RightLabel{$(E?)$}
  \BIC{$\ww \Vvdash \neg t:A \R \ww E(?t,\neg t:A)$}
\end{prooftree}}}

wherer the top-sequent in the left branch is derivable by Lemma \ref{lem:Initial axiom for arbitrary A}.
\end{example}

Let $\varphi(\vv/\ww)$ denote the result of simultaneously
substituting label $\vv$ for all occurrences of label $\ww$ in
$\varphi$. For a multiset of labeled formulas $\g$, let $\g(\vv/\ww)=\{\varphi(\vv/\ww)~|~\varphi\in\g\}$. For a derivation ${\mathcal D}$, let ${\mathcal D}(\vv/\ww)$ denote the result of simultaneously
substituting label $\vv$ for all occurrences of label $\ww$ in ${\mathcal D}$. Next, we prove the substitution lemma for labels.
\begin{lemma}[Substitution of Labels]\label{lemma: substitution lemma}
 The rule of substitution
 \[\di{\f{\g\R\d}{\g(\vv/\ww)\R\d(\vv/\ww)}} Subs\]
 is height-preserving $\CS$-admissible in $\J_\CS$.
\end{lemma}
\begin{proof} The proof is by induction on the height $n$ of the derivation
of $\g\R\d$ in $\J_\CS$. If $n=0$, then $\g\R\d$ is an initial sequent,
and so is $\g(\vv/\ww)\R\d(\vv/\ww)$. If $n>0$, then suppose the last rule
of the derivation is $(R)$. If $(R)$ is any rule without label
condition, then use the induction hypothesis and then apply the rule $(R)$.
Now assume $(R)$ is $(R:)$ or $(Ser)$; we only sketch the
proof for $(R:)$,  the case for $(Ser)$ is similar. Suppose $(R:)$ is the last rule:
 \[ \di{\f{\ww R\uu,\ww E(t,A),\g'\R\d',\uu\Vvdash A}{\ww E(t,A),\g'\R\d',\ww\Vvdash
t:A}}(R:)\]
 If the eigenlabel $\uu$ is not $\vv$, then using the induction
hypothesis (substitute $\vv$ for $\ww$) and rule $(R:)$ we obtain a
derivation of height $n$
 \[ \di{\f{\vv R\uu, \vv E(t,A),\g'(\vv/\ww)\R\d'(\vv/\ww), \uu\Vvdash A}{\vv E(t,A),\g'(\vv/\ww)\R\d'(\vv/\ww),\vv\Vvdash
t:A}}(R:).\]
 If $\uu=\vv$, then by the induction
hypothesis (substitute $\vv'$ for $\uu$) we obtain a derivation of
height $n-1$ of
 \[\ww R\vv',\ww E(t,A),\g'\R\d',\vv'\Vvdash A,\]
 where $\vv'$ is a fresh label, i.e. $\vv'\neq \ww$ and
 $\vv'$ does not occur in $\g'\cup\d'$. Then using the induction
hypothesis (substitute $\vv$ for $\ww$) and rule $(R:)$ we obtain a
derivation of height $n$
 \[ \di{\f{\vv R\vv', \vv E(t,A),\g'(\vv/\ww)\R\d'(\vv/\ww), \vv'\Vvdash A}{\vv E(t,A),\g'(\vv/\ww)\R\d'(\vv/\ww),\vv\Vvdash
t:A}}(R:).\]

 Finally, suppose the rule $(R)$ is the axiom necessitation rule
 $(AN)$:
 \[ \di{\f{\uu E(c,A),\g\R\d}{\g\R\d}}(AN),\]
 where 
 $c:A\in\CS$. If $\uu \neq\ww$ then using the induction
hypothesis (substitute $\vv$ for $\ww$) and rule $(AN)$ we obtain a
derivation of height $n$
 \[ \di{\f{\uu E(c,A),\g(\vv/\ww)\R\d(\vv/\ww)}{\g(\vv/\ww)\R\d(\vv/\ww)}}(AN).\]
If $\uu=\ww$ then using the induction hypothesis (substitute $\vv$ for
$\ww$) and rule $(AN)$ we obtain a derivation of height $n$
 \[ \di{\f{\vv E(c,A),\g(\vv/\ww)\R\d(\vv/\ww)}{\g(\vv/\ww)\R\d(\vv/\ww)}}(AN).\]
Proceed similarly
if $(R)$ is $(IAN)$. In the induction step, a cases-by-case
analysis shows that no step (specially $(AN)$ and $(IAN)$)
change the underlying constant specification $\CS$, and hence the rule $Subs$ is
$\CS$-admissible.\qe \end{proof}
\section{Analyticity}\label{sec:Analyticity}
In this section we study the analyticity of our labeled sequent calculi, which is important to prove the termination of proof search. We begin by showing that labeled sequent calculi denoted by \JJ, i.e. {\sf G3J}, {\sf G3J4}, {\sf G3JD}, {\sf G3JD4}, {\sf G3JT}, {\sf G3LP}, enjoy the subformula property.  Let us first define \textit{labeled-subformulas} of a labeled formula.

\begin{definition}\label{def:subformula labeled formula}
Labeled-subformulas of labeled formulas are defined inductively as
follows. The only labeled-subformula of $\ww\Vvdash P$, for propositional variable
$P$, is $\ww\Vvdash P$ itself. The only labeled-subformula of $\ww\Vvdash \bot$ is
$\ww\Vvdash \bot$ itself. The labeled-subformulas of $\ww\Vvdash \neg A$ are
$\ww\Vvdash \neg A$ and all labeled-subformulas of $\ww\Vvdash A$. The labeled-subformulas of $\ww\Vvdash
A\star B$, where $\star$ is a propositional connective,  are $\ww\Vvdash A\star B$ and all labeled-subformulas of $\ww\Vvdash
A$ and $\ww\Vvdash B$. The labeled-subformulas of $\ww\Vvdash t:A$ are $\ww\Vvdash t:A$ and all labeled-subformulas of $\vv\Vvdash A$, for arbitrary label $\vv$.
\end{definition}

Observe that given a derivation in \JJ, if we trace only labeled formulas in the derivation, then they are labeled-subformulas of labeled formulas in the endsequent.
Now an easy inspection of all rules shows that:

\begin{proposition}[Labeled-subformula Property]\label{prop:subformula property}
All labeled formulas in a derivation in $\J^-_\CS$ are labeled-subformulas of labeled formulas in the endsequent.
\end{proposition}

\begin{proposition}[Weak Subformula Property]
All formulas in a derivation in $\J^-_\CS$ are either labeled-subformulas of labeled formulas in the endsequent or atomic formulas of the form $\ww E(t,A)$ or $\ww R\vv$.
\end{proposition}

The following examples show that the labeled-subformula property does not hold in ${\sf G3J5}$, ${\sf G3JB}$ and their extensions.

\begin{example}\label{ex:counterexample subformula G3J5}
Consider the following derivation in ${\sf G3J5}_\emptyset$:
\begin{prooftree}
\def\extraVskip{3pt}
\AXC{$(Ax)$}\noLine
  \UIC{$\vv\Vvdash P,\ww\Vvdash t+s:P,\ww E(t+s,P),\ww R\vv,\ww E(t,P)\R \vv\Vvdash P$}
  \RightLabel{$(L:)$}
  \UIC{$\ww\Vvdash t+s:P,\ww E(t+s,P),\ww R\vv,\ww E(t,P)\R \vv\Vvdash P$}
  \RightLabel{$(SE)$}
   \UIC{$\ww E(t+s,P),\ww R\vv,\ww E(t,P)\R \vv\Vvdash P$}
  \RightLabel{$(Er+)$}
  \UIC{$\ww R\vv,\ww E(t,P)\R \vv\Vvdash P$}
\end{prooftree}
where $P$ is a propositional variable. The labeled-subformula property does not hold in the above derivation since $\ww\Vvdash t+s:P$ is not a labeled-subformula of any labeled formula in the endsequent.
\end{example}

\begin{example}\label{ex:counterexample CS-subformula G3J5}
Let $\CS$ be a constant specification for {\sf J5} which contains $c:(\neg t:s:P\r(Q\r Q))$. Consider the following derivation $\mathcal{D}$ in ${\sf G3J5}_\CS$:\\
\scalebox{0.86}{\parbox{\textwidth}{%
\begin{prooftree}
\def\extraVskip{5pt}
  \AXC{$\mathcal{D}_1$}\noLine
     \UIC{$\ww E(t, s:P), \ww R\vv,\vv R \uu\R \uu\Vvdash P,\ww E(c\cdot \bar{?}t,Q\r Q)$}
    \AXC{$\mathcal{D}_2$}\noLine
      \UIC{$\ww E(?t,\neg t:s:P),\ww R\vv,\vv R \uu\R \uu\Vvdash P,\ww E(c\cdot ?t,Q\r Q)$}
  \RightLabel{$(E?)$}
  \BIC{$\ww R\vv,\vv R \uu\R \uu\Vvdash P, \ww E(c\cdot ?t,Q\r Q)$}
  \end{prooftree}}} \\
  in which the derivation $\mathcal{D}_1$ is as follows:

  \begin{prooftree}
\def\extraVskip{5pt}
  \AXC{$(Ax)$}\noLine
  \UIC{$\uu\Vvdash P,\vv\Vvdash s:P,\ww\Vvdash t:s:P,\ww E(t, s:P),  \ww R\vv,\vv R \uu\R \uu\Vvdash P,\ww E(c\cdot ?t,Q\r Q)$}
  \RightLabel{$(L:)$}
   \UIC{$\vv\Vvdash s:P,\ww\Vvdash t:s:P, \ww E(t, s:P), \ww R\vv,\vv R \uu\R \uu\Vvdash P,\ww E(c\cdot ?t,Q\r Q)$}
  \RightLabel{$(L:)$}
    \UIC{$\ww\Vvdash t:s:P,\ww E(t, s:P), \ww R\vv,\vv R \uu\R \uu\Vvdash P,\ww E(c\cdot ?t,Q\r Q)$}
    \RightLabel{$(SE)$}
    \UIC{$\ww E(t, s:P), \ww R\vv,\vv R \uu\R \uu\Vvdash P,\ww E(c\cdot \bar{?}t,Q\r Q)$}
    \end{prooftree}
  and the derivation $\mathcal{D}_2$ is as follows:\\
    \scalebox{0.93}{\parbox{\textwidth}{%
 \begin{prooftree}
 \AXC{$(AxE)$}\noLine
   \UIC{$\ww E(c\cdot ?t,Q\r Q),\ww E(c,\neg t:s:P\r(Q\r Q)),\ww E(?t,\neg t:s:P),\ww R\vv,\vv R \uu\R \uu\Vvdash P,\ww E(c\cdot ?t,Q\r Q)$}
  \RightLabel{$(E\cdot)$}
  \UIC{$\ww E(c,\neg t:s:P\r(Q\r Q)),\ww E(?t,\neg t:s:P),\ww R\vv,\vv R \uu\R \uu\Vvdash P,\ww E(c\cdot ?t,Q\r Q)$}
  \RightLabel{$(IAN)$}
  \UIC{$\ww E(?t,\neg t:s:P),\ww R\vv,\vv R \uu\R \uu\Vvdash P,\ww E(c\cdot ?t,Q\r Q)$}
    \end{prooftree}}}\\
 The labeled-subformula property does not hold in the above derivation since $\ww\Vvdash t:s:P$ and $\vv\Vvdash s:P$ are not  labeled-subformulas of any labeled formula in the endsequent.
\end{example}

\begin{example}\label{ex:counterexample subformula G3JB}
Consider the following derivation in ${\sf G3JB}_\emptyset$:\\
\scalebox{0.89}{\parbox{\textwidth}{%
\begin{prooftree}
\def\extraVskip{5pt}
  \AXC{$(Ax)$}\noLine
  \UIC{$\vv\Vvdash P,\ww\Vvdash s:P, \ww R\vv\R \vv\Vvdash P,\ww E(\bar{?}t,\neg t:s:P)$}
  \RightLabel{$(L:)$}
    \UIC{$\ww\Vvdash s:P, \ww R\vv\R \vv\Vvdash P,\ww E(\bar{?}t,\neg t:s:P)$}
    \AXC{$(AxE)$}\noLine
  \UIC{$\ww E(\bar{?}t,\neg t:s:P),\ww R\vv\R \vv\Vvdash P,\ww E(\bar{?}t,\neg t:s:P)$}
  \RightLabel{$(E\bar{?})$}
   \insertBetweenHyps{\hskip -8pt}
  \BIC{$\ww R\vv\R \vv\Vvdash P, \ww E(\bar{?}t,\neg t:s:P)$}
  \end{prooftree}}} \\
 The labeled-subformula property does not hold in the above derivation since $\ww\Vvdash s:P$ is not a labeled-subformula of any labeled formula in the endsequent. 
\end{example}

\begin{example}\label{ex:counterexample CS-subformula G3JB}
Let $\CS$ be a constant specification for {\sf JB} which contains $c:(\neg t:s:P\r(Q\r Q))$. Consider the following derivation $\mathcal{D}$ in ${\sf G3JB}_\CS$:\\
\scalebox{0.89}{\parbox{\textwidth}{%
\begin{prooftree}
\def\extraVskip{5pt}
  \AXC{$(Ax)$}\noLine
  \UIC{$\vv\Vvdash P,\ww\Vvdash s:P, \ww R\vv\R \vv\Vvdash P,\ww E(c\cdot \bar{?}t,Q\r Q)$}
  \RightLabel{$(L:)$}
    \UIC{$\ww\Vvdash s:P, \ww R\vv\R \vv\Vvdash P,\ww E(c\cdot \bar{?}t,Q\r Q)$}
    \AXC{$\mathcal{D}'$}\noLine
      \UIC{$\ww E(\bar{?}t,\neg t:s:P),\ww R\vv\R \vv\Vvdash P,\ww E(c\cdot \bar{?}t,Q\r Q)$}
  \RightLabel{$(E\bar{?})$}
  \insertBetweenHyps{\hskip -10pt}
  \BIC{$\ww R\vv\R \vv\Vvdash P, \ww E(c\cdot \bar{?}t,Q\r Q)$}
  \end{prooftree}}} \\
  in which the derivation $\mathcal{D}'$ is as follows:\\
  \scalebox{0.92}{\parbox{\textwidth}{%
 \begin{prooftree}
 \AXC{$(AxE)$}\noLine
   \UIC{$\ww E(c\cdot \bar{?}t,Q\r Q),\ww E(c,\neg t:s:P\r(Q\r Q)),\ww E(\bar{?}t,\neg t:s:P),\ww R\vv\R \vv\Vvdash P,\ww E(c\cdot \bar{?}t,Q\r Q)$}
  \RightLabel{$(E\cdot)$}
  \UIC{$\ww E(c,\neg t:s:P\r(Q\r Q)),\ww E(\bar{?}t,\neg t:s:P),\ww R\vv\R \vv\Vvdash P,\ww E(c\cdot \bar{?}t,Q\r Q)$}
  \RightLabel{$(IAN)$}
  \UIC{$\ww E(\bar{?}t,\neg t:s:P),\ww R\vv\R \vv\Vvdash P,\ww E(c\cdot \bar{?}t,Q\r Q)$}
    \end{prooftree}}}\\
 The labeled-subformula property does not hold in the above derivation since $\ww\Vvdash s:P$ is not a labeled-subformula of any labeled formula in the endsequent. 
\end{example}

Since the regular rule scheme $(Reg)$ and geometric rule scheme $(GRS)$ produce new atoms in their premises, when the rules are read upwardly,  those rules that are instances of $(Reg)$ and $(GRS)$ may violate the analyticity and termination of proof search. For example, in labeled systems {\sf G3JB}, {\sf G3J5}, and their extensions the labeled-subformula property does not hold, as Examples \ref{ex:counterexample subformula G3J5}-\ref{ex:counterexample CS-subformula G3JB} show. The reason is that the rules $(E\bar{?})$ and $(SE)$  are obtained by $(Reg)$ from
\begin{equation}\label{number equation:weak neg intros in labeled language}
\ww \Vvdash A \vee \ww E(\bar{?}t,\neg t:A),
\end{equation}
and
\begin{equation}\label{number equation:strong evidence in labeled language}
\ww E(t,A) \r \ww\Vvdash t:A,
\end{equation}
where (\ref{number equation:weak neg intros in labeled language}) and (\ref{number equation:strong evidence in labeled language}) are formal expressions of weak negative introspection $(\E5)$ and strong evidence $(\E 7)$ conditions in the extended labeled language. Moreover, in contrast to other conditions on evidence function, conditions $(\E5)$ and $(\E7)$ involve forcing assertion. In other words, both (\ref{number equation:weak neg intros in labeled language}) and (\ref{number equation:strong evidence in labeled language}) contain labeled formulas, and these labeled formulas are appeared only in the premise(s) of the rules $(E\bar{?})$ and $(SE)$.

Although $\J^-$ enjoy the labeled-subformula property, since the rules $(E\cdot)$, $(El+)$, $(Er+)$, $(E!)$, $(IAN)/(AN)$ produce new justification terms in evidence atoms in their premises, when the rules are read upwardly, they may violate the analyticity. Since some rules, like $(Ref)$, $(Ser)$, $(E?)$, produce (new) labels in their premise(s), when the rules are read upwardly, they may also violate the analyticity. In the following we shall show the subterm and sublabel property which are needed to prove the analyticity for $\J^-$. First, similar to \cite{DyckhoffNegri2012}, we show the sublabel property for all our labeled sequent calculi  \J.

\begin{definition}
A rule instance has the sublabel property if every label occurring in any premise of the rule is either an eigenlabel or occurs in the conclusion. A derivation has the sublabel property if all rule instances occurring in it has the sublabel property.
\end{definition}

In a derivation with sublabel property all labels are either an eigenlabel or labels in the endsequent. This property is called the \textit{sublabel property}.\footnote{It is called the subterm property in \cite{Negri2005}, and does not to be confused with the subterm property stated in Proposition \ref{prop:subterm property}.} Instances of the rules $(Ref)$, $(Ser)$, $(E\bar{?})$, $(E?)$, $(IAN)/(AN)$ in derivations may produce a derivation which does not have the sublabel property. In the following we try to show the sublabel property. First we need a lemma.

\begin{lemma}\label{lem:substitution for analytic derivations}
Suppose ${\mathcal D}$ is a derivation with the sublabel property for $\g\R\d$  in $\J_\CS$, and labels $\ww,\vv$ occur in ${\mathcal D}$ such that $\ww$ occurs in $\g\cup\d$ and $\vv$ is not an eigenlabel. Then the derivation ${\mathcal D}(\ww/\vv)$ is a derivation with the sublabel property for $\g(\ww/\vv)\R\d(\ww/\vv)$ in $\J_\CS$.
\end{lemma}
\begin{proof} The proof involves a routine induction on the height of the derivation ${\mathcal D}$ of $\g\R\d$ in $\J_\CS$.
\qe \end{proof}

\begin{proposition}[Sublabel Property]\label{prop:sublabel property}
Every sequent derivable in $\J_\CS$  has a derivation with the sublabel property; in other words, every derivable sequent $\g\R\d$ has a derivation in which all labels are eigenlabels or labels in $\g\cup\d$.
\end{proposition}
\begin{proof} By induction on the height of the derivation we transform every derivation into  one with the sublabel property.
The base case is trivial, since every initial sequent has the sublabel property. For the induction step, suppose $(R)$ is a topmost rule which does not have the sublabel property.
For each label, say $\vv$, in any premise of $(R)$ that is not in the conclusion and is not an eigenlabel, we substitute for it (in the derivations of all the premises) any label, say $\ww$, that is in the endsequent. Note that in any derivation in $\J_\CS$ all labels in a rule's conclusion are in all its premises. Thus $\ww$ is already in the conclusion of each premise, and therefore by Lemma \ref{lem:substitution for analytic derivations} the derivations of the premise(s) have still the sublabel property. Moreover, the conclusion of $(R)$ is unchanged under this substitution, and the resulting rule instance has now the sublabel property. 
By repeating the above argument for all rules which do not have the sublabel property we finally obtain a derivation with the sublabel property.\qe \end{proof}

\begin{example}
Here is a derivation of $\ww \Vvdash x:P\R \ww\Vvdash P$ in {\sf G3JT} which does not have the sublabel property:
\begin{prooftree}
\def\extraVskip{3pt}
\AXC{$(Ax)$}\noLine
\UIC{$\ww\Vvdash P,\vv R\vv,\ww R\ww,\ww \Vvdash x:P\R \ww\Vvdash P$}
  \RightLabel{$(L:)$}
 \UIC{$\vv R\vv,\ww R\ww,\ww \Vvdash x:P\R \ww\Vvdash P$}
  \RightLabel{$(Ref)$}
  \UIC{$\ww R\ww,\ww \Vvdash x:P\R \ww\Vvdash P$}
  \RightLabel{$(Ref)$}
   \UIC{$\ww \Vvdash x:P\R \ww\Vvdash P$}
  \end{prooftree}
The topmost instance of rule $(Ref)$ does not have the sublabel property, because the label $\vv$ is not in the conclusion of the rule. By the substitution described in the proof of Proposition \ref{prop:sublabel property}, substitute $\ww$ for $\vv$, the derivation is transformed to one with the sublabel property:
\begin{prooftree}
\def\extraVskip{3pt}
\AXC{$(Ax)$}\noLine
\UIC{$\ww\Vvdash P,\ww R\ww,\ww R\ww,\ww \Vvdash x:P\R \ww\Vvdash P$}
  \RightLabel{$(L:)$}
 \UIC{$\ww R\ww,\ww R\ww,\ww \Vvdash x:P\R \ww\Vvdash P$}
  \RightLabel{$(Ref)$}
  \UIC{$\ww R\ww,\ww \Vvdash x:P\R \ww\Vvdash P$}
  \RightLabel{$(Ref)$}
   \UIC{$\ww \Vvdash x:P\R \ww\Vvdash P$}
  \end{prooftree}
\end{example}

Now we prove the subterm property for {\sf G3J}, {\sf G3J4}, {\sf G3JD}, {\sf G3JD4}, {\sf G3JT}, {\sf G3LP}.
We use the following notations in the rest of this paper. For an arbitrary sequent $\g\R\d$ in the language of \J, and constant specification \CS~for \JL, let
\begin{itemize}
\item $Tm(\g\R\d)$ denote the set of all terms which occur in a labeled formula or an evidence atom  in $\g\cup\d$,
\item $Fm(\g\R\d)$ denote the set of all \JL-formulas which occur in a labeled formula or an evidence atom  in $\g\cup\d$,
\item $Sub_{Tm}(\g\R\d)$ denote the set of all subterms of the terms from $Tm(\g\R\d)$,
\item $Sub_{Fm}(\g\R\d)$ denote the set of all \JL-subformulas of the formulas from $Fm(\g\R\d)$,
\end{itemize}

\begin{definition}\label{def:thread-family-subterm}
By \textit{E-rules} we mean the following rules: $(E\cdot)$, $(El+)$, $(Er+)$, $(E!)$, $(IAN)/(AN)$.
\begin{enumerate}
\item {\bf $(R)$-thread}. Suppose $(R)$ is an instance of an E-rule 
 in a derivation $\mathcal{D}$ in $\J^-_\CS$. A sequence of sequents in $\mathcal{D}$ is called an $(R)$-thread if the sequence begins with an initial sequent and ends with the premise of $(R)$, and every sequent in the sequence (except the last one) is the premise of a rule of $\J^-_\CS$, and is immediately followed by the conclusion of this rule.
\item {\bf Related evidence atoms}. Related evidence atoms in a derivation in $\J^-_\CS$ are defined as follows:
\begin{enumerate}
\item Corresponding evidence atoms in side formulas in premise(s) and conclusion of rules are related.
\item Active and principal evidence atoms in the rules $(Mon)$, $(R:)$, $(El+)$, $(Er+)$, $(E!)$ are related.
\item In an instance of rule $(E\cdot)$
$$\di{\f{\ww E(s\cdot t,B),\ww E(s,A\r B),\ww E(t,A),\g\R\d}{\ww E(s,A\r B),\ww  E(t,A),\g\R\d}}(E\cdot)$$
evidence atoms $\ww E(s\cdot t,B)$ and $\ww E(t,A)$ in the premise are related to $\ww E(t,A)$ in the conclusion, and evidence atoms $\ww E(s\cdot t,B)$ and $\ww E(s,A\r B)$ in the premise are related to $\ww E(s,A\r B)$ in the conclusion of the rule.
\item The relation `related evidence atoms', defined in the above clauses, is extended by transitivity.
\end{enumerate}
\item {\bf Family of evidence atoms}. Suppose $(R)$ is an instance of an E-rule  in a derivation $\mathcal{D}$ in $\J^-_\CS$, and e is an evidence atom in the premise of $(R)$. The set of all evidence atoms related to e in all $(R)$-threads in $\mathcal{D}$ is called the family of e in $\mathcal{D}$.
\item {\bf Subterm property of E-rules}.
\begin{enumerate}
\item An instance of  rule $(El+)$ with active evidence atoms $\ww E(s+t,A)$, $\ww E(t,A)$ in a derivation with endsequent $\g\R\d$ has the subterm property if $s+t\in Sub_{Tm}(\g\R\d)$.
\item An instance of rule $(Er+)$ with active evidence atoms $\ww E(t+s,A)$, $\ww E(t,A)$ in a derivation with endsequent $\g\R\d$ has the subterm property if  $t+s\in Sub_{Tm}(\g\R\d)$.
\item An instance of  rule $(E\cdot)$ with active evidence atoms $\ww E(s\cdot t,B)$, $\ww E(s,A\r B)$, $\ww E(t,A)$ in a derivation with endsequent $\g\R\d$ has the subterm property if  $s\cdot t\in Sub_{Tm}(\g\R\d)$.
\item An instance of  rule $(E!)$ with active evidence atoms $\ww E(!t,t:A)$, $\ww E(t,A)$ in a derivation with endsequent $\g\R\d$ has the subterm property if  $!t\in Sub_{Tm}(\g\R\d)$.
\item An instance of  rule $(AN)$ with active evidence atom $\ww E(c,A)$ in a derivation with endsequent $\g\R\d$ has the subterm property if  $c\in Sub_{Tm}(\g\R\d)$.
\item An instance of rule $(IAN)$ with active evidence atom $\ww E(c_{i_n},c_{i_{n-1}}:\ldots:c_{i_1}:A)$ in a derivation with endsequent $\g\R\d$ has the subterm property if  $c_{i_n},c_{i_{n-1}},\ldots,c_{i_1}\in Sub_{Tm}(\g\R\d)$.
\end{enumerate}

\item {\bf Superfluous applications of rules}.
\begin{enumerate}
\item An application of  rule $(El+)$ in a derivation  with active evidence atoms $\ww E(s+t,A)$, $\ww E(t,A)$ is superfluous if no evidence atom in the family of $\ww E(s+t,A)$ is principal in an initial sequent $(AxE)$ or in a rule $(R:)$.
\item An application of rule $(Er+)$ in a derivation with active evidence atoms $\ww E(t+s,A)$, $\ww E(t,A)$ is superfluous if no evidence atom in the family of $\ww E(t+s,A)$ is principal in an initial sequent $(AxE)$ or in a rule $(R:)$.
\item An application of  rule $(E\cdot)$ in a derivation with active evidence atoms $\ww E(s\cdot t,B)$, $\ww E(s,A\r B)$, $\ww E(t,A)$ is superfluous if no evidence atom in the family of $\ww E(s\cdot t,B)$ is principal in an initial sequent $(AxE)$ or in a rule $(R:)$.
\item An application of  rule $(E!)$ in a derivation with active evidence atoms $\ww E(!t,t:A)$, $\ww E(t,A)$ is superfluous if no evidence atom in the family of $\ww E(!t,t:A)$ is principal in an initial sequent $(AxE)$ or in a rule $(R:)$.
\item An application of  rule $(AN)$ in a derivation with active evidence atom $\ww E(c,A)$ is superfluous if no evidence atom in the family of $\ww E(c,A)$ is principal in an initial sequent $(AxE)$ or in a rule $(R:)$.
\item An application of rule $(IAN)$ in a derivation with active evidence atom $\ww E(c_{i_n},c_{i_{n-1}}:\ldots:c_{i_1}:A)$ is superfluous if no evidence atom in the family of $\ww E(c_{i_n},c_{i_{n-1}}:\ldots:c_{i_1}:A)$ is principal in an initial sequent $(AxE)$ or in a rule $(R:)$.
\item  An application of  rule $(Mon)$ in a derivation with active formulas $\vv E(t,A)$, $\ww E(t,A)$, $\ww R \vv$ is superfluous if no evidence atom in the family of $\vv E(t,A)$ is principal in an initial sequent $(AxE)$ or in a rule $(R:)$.
\end{enumerate}
\end{enumerate}
\end{definition}

 It is easy to verify that for an evidence atom $\ww E(t,A)$ in the premise of an E-rule 
in a derivation in $\J^-_\CS$, the typical form of evidence atoms in its family is $\vv E(r,B)$, where $t$ is a subterm of $r$, $\vv$ is a label,  and $B$ is a formula.

\begin{lemma}\label{lem:subterm property superfluous}
In any derivation in $\J^-_\CS$, every application of an E-rule which does not have the subterm property is superfluous.
\end{lemma}
\begin{proof}

We detail the proof only in the case of rule $(El+)$, the cases of other E-rules are handled in a similar way. Consider an application of $(El+)$ in a derivation $\mathcal{D}$:
$$
\di{\f{\ww E(s+t,A),\ww E(t,A),\g\R\d}{\ww E(t,A),\g\R\d}}(El+)
$$
which does not have the subterm property. Consider the family of the evidence atom $\ww E(s+t,A)$ (in the premise of $(El+)$) in $\mathcal{D}$. We show that the following possibilities cannot happen, and therefore this application of $(El+)$ is superfluous.
\begin{enumerate}
\item An evidence atom $\vv E(r,B)$ in the family of $\ww E(s+t,A)$ is principal in an initial sequent $(AxE)$.
\item An evidence atom $\vv E(r,B)$ in the family of $\ww E(s+t,A)$ is principal in a rule $(R:)$.
\end{enumerate}

In the first case, the initial sequent is of the form $\vv E(r,B),\g'\R\d',\vv E(r,B)$. Since the rules of $\J^-_\CS$ do not remove any evidence atom from succedent of sequents, the evidence atom $\vv E(r,B)$ would appear in the endsequent. Note that $s+t$ is a subterm of $r$. Therefore, $s+t$ occurs as the subterm of $r$ in the endsequent, which would contradict to the assumption that $(El+)$ does not have the subterm property. Thus, this case cannot happen.

In the second case, the rule $(R:)$ is of the form:
$$\di{\f{\vv R\uu,\vv E(r,B),\g'\R\d',\uu\Vvdash B}{\vv E(r,B),\g'\R\d',\vv\Vvdash r:B}}(R:) $$
By the labeled-subformula property (Proposition \ref{prop:subformula property}) for $\J^-_\CS$, the labeled formula $\vv\Vvdash r:B$ is a labeled-subformula of a labeled formula in the endsequent. Therefore, $s+t$ occurs as the subterm of $r$ in the endsequent, which would contradict to the assumption that $(El+)$ does not have the subterm property. Thus, this case cannot happen.\qe
\end{proof}
Note that the converse of the above lemma does not hold, that means a superfluous application of an E-rule can have the subterm property.

\begin{lemma}\label{lem:subterm property E-rules}
Every sequent derivable in $\J^-_\CS$ has a derivation in which all instances of E-rules  has the subterm property.
\end{lemma}
\begin{proof}
Suppose $\mathcal{D}$ is a derivation of a sequent in $\J^-_\CS$. If all applications of E-rules in $\mathcal{D}$ has the subterm property, then $\mathcal{D}$ is the desired derivation. Otherwise, consider the bottommost application of an E-rule $(R)$ in $\mathcal{D}$ which does not have the subterm property, and let $e$ denote the active evidence atom of $(R)$ that does not occur in the conclusion of $(R)$. By Lemma \ref{lem:subterm property superfluous}, this application of $(R)$ is superfluous. In the following we show how to remove this application of the rule $(R)$ from the derivation.

Find the family of $e$ in the derivation, and all rules which have an active evidence atom from this family. Remove all occurrences of evidence atoms in the family of $e$ from the derivation, and also remove all rules which have these evidence atoms as active formulas (by removing a rule we mean removing the premise of that rule from the derivation). Note that, since rule $(R)$ is superfluous, evidence atoms in the family of $e$ can only be active in $(Mon)$ and E-rules, and they are side formulas in the other rules of the derivation.  Thus, removing these evidence atoms only affect on the rule $(Mon)$ and E-rules, and do not affect the validity of the remaining rules in the derivation. It is easy to see that for a superfluous application of $(R)$ with active evidence atom $e$ in a derivation, all applications of the rule $(Mon)$ and E-rules that are above $(R)$ and have an active evidence atom from the family of $e$ are superfluous. Therefore, the result of removing those rules which has an active evidence atom from the family of $e$ is still a derivation.  Note that, the superfluous application of $(R)$ is also removed from the derivation, and this new derivation produces no new superfluous application of $(R)$ (or other E-rules). By repeating the above argument for all bottommost superfluous applications\footnote{In fact removing the bottommost superfluous application of $(R)$, instead of an arbitrary one, will remove all other superfluous applications of E-rules that are above it and have an active evidence atom from the family of $e$.} of E-rules in the derivation, we finally find a derivation in which all instances of E-rules have the subterm property, and moreover this procedure will finally terminate.\qe
\end{proof}

\begin{example}
Here is an example of a derivation of $\ww \Vvdash x:P\R \ww\Vvdash (x+y):P$ in {\sf G3J4} with a superfluous application of rule $(E!)$:\\
\scalebox{0.91}{\parbox{\textwidth}{%
\begin{prooftree}
\def\extraVskip{3pt}
\AXC{$(Ax)$}\noLine
\UIC{$\vv\Vvdash P,\vv E(!(x+y),x+y:P),\ww R \vv,\ww E(!(x+y),x+y:P),\ww E(x+y,P),\ww E(x,P),\ww \Vvdash x:P\R \vv\Vvdash P$}
  \RightLabel{$(L:)$}
  \UIC{$\vv E(!(x+y),x+y:P),\ww R \vv,\ww E(!(x+y),x+y:P),\ww E(x+y,P),\ww E(x,P),\ww \Vvdash x:P\R \vv\Vvdash P$}
  \RightLabel{$(Mon)$}
\UIC{$\ww R \vv,\ww E(!(x+y),x+y:P),\ww E(x+y,P),\ww E(x,P),\ww \Vvdash x:P\R \vv\Vvdash P$}
  \RightLabel{$(R:)$}
\UIC{$\ww E(!(x+y),x+y:P),\ww E(x+y,P),\ww E(x,P),\ww \Vvdash x:P\R \ww\Vvdash (x+y):P$}
  \RightLabel{$(E!)$}
  \UIC{$\ww E(x+y,P),\ww E(x,P),\ww \Vvdash x:P\R \ww\Vvdash (x+y):P$}
  \RightLabel{$(Er+)$}
  \UIC{$\ww E(x,P),\ww \Vvdash x:P\R \ww\Vvdash (x+y):P$}
  \RightLabel{$(E)$}
   \UIC{$\ww \Vvdash x:P\R \ww\Vvdash (x+y):P$}
  \end{prooftree}}} \\
where $P$ is a propositional variable, and $x,y$ are justification variables. It is clear that the application of rule $(Er+)$ in the derivation is not superfluous (because the evidence atom $\ww E(x+y,P)$ in the family of its active evidence atom $\ww E(x+y,P)$ is principal in the rule $(R:)$), but the application of rule $(E!)$ in the derivation is superfluous. In fact, the rule $(E!)$ does not have the subterm property and the term $!(x+y)$ does not occur in the endsequent. To transform this derivation into one with the subterm property, we first find the family of evidence atoms of $\ww E(!(x+y),x+y:P)$, which are indicated in boxes in the following derivation:
\\
\scalebox{0.87}{\parbox{\textwidth}{%
\begin{prooftree}
\def\extraVskip{3pt}
\AXC{$(Ax)$}\noLine
\UIC{$\vv\Vvdash P,\textit{\fbox{$\vv E(!(x+y),x+y:P)$}}~,\ww R \vv,\textit{\fbox{$\ww E(!(x+y),x+y:P)$}}~,\ww E(x+y,P),\ww E(x,P),\ww \Vvdash x:P\R \vv\Vvdash P$}
  \RightLabel{$(L:)$}
  \UIC{$\textit{\fbox{$\vv E(!(x+y),x+y:P)$}}~,\ww R \vv,\textit{\fbox{$\ww E(!(x+y),x+y:P)$}}~,\ww E(x+y,P),\ww E(x,P),\ww \Vvdash x:P\R \vv\Vvdash P$}
  \RightLabel{$\textit{\fbox{(Mon)}}$}
\UIC{$\ww R \vv,\textit{\fbox{$\ww E(!(x+y),x+y:P)$}}~,\ww E(x+y,P),\ww E(x,P),\ww \Vvdash x:P\R \vv\Vvdash P$}
  \RightLabel{$(R:)$}
\UIC{$\textit{\fbox{$\ww E(!(x+y),x+y:P)$}}~,\ww E(x+y,P),\ww E(x,P),\ww \Vvdash x:P\R \ww\Vvdash (x+y):P$}
  \RightLabel{$\textit{\fbox{(E!)}}$}
  \UIC{$\ww E(x+y,P),\ww E(x,P),\ww \Vvdash x:P\R \ww\Vvdash (x+y):P$}
  \RightLabel{$(Er+)$}
  \UIC{$\ww E(x,P),\ww \Vvdash x:P\R \ww\Vvdash (x+y):P$}
  \RightLabel{$(E)$}
   \UIC{$\ww \Vvdash x:P\R \ww\Vvdash (x+y):P$}
  \end{prooftree}}} \\
Then we find those rules which have an active evidence atom from this family, that are rules $(E!)$ and $(Mon)$ in the derivation. Finally, by removing the evidence atoms in boxes and the rules $(E!)$ and $(Mon)$ from the derivation we obtain a derivation of the same endsequent with the subterm property:
\begin{prooftree}
\def\extraVskip{3pt}
\AXC{$(Ax)$}\noLine
\UIC{$\vv\Vvdash P,\ww R \vv,\ww E(x+y,P),\ww E(x,P),\ww \Vvdash x:P\R \vv\Vvdash P$}
  \RightLabel{$(L:)$}
 \UIC{$\ww R \vv,\ww E(x+y,P),\ww E(x,P),\ww \Vvdash x:P\R \vv\Vvdash P$}
  \RightLabel{$(R:)$}
  \UIC{$\ww E(x+y,P),\ww E(x,P),\ww \Vvdash x:P\R \ww\Vvdash (x+y):P$}
  \RightLabel{$(Er+)$}
  \UIC{$\ww E(x,P),\ww \Vvdash x:P\R \ww\Vvdash (x+y):P$}
  \RightLabel{$(E)$}
   \UIC{$\ww \Vvdash x:P\R \ww\Vvdash (x+y):P$}
  \end{prooftree}
\end{example}

Now the subterm property follows from the labeled-subformula property and Lemma \ref{lem:subterm property E-rules}.
\begin{proposition}[Subterm Property]\label{prop:subterm property}
Every sequent $\g\R\d$ derivable in $\J^-_\CS$ has a derivation in which all terms in the derivation are in $Sub_{Tm}(\g\R\d)$.
\end{proposition}
\begin{definition}
A derivation is called analytic if all labeled formulas in the derivation are labeled-subformulas of labeled formulas in the endsequent, all terms in the derivation are subterms of terms in the endsequent, and all labels are eigenlabels or labels in the endsequent.
\end{definition}
Now from Propositions \ref{prop:subformula property},  \ref{prop:sublabel property}, \ref{prop:subterm property} it follows that:
\begin{corollary}[Analyticity]\label{cor:analyticity}
Every sequent $\g\R\d$ derivable in $\J^-_\CS$  has an anlytic derivation.
\end{corollary}
The analyticity will be used in the proof search procedure described in the proof of Theorem \ref{thm:reduction tree} (in fact we search an analytic derivation for a derivable sequent).

\section{Structural properties}\label{sec: Admissibility of structural rules}
In this section, we show that all structural rules (weakening and contraction) and cut are admissible in \J. All proofs are similar to those of
modal logic \cite{Negri2005,Negri Plato2001} adapted for justification logics, and
so details are omitted safely. Again in this section, $\varphi$ stands for one of the
formulas $\ww\Vvdash A$, $\ww R\vv$, or $\ww E(t,A)$.

\begin{theorem}\label{thm:admiss weak.}
 The rules of weakening
\begin{eqnarray*}
&\di{\f{\g\R\d}{ \varphi,\g\R\d}}(LW)&\hspace{2cm}
\di{\f{\g\R\d}{\g\R\d, \varphi}}(RW)
\end{eqnarray*}
are height-preserving $\CS$-admissible in $\J_\CS$ .
\end{theorem}
\begin{proof} By induction on the height $n$ of the derivation of the
premise. If $n=0$, then $\g\R\d$ is an initial sequent, and so are
$\varphi,\g\R\d$ and $\g\R\d, \varphi$. If $n>0$, then suppose the
last rule of the derivation is $(R)$. If $(R)$ is any rule without
label condition or $\varphi$ does
not contain any eigenlabel of $(R)$, then use the induction hypothesis
and apply rule $(R)$. Now suppose $(R)$ is $(R:)$ or $(Ser)$, and $\varphi$ contains the eigenlabel of $(R)$. We only
sketch the proof for $(R:)$
 \[ \di{\f{\ww R\vv,\ww E(t,A),\g'\R\d',\vv\Vvdash A}{\ww E(t,A),\g'\R\d',\ww\Vvdash
t:A}}(R:)\]
 By height-preserving substitution (substitute $\uu$ for $\vv$) we obtain a
derivation of height $n-1$ of
\[ \ww R\uu,\ww E(t,A),\g'\R\d',\uu\Vvdash A,\]
 where $\uu\neq \ww$, $\uu$ does not occur in $\g'\cup\d'$, and $\uu$ does not occur in
$\varphi$. Then by the induction hypothesis we obtain a derivation
of height $n-1$ of
\[ \varphi, \ww R\uu,\ww E(t,A),\g'\R\d',\uu\Vvdash A,\] or of
\[ \ww R\uu,\ww E(t,A),\g'\R\d',\uu\Vvdash A,\varphi.\]
 Then by $(R:)$ we obtain a
derivation of height $n$ of
\[\varphi,\ww E(t,A),\g'\R\d',\ww\Vvdash t:A,\] or of
\[\ww E(t,A),\g'\R\d',\ww\Vvdash t:A,\varphi.\]
 The case for  $(Ser)$ is similar.\qe \end{proof}

 A rule is said to be height-preserving $\CS$-\textit{invertible}
 if whenever an instance of its conclusion is derivable in $\J_\CS$ with height
$n$, then so is the corresponding instance of its premise(s).
\begin{proposition}\label{prop:inversion lemma}
 All the rules of $\J_\CS$ are height-preserving $\CS$-invertible.
\end{proposition}
\begin{proof} The invertibility of propositional rules is proved similar to
that in the ordinary sequent calculus {\sf G3c} in \cite{TS}. The
invertibility of other rules, except rule $(R:)$, are obtained by admissible weakening. We establish
the height-preserving invertibility of $(R:)$, by induction on the
height $n$ of the derivation $\ww E(t,A),\g\R\d,\ww\Vvdash t:A$. If
$n=0$, then $\ww E(t,A),\g\R\d,\ww\Vvdash t:A$, and consequently
$\ww R\vv,\ww E(t,A),\g\R\d,\vv\Vvdash A$ are initial sequents (for any
fresh label $v$). If $n>0$, then suppose the last rule of the
derivation is $(R)$. If $(R)$ is any rule without label condition, then
use the induction hypothesis and apply rule $(R)$. For example, suppose the last rule is $(AN)$
\[\di{\f{\uu E(c,B),\ww E(t,A),\g\R\d,\ww\Vvdash t:A}{\ww E(t,A),\g\R\d,\ww\Vvdash t:A}\,(AN)}\]
 where $c:B\in\CS$, and possibly $\uu=\ww$. By the induction hypothesis we obtain a derivation of height $n-1$ of
\[\uu E(c,B),\ww R\vv,\ww E(t,A),\g\R\d,\vv\Vvdash A,\]
 for any fresh label $v$. Then by applying the rule $(AN)$ we obtain a
derivation of height $n$ of
\[\ww R\vv,\ww E(t,A),\g\R\d,\vv\Vvdash A.\]
The case for $(IAN)$ is similar. We now
check the case in which $\ww E(t,A),\g\R\d,\ww\Vvdash t:A$ is the
conclusion of rule $(Ser)$.
\[\di{\f{\uu R\vv,\ww E(t,A),\g\R\d,\ww\Vvdash t:A}{\ww E(t,A),\g\R\d,\ww\Vvdash t:A}\,(Ser)}\]
where the eigenlabel $\vv$ is not in the conclusion. By the
induction hypothesis we obtain a derivation of height $n-1$ of
\[\ww R\ww',\uu R\vv,\ww E(t,A),\g\R\d,\ww'\Vvdash A,\]
 for any fresh
label $\ww'$ (specially $\ww'\neq \vv$). Then by applying the rule $(Ser)$ we obtain a
derivation of height $n$ of
 \[\ww R\ww',\ww E(t,A),\g\R\d,\ww'\Vvdash A,\]
 as desire.
 \qe \end{proof}

In order to show the admissibility of contraction we need to show that contracted instances of $(Trans)$ are admissible (the proof is similar to the proof
of Proposition 3 in \cite{HakliNegri2011}).
\begin{lemma}\label{lemma: admissibility of Trans*}
The rule
\[\di{\f{\ww R\ww,\ww R\ww,\g\R\d}{\ww R\ww,\g\R\d}\,(Trans_*)}\]
is height-preserving $\CS$-admissible in all labeled systems $\J_\CS$
which contain $(Trans)$.
\end{lemma}
\begin{proof} The rule $(Trans_*)$ is obviously admissible in {\sf G3LP}, {\sf G3T45}, {\sf G3TB45}, {\sf G3TB4} due
to the fact that they contain rule $(Ref)$. Now by induction on
the height of the derivation of the premise of $(Trans_*)$,  $\ww R\ww,\ww R\ww,\g\R\d$, we
show that $(Trans_*)$ is admissible in {\sf G3J4}, {\sf G3JD4}, {\sf G3JB4},
{\sf G3J45}, {\sf G3JB45}, {\sf G3JD45}, {\sf G3JDB45}, and  {\sf G3JDB4}.

If the premise is an initial sequent, then so is the conclusion.
If the premise is obtained by a rule $(R)$, we have two cases:
(i) None of the occurrences of $wRw$ is principal in $(R)$.
For example, suppose $(R)$ is a one premise rule and consider the following derivation of height $n$
\[\di{\f{
\begin{array}{c}
{\mathcal D}\\
\ww R\ww,\ww R\ww,\g'\R\d'
\end{array}
}{\ww R\ww,\ww R\ww,\g\R\d}\,(R)}\]
 By the induction hypothesis we have a derivation of
height $n-1$ of $wRw, \g'\R\d'$. Then by applying the rule $(R)$
we obtain a derivation of height $n$ of $\ww R\ww, \g\R\d$. (ii) One of
the occurrences of $\ww R\ww$ is principal in $(R)$. Then $(R)$ may be
$(Trans)$, $(Sym)$, $(L:)$, or $(Mon)$. Suppose $(R)$ is $(Trans)$ and consider the
following derivation of height $n$
\[\di{\f{
\begin{array}{c}
{\mathcal D}\\
\ww R\ww, \ww R\ww,\ww R\ww,\g\R\d
\end{array}
}{\ww R\ww,\ww R\ww,\g\R\d}\,(Trans)}\]
 Then by the induction hypothesis (applied twice) we get a
derivation of height $n-1$ of $\ww R\ww,\g\R\d$. The case for $(Sym)$ is similar. Suppose $(R)$ is $(L:)$ and
consider the following derivation of height $n$
\[\di{\f{
\begin{array}{c}
{\mathcal D}\\
\ww\Vvdash A,\ww\Vvdash t:A, \ww R\ww,\ww R\ww,\g\R\d
\end{array}
}{\ww\Vvdash t:A,\ww R\ww,\ww R\ww,\g\R\d}\,(L:)}\]
 By the induction hypothesis we get a
derivation of height $n-1$ of $\ww\Vvdash A,\ww\Vvdash t:A,
\ww R\ww,\g\R\d$. Then by applying $(L:)$ we obtain a derivation of
height $n$ of $\ww\Vvdash t:A, \ww R\ww,\g\R\d$. Suppose $(R)$ is $(Mon)$
and consider the following derivation of height $n$
\[\di{\f{
\begin{array}{c}
{\mathcal D}\\
\ww E(t,A), \ww E(t,A), \ww R\ww,\ww R\ww,\g\R\d
\end{array}
}{\ww E(t,A),\ww R\ww,\ww R\ww,\g\R\d}\,(Mon)}\]
 By the induction hypothesis we get a
derivation of height $n-1$ of $\ww E(t,A), \ww E(t,A),
\ww R\ww,\g\R\d$. Then by applying $(Mon)$ we obtain a derivation of
height $n$ of  $ \ww E(t,A), \ww R\ww,\g\R\d$. \qe \end{proof}

\begin{theorem}\label{thm:admiss Contr.}
 The rules of contraction
\[\di{\f{\varphi,\varphi,\g\R\d}{ \varphi,\g\R\d}}(LC)\hspace{2cm} \di{\f{\g\R\d,\varphi,\varphi}{\g\R\d, \varphi}}(RC)\]
are height-preserving $\CS$-admissible in $\J_\CS$.
\end{theorem}
\begin{proof} The proof is by simultaneous induction on the height of
derivation for left and right contraction. We consider in detail
only the case in which contraction formula is principal and the
last rule is $(R:)$. Consider the following derivation of height
$n$
\[\di{\f{
\begin{array}{c}
{\mathcal D}\\
\ww E(t,A),\ww R\vv,\g\R\d,\ww\Vvdash t:A,\vv\Vvdash A
\end{array}
}{\ww E(t,A),\g\R\d,\ww\Vvdash t:A,\ww\Vvdash t:A}\,(R:)}. \] By
Proposition \ref{prop:inversion lemma}, we have a derivation of
height $n-1$ of
 \[ \ww E(t,A) ,\ww R\vv,\ww R\uu,\g\R\d,\uu\Vvdash A,\vv\Vvdash A,\]
 for any fresh
label $\uu$. Then by height-preserving substitution, we obtain a
derivation of height $n-1$ of
\[ \ww E(t,A),\ww R\vv,\ww R\vv,\g\R\d,\vv\Vvdash A,\vv\Vvdash A.\]
 By the induction hypothesis  for left and right contraction, we get $\ww E(t,A),\ww R\vv,\g\R\d,\vv\Vvdash A$. Finally, by applying the rule
$(R:)$, we obtain a derivation of height $n$ for $\ww E(t,A),\g\R\d,\ww\Vvdash t:A$.

If the last rule is $(Trans)$ and $\ww=\vv=\uu$ then
use Lemma \ref{lemma: admissibility of Trans*}. The other cases of the
proof is similar to the proof of Theorem 4.12 in \cite{Negri2005},
and is omitted here.\qe \end{proof}

\begin{theorem}\label{thm:cut elimination}
 The $Cut$ rule
\[\di{\f{\g\R\d,\varphi\qquad \varphi,\g'\R\d'}{\g,\g'\R\d,\d'}\,Cut}\]
 is \CS-admissible in $\J_\CS$.
\end{theorem}
\begin{proof} 
The proof is by induction on the size of cut formula $\varphi$
with a subinduction on the level of $Cut$. The \textit{size} of
$\ww\Vvdash A$ is defined as the size of the formula $A$, and formulas $\ww E(t,A)$ and $\ww R\vv$ is considered to be of size
1. The \textit{level} of a $Cut$ is the sum of the heights of the
derivations of the premises. Let us first consider the cases that
the cut formula is of the form $\ww\Vvdash t:A$. Suppose $\ww\Vvdash
t:A$ is principal in both premises of the cut $(R:)-(L:)$:
\[ \di{\f{
\di{ \f{
\begin{array}{c}
{\mathcal D}_1\\
\ww R\vv,\ww E(t,A),\g\R\d,\vv\Vvdash A
\end{array}
}{\ww E(t,A),\g\R\d,\ww\Vvdash t:A}\,(R:)} \quad\quad
 \di{
 \f{
\begin{array}{c}
{\mathcal D}_2\\
\uu\Vvdash A,\ww\Vvdash t:A,\ww R\uu,\g'\R\d'
\end{array}
 }{\ww\Vvdash t:A,\ww R\uu,\g'\R\d'}\,(L:)}}
 {\ww E(t,A),\ww R\uu,\g,\g'\R\d,\d'}\,Cut} \]
where $\vv$ is not in the conclusion of $(R:)$. First we construct the derivation ${\mathcal D}_3$ as follows:
\[\di{\f{
\di{ \f{
\begin{array}{c}
{\mathcal D}_1\\
\ww R\vv,\ww E(t,A),\g\R\d,\vv\Vvdash A
\end{array}
}{\ww E(t,A),\g\R\d,\ww\Vvdash t:A}\,(R:)} \quad
\begin{array}{c}
 {\mathcal D}_2\\
\uu\Vvdash A,\ww\Vvdash t:A,\ww R\uu,\g'\R\d'
\end{array}
}{\ww E(t,A),\uu\Vvdash A,\ww R\uu,\g,\g'\R\d,\d'}\,Cut_1}\]
 And then we have
\[ \di{\f{
\begin{array}{c}
{\mathcal D}_1(\uu/\vv)\\
\ww R\uu,\ww E(t,A),\g\R\d,\uu\Vvdash A
\end{array}
\quad
\begin{array}{c}
{\mathcal D}_3\\
\ww E(t,A),\uu\Vvdash A,\ww R\uu,\g,\g'\R\d,\d'
\end{array}
}{ \di{\f{\ww E(t,A),\ww E(t,A),\ww R\uu,\ww R\uu,\g,\g,\g'\R\d,\d,\d'}{\ww E(t,A),\ww R\uu,\g,\g'\R\d,\d'}\,Ctr} }\,Cut_2}\]
where $Ctr$ denotes
repeated applications of left and right contraction rules, and the
premise $$\ww R\uu,\ww E(t,A),\g\R\d,\uu\Vvdash A$$ of $Cut_2$ is obtained
by the substitution lemma (Lemma \ref{lemma: substitution lemma})
from $\ww R\vv,\ww E(t,A),\g\R\d,\vv\Vvdash A$. The $Cut_1$ has smaller
level and $Cut_2$ has smaller cut formula. Hence, by the induction
hypothesis, they are admissible.

Now, consider the cut $(R:)-(E)$ as follows:
\[ \di{\f{
\di{ \f{
\begin{array}{c}
{\mathcal D}_1\\
\ww R\vv,\ww E(t,A),\g\R\d,\vv\Vvdash A
\end{array}
}{\ww E(t,A),\g\R\d,\ww\Vvdash t:A}\,(R:)} \quad\quad
 \di{
 \f{
\begin{array}{c}
{\mathcal D}_2\\
\ww E(t,A),\ww\Vvdash t:A,\g'\R\d'
\end{array}
 }{\ww\Vvdash t:A,\g'\R\d'}\,(E)}}
 {\ww E(t,A),\g,\g'\R\d,\d'}\,Cut} \]
where $\vv$ is not in the conclusion of $(R:)$. This derivation is transformed into
\[\di{\f{
\di{ \f{
\begin{array}{c}
{\mathcal D}_1\\
\ww R\vv,\ww E(t,A),\g\R\d,\vv\Vvdash A
\end{array}
}{\ww E(t,A),\g\R\d,\ww\Vvdash t:A}\,(R:)} \quad
\begin{array}{c}
 {\mathcal D}_2\\
\ww E(t,A),\ww\Vvdash t:A,\g'\R\d'
\end{array}
}{\di{\f{\ww E(t,A),\ww E(t,A),\g,\g'\R\d,\d'}{\ww E(t,A),\g,\g'\R\d,\d'}}\,(LC)}\,Cut}\]
 where the above $Cut$ is of smaller level, and thus by the induction hypothesis is admissible.

Now suppose the cut formula is of the form $\ww E(t,A)$ or $\ww R\vv$. We study two cases here. First note that, by a simple
inspection of all rules, we find out that no relational atoms
$\ww R\vv$  and evidence atoms $\ww E(t,A)$ can be
principal in the succedent of sequents of rules (so as a cut formula they might be a side
formula in a rule, or a principal formula in an initial sequent). Let us consider the following $Cut$ with cut
formula $\ww E(t,A)$:
\[ \di{\f{
\di{ \f{
\begin{array}{c}
{\mathcal D}_1\\
\g''\R\d'',\ww E(t,A)
\end{array}
}{\g\R\d,\ww E(t:A)}\,(R)} \quad\quad
 \di{
 \f{
\begin{array}{c}
{\mathcal D}_2\\
\ww R\vv,\ww E(t,A),\g'\R\d',\vv\Vvdash A
\end{array}
 }{\ww E(t,A),\g'\R\d',\ww\Vvdash t:A}\,(R:)}}
 {\g,\g'\R\d,\d',\ww\Vvdash t:A}\,Cut} \]
where the rule $(R)$ in the left premise of the $Cut$ can be any
rule. We permute the cut upward to obtain a $Cut$ with lower level,
and then we apply the rule $(R)$:
\[ \di{\f{
\begin{array}{c}
{\mathcal D}_1\\
\g''\R\d'',\ww E(t:A)
\end{array}
\quad \di{
 \f{
\begin{array}{c}
{\mathcal D}_2\\
\ww R\vv,\ww E(t,A),\g'\R\d',\vv\Vvdash A
\end{array}
 }{\ww E(t,A),\g'\R\d',\ww\Vvdash t:A}\,(R:)}
}{\di{\f{\g'',\g'\R\d'',\d',\ww\Vvdash t:A}{\g,\g'\R\d,\d',\ww\Vvdash
t:A}}\,(R)}\,Cut} \]
 In the case that the rule $(R)$ is a rule
with label condition, such as $(R:)$, first we apply a
suitable substitution on labels (substitute the eigenlabel of
$(R:)$ with a fresh label not used in the derivation), and then we
proceed as the above argument. For example, consider the following
$Cut$ on $\ww R\vv$:
\[ \di{\f{
\di{ \f{
\begin{array}{c}
{\mathcal D}_1\\
\uu R\uu',\uu E(s,B),\g\R\d,\ww R\vv,\uu'\Vvdash B
\end{array}
}{\uu E(s,B),\g\R\d,\ww R\vv,\uu\Vvdash s:B}\,(R:)} \quad\quad
 \di{
 \f{
\begin{array}{c}
{\mathcal D}_2\\
\vv\Vvdash A,\ww R\vv,\ww\Vvdash t:A,\g'\R\d'
\end{array}
 }{\ww R\vv,\ww\Vvdash t:A,\g'\R\d'}\,(L:)}}
 {\uu E(s,B),\ww\Vvdash t:A,\g,\g'\R\d,\d',\uu\Vvdash s:B}\,Cut} \]
It is transformed into the following $Cut$ with lower level:
\[ \di{\f{
\begin{array}{c}
{\mathcal D}_1(\uu''/\uu')\\
\uu R\uu'',\uu E(s,B),\g\R\d,\ww R\vv,\uu''\Vvdash B
\end{array}
\quad \di{
 \f{
\begin{array}{c}
{\mathcal D}_2\\
\vv\Vvdash A,\ww R\vv,\ww\Vvdash t:A,\g'\R\d'
\end{array}
 }{\ww R\vv,\ww\Vvdash t:A,\g'\R\d'}\,(L:)}
}{\di{\f{\uu R\uu'',\uu E(s,B),\ww\Vvdash t:A,\g,\g'\R\d,\d',\uu''\Vvdash
B}{\uu E(s,B),\ww\Vvdash t:A,\g,\g'\R\d,\d',\uu\Vvdash
s:B}}\,(R:)}\,Cut}
\]
 where $\uu''$ is a fresh label not used in the derivation.\qe \end{proof}

\section{Soundness and completeness}\label{sec: Soundness Completeness}
In this section, we shall prove the soundness and completeness of
our labeled sequent calculi with respect to F-models. The proofs of soundness (Theorem \ref{Soundness
labeled systems}) and completeness (Theorem \ref{thm:reduction tree}) are similar to those of modal logics given by
Negri in \cite{Negri2009}. The proof of completeness presents a
derivation of a given valid sequent or a countermodel in the case
of failure of proof search. Hence, it helps us to give a proof search for sequents. We start with the
definition of an interpretation (offered by Negri in
\cite{Negri2009}) which provides a translation between labels used in derivations of a labeled sequent calculus and possible worlds in F-models.
\begin{definition}
let $\M=(\W,\RR,\E,\V)$ be a ${\sf JL}_\CS$-model and $L$ be the set of all labels. An interpretation
$[\cdot]$ of the labels $L$ in model $\M$, or simply an $\M$-interpretation, is a function $[\cdot]:L\r \W$. An
$\M$-interpretation $[\cdot]$ validates a formula of the extended labeled
language in the following sense:
\begin{itemize}
 \item $[\cdot]$ validates the labeled formula $\ww\Vvdash A$, provided that $(\M,[\ww])\Vdash A$,
  \item $[\cdot]$ validates the relational atom $\ww R\vv$, provided that $[\ww]\RR [\vv]$,
   \item $[\cdot]$ validates the evidence atom $\ww E(t,A)$, provided that $[\ww]\in\E(t,A)$.
\end{itemize}
A sequent $\g\R\d$ is valid for an $\M$-interpretation $[\cdot]$, if whenever $[\cdot]$ validates all the formulas in $\g$
then it validates at least one formula in $\d$. A sequent is valid
in a model $\M$ if it is valid for every $\M$-interpretation. 
\end{definition}
 The following lemma is helpful in the
rest of the paper.
\begin{lemma}\label{validity of formula and sequence}
Given a \JL-formula $A$ and a ${\sf JL}_\CS$-model $\M$, the formula $A$
is true in $\M$ (i.e. $\M\Vdash A$) if and only if the sequent $\R \ww\Vvdash A$ is
valid in $\M$, for arbitrary label $\ww$.
\end{lemma}
\begin{proof} Suppose that $A$ is true in the model $\M=(\W,\RR,\E,\V)$. Then for
every $\M$-interpretation $[\cdot]$, and every label $\ww$
we have $[\ww]\in\W$, and therefore $(\M,[\ww])\Vdash A$. Thus, the
interpretation $[\cdot]$ validates the sequent $\R \ww\Vvdash A$.
Since the interpretation $[\cdot]$ is arbitrary, the sequent $\R \ww\Vvdash A$ is valid in $\M$.

Conversely, suppose the sequent $\R \ww\Vvdash A$ is valid in $\M$,
i.e., it is valid for every $\M$-interpretation. For an
arbitrary world $w\in\W$, define the interpretation $[\cdot]$ on
$\M$ such that $[\ww]=w$. Since $[\cdot]$ validates $\R \ww\Vvdash A$,
we have $(\M,w)\Vdash A$. Thus, $A$ is true in $\M$.\qe \end{proof}

Now we show the soundness of labeled sequent calculi with respect to F-models.
\begin{theorem}[Soundness]\label{Soundness labeled systems}
If the sequent $\g\R\d$ is derivable in $\J_\CS$, then it is
valid in every ${\sf JL}_\CS$-model.
\end{theorem}
\begin{proof} By induction on the height of the derivation of $\g\R\d$ in
$\J_\CS$. Initial sequents are obviously valid in every ${\sf
JL_\CS}$-model. We only check the induction step for the rules
$(E)$, $(R:)$, $(E?)$, $(AN)$, and $(Ser)$. For the case of propositional
rules and the rules of the accessibility relation, refer
to Theorem 5.3 in \cite{Negri2009}. The case of other rules
are similar or simpler.

Suppose $\g\R\d$ is $\ww\Vvdash t:A,\g'\R\d$, the conclusion of
rule $(E)$, with the premise $\ww E(t,A),\ww\Vvdash t:A,\g'\R\d$,
and assume by the induction hypothesis that the premise is valid
in every ${\sf JL}_\CS$-model. Let $\M=(\W,\RR,\E,\V)$ be a ${\sf
JL}_\CS$-model and $[\cdot]$ be an arbitrary
$\M$-interpretation which validates $\ww\Vvdash t:A$ and all the formulas in
$\g'$. In particular, $(\M,[\ww])\Vdash t:A$. We have to
prove that this interpretation validates one of the formulas in
$\d$. Since $(\M,[\ww])\Vdash t:A$, by the definition of forcing
relation $\Vdash$, we have $[\ww]\in\E(t,A)$. Thus, $[\cdot]$
validates all the formulas in the antecedent of the premise. Hence
by the induction hypothesis, it validates one of the formulas in
$\d$, as desire.

Suppose $\g\R\d$ is $\ww E(t,A),\g'\R\d',\ww\Vvdash t:A$, the
conclusion of  rule $(R:)$, with the premise $\ww R\vv,\ww E(t,A),\g'\R\d',\vv\Vvdash A$, and assume by the induction
hypothesis that the premise is valid in every ${\sf JL}_\CS$-model. Let $\M=(\W,\RR,\E,\V)$ be a ${\sf JL}_\CS$-model and $[\cdot]$ be an arbitrary $\M$-interpretation
which validates $\ww E(t,A)$ and all the formulas in $\g'$. In particular,
$[\ww]\in\E(t,A)$. We have to prove that this interpretation
validates one of the formulas in $\d'$ or validates $\ww\Vvdash t:A$. Suppose
$w$ is an arbitrary element of $\W$ such that $[\ww]\RR w$ (if such
a world $w$ does not exist, then obviously we have $(\M,[\ww])\Vdash
t:A$) and $[\cdot]'$ be the interpretation identical to $[\cdot]$
except possibly on $\vv$, where we put $[\vv]'= w$. Clearly,
$[\cdot]'$ validates all the formulas in the antecedent of the
premise, so it validates a formula in $\d'$ or validates $\vv\Vvdash A$. In the former case, since $\vv$ is not in $\d'$,
$[\cdot]$ validates a formula in $\d'$. In the latter case, we have $(\M,w)\Vdash A$. Since by the assumption $[\ww]\in\E(t,A)$ and $w$ is arbitrary, $[\cdot]$
validates $\ww\Vvdash t:A$.

Among rules for evidence atoms we only check the induction step for $(E?)$ in
{\sf G3J5} and its extensions. Let \JL~be {\sf J5} or one of its extensions. Suppose $\g\R\d$ is the conclusion of the rule $(E?)$, with
the premises $\ww E(t,A),\g\R\d$ and $\ww E(?t,\neg t:A),\g\R\d$, and assume
by the induction hypothesis that the premises are valid in every
${\sf JL}_\CS$-model. Let $\M=(\W,\RR,\E,\V)$ be a ${\sf JL}_\CS$-model and $[\cdot]$ be an arbitrary $\M$-interpretation which validates all the formulas in $\g$. We have to
prove that this interpretation validates one of the formulas in
$\d$. There are two cases: (i) If $[\ww]\in\E(t,A)$ then the antecedent of the premise $\ww E(t,A),\g\R\d$ is
validated in $[\cdot]$, and
hence one of the formulas in
$\d$ is validated. (ii) If $[\ww]\not\in\E(t,A)$, then by negative introspection condition $(\E6)$ of ${\sf JL}_\CS$-models we have
$[\ww]\in \E(?t,\neg t:A)$. Thus the antecedent of the premise $\ww E(?t,\neg t:A),\g\R\d$ is
validated in $[\cdot]$. Hence, $[\cdot]$ validates a formula in
$\d$. Therefore the conclusion is valid in $\M$.

For axiom necessitation rules we only consider $(AN)$ (rule
$(IAN)$ is treated similarly). Suppose $\g\R\d$ is the conclusion
of the rule $(AN)$, with the premise $\ww E(c,A),\g\R\d$, and
assume by the induction hypothesis that the premise is valid in
every ${\sf JL}_\CS$-model, and $c:A\in \CS$. Let
$\M=(\W,\RR,\E,\V)$ be a ${\sf JL}_\CS$-model and $[\cdot]$ be an
arbitrary $\M$-interpretation which validates all the
formulas in $\g$. We
have to prove that this interpretation validates one of the
formulas in $\d$. Since $\M$ is a ${\sf JL}_\CS$-model, we have $\E(c,A)=\W$. Thus $[\ww]\in\E(c,A)$,
and $[\cdot]$ validates the antecedent of the premise. Hence
$[\cdot]$ validates a formula in $\d$, as desire.

Let \JL~be {\sf JD} or one of its extensions. Suppose $\g\R\d$ is the
conclusion of rule $(Ser)$, with the premise $\ww R\vv,\g\R\d$, and assume by the induction
hypothesis that the premise is valid in every ${\sf JL}_\CS$-model. Let $\M=(\W,\RR,\E,\V)$ be a ${\sf JL}_\CS$-model and $[\cdot]$ be an arbitrary $\M$-interpretation
which validates all the formulas in $\g$. We have to prove that $[\cdot]$
validates one of the formulas in $\d$. Since $[\ww]\in\W$ and $\RR$ is serial, there exists $w\in\W$ such that $[\ww]\RR w$. Let $[\cdot]'$ be the interpretation identical to $[\cdot]$
except possibly on $\vv$, where we put $[\vv]'= w$. Since the eigenlabel $\vv$ is not in $\g$, $[\cdot]'$ validates all the formulas in $\g$. Thus, $[\cdot]'$ validates all the formulas in the antecedent of the
premise, so it validates a formula in $\d$. Since the eigenlabel $\vv$ is not in $\d$, $[\cdot]$ validates all the formulas in $\d$.
\qe \end{proof}
Now we show that theorems of Hilbert systems of justification logics are derivable in their labeled systems.
\begin{proposition}\label{derivability of theorems in labels systems}
If $A$ is a theorem of ${\sf JL}_{\CS}$, then $\R \ww\Vvdash A$ is derivable in $\J_\CS$, for arbitrary label $\ww$.
\end{proposition}
\begin{proof}
By induction on the derivation of $A$ in ${\sf JL}_\CS$. If $A$ is an axiom of ${\sf JL}_\CS$, then apply root-first the rules of $\J_\CS$ and possibly use Lemma \ref{lem:Initial axiom for arbitrary A}.

For example, we derive the axiom {\bf jD}, $t:\bot\r\bot$, in ${\sf
G3JD}_\CS$:
\begin{prooftree}
\AXC{$( Ax\bot)$} \noLine
\UIC{$\vv\Vvdash \bot,\ww R\vv,\ww\Vvdash t:\bot
\fCenter \ww\Vvdash \bot$}
 \RightLabel{$(L:)$}
 \UIC{$\ww R\vv,\ww\Vvdash t:\bot
\fCenter \ww\Vvdash \bot$}
 \RightLabel{$(Ser)$}
 \UIC{$\ww\Vvdash t:\bot \fCenter \ww\Vvdash \bot$}
 \RightLabel{($R\r)$}
  \UIC{$\fCenter \ww\Vvdash t:\bot\r  \bot$}
\end{prooftree}
If $A$ is obtained by Modus Ponens from $B\r A$ and $B$, then by
the induction hypothesis and admissibility of rule $Cut$, the sequent $\R
\ww\Vvdash A$ is derivable in \J. 
If $A=c:B\in\CS$ is obtained by Axiom
Necessitation rule, then we have:
\[\di{\f{
\begin{array}{c}
{\mathcal D}\\
\ww E(c,B),\ww R\vv\R \vv\Vvdash B
\end{array}
}{\di{\f{\ww E(c,B)\R \ww\Vvdash c:B}{\R \ww\Vvdash
c:B}\,(AN)}}\,(R:)}\]
 where ${\mathcal D}$ is the standard
derivation of axiom $B$ in $\J_\CS$. Now, suppose $A=c_{i_n}:c_{i_{n-1}}:\ldots:c_{i_1}:B\in\CS$ is obtained by Iterated Axiom
Necessitation rule. Note that
constant specifications are downward closed. Thus, in order to
prove $\R \ww\Vvdash c_{i_n}:c_{i_{n-1}}:\ldots:c_{i_1}:B$, we can add $\ww E(c_{i_m}:c_{i_{m-1}}:\ldots:c_{i_1}:B)$, for every $1\leq m\leq n$, to the
antecedent of the premise of $(IAN)$, when we apply it upward. We have:
\scalebox{0.95}{\parbox{\textwidth}{%
\begin{prooftree}
\AXC{${\mathcal D}$}
 \noLine
 \UIC{$\g\fCenter \ww_n \Vvdash B$}
 \noLine
 \UIC{$\vdots$} \noLine
 \UIC{$\ww_1 R \ww_2,\ww_1  E(c_{i_{n-1}},c_{i_{n-2}}\ldots:c_{i_1}:B), \ww R \ww_1 ,
 \ww_1  E(c_{i_n},c_{i_{n-1}}:\ldots:c_{i_1}:B)\fCenter \ww_2\Vvdash
 c_{i_{n-2}}\ldots:c_{i_1}:B$}
 \RightLabel{$(R:)$}
 \UIC{$\ww_1  E(c_{i_{n-1}},c_{i_{n-2}}\ldots:c_{i_1}:B), \ww R \ww_1 ,
 \ww E(c_{i_n},c_{i_{n-1}}:\ldots:c_{i_1}:B)\fCenter \ww_1 \Vvdash
 c_{i_{n-1}}:\ldots:c_{i_1}:B$}
 \RightLabel{$(IAN)$}
 \UIC{$\ww R \ww_1 ,
 \ww E(c_{i_n},c_{i_{n-1}}:\ldots:c_{i_1}:B)\fCenter \ww_1 \Vvdash
 c_{i_{n-1}}:\ldots:c_{i_1}:B$}
 \RightLabel{$(R:)$}
 \UIC{$\ww E(c_{i_n},c_{i_{n-1}}:\ldots:c_{i_1}:B)\fCenter \ww\Vvdash
 c_{i_n}:c_{i_{n-1}}:\ldots:c_{i_1}:B$}
 \RightLabel{$(IAN)$}
 \UIC{$\fCenter \ww\Vvdash
 c_{i_n}:c_{i_{n-1}}:\ldots:c_{i_1}:B$}
\end{prooftree}}}

where
\[\g=\{\ww_{n-1} R \ww_n , \ww_{n-2} R \ww_{n-1}, \ldots, \ww R\ww_1 , \ww_{n-1}\in
E(c_{i_1},B),\] \[ \ww_{n-2} E(c_{i_2},c_{i_1}:B),\ldots,\ww_1 \in
E(c_{i_{n-1}},c_{i_{n-2}}\ldots:c_{i_1}:B), \ww E(c_{i_n},c_{i_{n-1}}:\ldots:c_{i_1}:B)\},\]
 and ${\mathcal D}$ is the standard derivation of axiom $B$ in $\J_\CS$.\qe \end{proof}

Now we prove that the labeled sequent calculi of justification
logics are equivalent to their Hilbert systems.
\begin{corollary}\label{cor:equivalence JL and G3JL}
Let \JL~be a justification logic, and $\CS$ be a constant specification for \JL, with the requirement that if \JL~contains axiom scheme {\bf jD} then \CS~should be axiomatically appropriate. Then $A$ is true in every $\JL_\CS$-model if and only if the sequent $\R \ww\Vvdash A$ is provable in $\J_\CS$, for arbitrary label $\ww$.
\end{corollary}
\begin{proof}  Suppose that the sequent $\R \ww\Vvdash A$ is provable in
$\J_\CS$. Then by Soundness Theorem \ref{Soundness labeled
systems}, $\R \ww\Vvdash A$ is valid in every ${\sf
JL}_\CS$-model, and therefore by Lemma
\ref{validity of formula and sequence}, the formula $A$ is true in
every ${\sf JL}_\CS$-model.

Conversely, suppose that the formula $A$ is true in every
$\JL_\CS$-model. Then by Completeness Theorem
\ref{Sound Compl JL}, the formula $A$ is provable in $\JL_\CS$, and therefore by Proposition \ref{derivability of theorems
in labels systems}, the sequent $\R \ww\Vvdash A$ is provable in
$\J_\CS$. \qe \end{proof}

We present another proof for the completeness theorem of labeled sequent calculi $\J^-$ with respect to F-models, by describing a procedure (in fact a backward proof search for a sequent) that produces a derivation for valid sequents and a countermodel for non-valid  sequents. In backward proof search of a sequent $\g\R\d$, some rules repeat the main formula in their premise(s), and therefore they can be applied infinitely many times. For example, applying rule $(L:)$
backwardly on the sequent
\[ \ww\Vvdash t:A, \ww R\vv,\g'\R\d'\]
we get
\[\vv\Vvdash A ,\ww\Vvdash t:A, \ww R\vv,\g'\R\d'\]
Because of the formulas $\ww\Vvdash t:A$ and $\ww R\vv$ in the antecedent of the premise, we can apply this rule again, and indeed  infinitely many
times, backwardly. Since the rules of contraction are height-preserving admissible, it seems that the rule $(L:)$ does not need to apply on each pair of formulas $\ww\Vvdash t:A$
and $\ww R\vv$ more than once. To show this fact, we first show that applications of $(L:)$ on the same pair of principal formulas can be made consecutive by the permutation of rule $(L:)$ over other rules.

\begin{lemma}\label{lem:permut down L:}
Rule $(L:)$ permutes down with respect to all rules of $\J_\CS$. The permutability with respect to $(R:)$, and rules for relational atoms (Table \ref{table: rules for accessibility relation}) have the condition that the principal formulas of $(L:)$ are not active in them.
\end{lemma}
\begin{proof} The proof is similar to that in \cite{Negri2005} (rule $(L:)$ is treated like rule $(L\b)$ in modal logics). We only show the permutation with respect to $(E)$ and $(R:)$. In the derivation
\begin{prooftree}
\AXC{$\mathcal{D}$}\noLine
\UIC{$\vv\Vvdash A, \uu E(s,B),\uu\Vvdash s:B,\ww\Vvdash t:A,\ww R\vv,\g\R\d$}
\RightLabel{$(L:)$}
\UIC{$ \uu E(s,B),\uu\Vvdash s:B,\ww\Vvdash t:A,\ww R\vv,\g\R\d$}
\RightLabel{$(E)$}
\UIC{$ \uu\Vvdash s:B,\ww\Vvdash t:A,\ww R\vv,\g\R\d$}
\end{prooftree}
$(L:)$ can permute down as follows
\begin{prooftree}
\AXC{$\mathcal{D}$}\noLine
\UIC{$ \uu E(s,B),\vv\Vvdash A,\uu\Vvdash s:B,\ww\Vvdash t:A,\ww R\vv,\g\R\d$}
\RightLabel{$(E)$}
\UIC{$ \vv\Vvdash A,\uu\Vvdash s:B,\ww\Vvdash t:A,\ww R\vv,\g\R\d$}
\RightLabel{$(L:)$}
\UIC{$ \uu\Vvdash s:B,\ww\Vvdash t:A,\ww R\vv,\g\R\d$}
\end{prooftree}
Note that in this case the principal formula of $(L:)$ could be active in $(E)$, i.e. in the above derivations $\uu\Vvdash s:B$ can be equal to $\ww\Vvdash t:A$. For $(R:)$, in the derivation
\begin{prooftree}
\AXC{$\mathcal{D}$}\noLine
\UIC{$\vv\Vvdash A,\uu R\uu',\uu E(s,B),\ww\Vvdash t:A,\ww R\vv,\g\R\d, \uu'\Vvdash B$}
\RightLabel{$(L:)$}
\UIC{$ \uu R\uu',\uu E(s,B),\ww\Vvdash t:A,\ww R\vv,\g\R\d, \uu'\Vvdash B$}
\RightLabel{$(R:)$}
\UIC{$ \uu  E(s,B),\ww\Vvdash t:A,\ww R\vv,\g\R\d,\uu\Vvdash s:B$}
\end{prooftree}
where principal formulas of $(L:)$, i.e. $\ww\Vvdash t:A,\ww R\vv$, are not active in $(R:)$, the permutation of $(L:)$ over $(R:)$ is as follows
\begin{prooftree}
\AXC{$\mathcal{D}$}\noLine
\UIC{$\uu R\uu',\vv\Vvdash A,\uu E(s,B),\ww\Vvdash t:A,\ww R\vv,\g\R\d, \uu'\Vvdash B$}
\RightLabel{$(R:)$}
\UIC{$\vv\Vvdash A,\uu E(s,B),\ww\Vvdash t:A,\ww R\vv,\g\R\d, \uu\Vvdash s:B$}
\RightLabel{$(L:)$}
\UIC{$ \uu  E(s,B),\ww\Vvdash t:A,\ww R\vv,\g\R\d,\uu\Vvdash s:B$}
\end{prooftree}
\qe \end{proof}

\begin{corollary}
In  a branch of a derivation in $\J_\CS$,  it is enough to apply rule $(L:)$ only once on the same pair of principal formulas.
\end{corollary}
\begin{proof} Suppose $(L:)$ is applied twice on the same principal formula, for example
\begin{prooftree}
\AXC{$\mathcal{D}$}\noLine
\UIC{$\vv\Vvdash A,\ww\Vvdash t:A,\ww R\vv,\g'\R\d'$}
\RightLabel{$(L:)$}
\UIC{$\ww\Vvdash t:A,\ww R\vv,\g'\R\d'$}
\noLine
\UIC{$\vdots~\mathcal{D}'$}
\noLine
\UIC{$\vv\Vvdash A,\ww\Vvdash t:A,\ww R\vv,\g\R\d$}
\RightLabel{$(L:)$}
\UIC{$\ww\Vvdash t:A,\ww R\vv,\g\R\d$}
\end{prooftree}
Then, since the upper application of $(L:)$ is applied on the same principal formulas $\ww\Vvdash t:A,\ww R\vv$, in the part $\mathcal{D}'$ of the derivation there is no rule $(R:)$, or rules for relational atoms with active formulas  $\ww\Vvdash t:A,\ww R\vv$. Therefore, by Lemma \ref{lem:permut down L:}, by permuting down the upper $(L:)$ we obtain
\begin{prooftree}
\AXC{$\mathcal{D}$}\noLine
\UIC{$\vv\Vvdash A,\ww\Vvdash t:A,\ww R\vv,\g'\R\d'$}\noLine
\UIC{$\vdots~\mathcal{D}'$}
\noLine
\UIC{$\vv\Vvdash A,\vv\Vvdash A,\ww\Vvdash t:A,\ww R\vv,\g\R\d$}
\RightLabel{$(L:)$}
\UIC{$\vv\Vvdash A,\ww\Vvdash t:A,\ww R\vv,\g\R\d$}
\RightLabel{$(L:)$}
\UIC{$\ww\Vvdash t:A,\ww R\vv,\g\R\d$}
\end{prooftree}
By applying height-preserving contraction, we see that the upper application of $(L:)$ is redundant.\qe \end{proof}

Since the rules of contraction are height-preserving admissible, we consider the
following requirement for all rules:

\begin{quote}
$(\dagger)$ ~~\textit{In backward proof search, apply a rule backwardly only when it does not produce a formula in the antecedent of the premise which exists there already.}
\end{quote}

In the following theorem we give a procedure that determines the derivability of sequents in $\J^-$.

\begin{theorem}\label{thm:reduction tree}
Let $\JL$ be one of the justification logics {\sf J}, {\sf J4}, {\sf JD}, {\sf JD4}, {\sf JT}, {\sf LP},  and $\CS$ be a finite constant specification for \JL. Every sequent
$\g\R\d$ in the language of \J~is either derivable in $\J_\CS$ or it has a ${\sf
JL_\CS}$-countermodel.
\end{theorem}
\begin{proof} Following the proof of completeness of labeled systems of
modal logics (Theorem 5.4 in \cite{Negri2009}), we present a procedure which constructs a
finitely branching \textit{reduction tree} with the root $\g\R\d$
by applying the rules of $\J^-_\CS$ in all possible ways. If all branches
of the reduction tree reach initial sequents we obtain a proof for
$\g\R\d$. Otherwise we have a branch in which its topmost sequent is not an initial sequent and no reduction step can be applied or it is an infinite branch. In this case, we construct a countermodel by means of this branch. In constructing the reduction tree the procedure always obey the condition $(\dagger)$. Moreover, by Corollary \ref{cor:analyticity}, the procedure for ${\sf G3JL}^-$ obeys the subterm and sublabel properties.

 {\bf Reduction tree}: In stage 0, we put $\g\R\d$ at the root of
the tree. For each branch, in stage $n>0$, if the top-sequent of the branch is an initial sequent, i.e. $(Ax)$, $(Ax\bot)$, $(AxR)$, $(AxE)$, then we terminate the
construction of the branch. Otherwise, we continue the construction of the branch by writing, above its top-sequent, other sequent(s) that are obtained by applying the following stages, depend on the rules of the logic. The number of stages in
the construction of the reduction tree depends on the number of
rules which $\J^-_\CS$ contains. In general, the following
list (which determines an order on the rules of the system) can be used for various systems dealt here:

\begin{equation}\label{number equation: order of rules}
\begin{array}{ccccccccccc}
(L\neg), &(R\neg), &(L\wedge), &(R\wedge), &(L\vee), &(R\vee),&
(L\r), & (R\r), &(L:), &(R:),\\ (E),& (AN)/(IAN), & (El+), &(Er+),&
 (E\cdot), &(E!),  &(Mon),  & (Ref),  &(Ser),& (Trans).
\end{array}
\end{equation}

There are 15 rules common in all labeled systems in the list (\ref{number equation: order of rules}): $(L\neg), (R\neg), \ldots, (E\cdot)$. Thus each $\J^-_\CS$ system has $15+r$ stages, for $r\geq 0$. At stage $15+r+1$ we come back to stage 1, and continue until an initial sequent  is found or the branch becomes saturated.\footnote{Another alternative to define the reduction tree is to stipulate that, instead of applying  the stages consecutively according to list (\ref{number equation: order of rules}), at each stage apply one of the stages non-deterministically.} A branch is called \textit{saturated} if it is an infinite branch or if its top-sequent is not an initial sequent and no more stage of the reduction tree can be applied. Otherwise it is said to be \textit{unsaturated}.\footnote{The terminology is due to Dyckhoff and Negri \cite{DyckhoffNegri2013}, but our definition is a bit different. In fact in the reduction tree described in \cite{DyckhoffNegri2013} all branches are finite.} We first consider 15 common
stages in all of the justification logics.

Stage $n=1$ of rule $(L\neg)$: If the top-sequent is of the form:
\[ \ww_1 \Vvdash \neg A_1,\ldots,\ww_m \Vvdash \neg A_m, \g'\R\d'\]
where $\ww_1 \Vvdash \neg A_1,\ldots,\ww_m \Vvdash \neg A_m$ are all the
formulas in the antecedent with a negation as the outermost
logical connective, we write
\[  \g'\R\d',\ww_1 \Vvdash  A_1,\ldots,\ww_m \Vvdash  A_m\]
on top of it. This stage corresponds to applying $m$ times rule
$(L\neg)$.

Stage $n=2$ of rule $(R\neg)$: If the top-sequent is of the form:
\[  \g'\R\d',\ww_1 \Vvdash \neg A_1,\ldots,\ww_m \Vvdash \neg A_m\]
where $\ww_1 \Vvdash \neg A_1,\ldots,\ww_m \Vvdash \neg A_m$ are all the
formulas in the succedent with a negation as the outermost logical
connective, we write
\[ \ww_1 \Vvdash  A_1,\ldots,\ww_m \Vvdash  A_m, \g'\R\d'\]
on top of it. This stage corresponds to applying $m$ times rule
$(R\neg)$. For the stages $n=3,4,\ldots,8$, correspond to other
propositional rules in the list (\ref{number equation:
order of rules}), refer to the proof of Theorem 5.4 in
\cite{Negri2009}.

Stage $n=9$ of rule $(L:)$: If the top-sequent is of the form:
\[ \ww_1 \Vvdash t_1:A_1,\ldots,\ww_m \Vvdash t_m:A_m,\ww_1 R\vv_1,\ldots,\ww_m R\vv_m,\g'\R\d'\]
where all pairs $\ww_i\Vvdash t_i:A_i$ and $\ww_i R\vv_i$ from the
antecedent of the topmost sequent are listed,  we write the following node on top of it (regarding
condition $(\dagger)$):
\[  \vv_1\Vvdash A_1,\ldots,\vv_m\Vvdash A_m,\ww_1 \Vvdash t_1:A_1,\ldots,\ww_m \Vvdash t_m:A_m,\ww_1 R\vv_1,\ldots,\ww_m R\vv_m,\g'\R\d'.\]
This stage corresponds to applying $m$ times rule $(L:)$.

Stage $n=10$ of rule $(R:)$: If the top-sequent is of the form:
\[\ww_1  E(t_1,A_1),\ldots,\ww_m  E(t_m,A_m), \g'\R\d',\ww_1 \Vvdash t_1:A_1,\ldots,\ww_m \Vvdash t_m:A_m\]
where all pairs $\ww_i E(t_i,A_i)$ and $\ww_i\Vvdash t_i:A_i$ from
the topmost sequent are listed, we write the following node
on top of it:
\[ \ww_1 R\vv_1,\ldots,\ww_m R\vv_m,\ww_1  E(t_1,A_1),\ldots,\ww_m  E(t_m,A_m),\g'\R\d',\vv_1\Vvdash A_1,\ldots,\vv_m\Vvdash A_m\]
where $\vv_1,\dots,\vv_m$ are fresh labels, not yet used in the
reduction tree. This stage corresponds to applying $m$ times rule
$(R:)$.

Stage $n=11$ of rule $(E)$: If the top-sequent is of the form:
\[ \ww_1 \Vvdash t_1:A_1,\ldots,\ww_m \Vvdash t_m:A_m,\g'\R\d'\]
where all labeled formulas $\ww_i\Vvdash t_i:A_i$ from
the topmost sequent are listed, we write the following node on
top of it (regarding condition $(\dagger)$):
\[\ww_1  E(t_1,A_1),\ldots,\ww_m  E(t_m,A_m),\ww_1 \Vvdash t_1:A_1,\ldots,\ww_m \Vvdash t_m:A_m, \g'\R\d'\]
This stage corresponds to applying $m$ times rule $(E)$.

Stage $n=12$ of rule $(IAN)/(AN)$:  If $\g'\R\d'$ is the top-sequent of the branch, then write a similar sequent on top of it (regarding condition $(\dagger)$) with additional evidence atoms of the form $\ww E(c,F)$ in the antecedent,  for every formula $c:F$ in $\CS$ and every $\ww$ in $\g'\cup\d'$.

Stage $n=13$ of rule $(El+)$: If the top-sequent is of the form:
\[\ww_1  E(t_1,A_1),\ldots,\ww_m  E(t_m,A_m),\g'\R\d'\]
where $\ww_i E(t_i,A_i)$ are all evidence atoms for which $s_i+t_i\in Sub_{Tm}(\g'\R\d')$ for some term $s_i$ (Proposition \ref{prop:subterm property}) and $\ww_i E(s_i+t_i,A_i)$ is not in $\g'$ (condition
$(\dagger)$), then we add the following node on top of it:
\[
\ww_1  E(s_1+t_1,A_1),\ldots,\ww_m  E(s_m+t_m,A_m),\ww_1  E(t_1,A_1),\ldots,\ww_m  E(t_m,A_m),\g'\R\d'
\]
 This corresponds to applying  $m$ times rule $(El+)$. The stage $n=14$ of rule $(Er+)$ is similar.

 Stage $n=15$ of rule $(E\cdot)$: If the top-sequent is of the form:
\[\ww_1  E(s_1,A_1\r B_1),\ldots,\ww_m  E(s_m,A_m\r B_m),\ww_1  E(t_1,A_1),\ldots,\ww_m  E(t_m,A_m),\g'\R\d'\]
where $\ww_i E(s_i,A_i\r B_i)$ and $\ww_i  E(t_i,A_i)$ are all pairs of evidence atoms for which $s_i\cdot t_i\in Sub_{Tm}(\g'\R\d')$ and $\ww_i E(s_i\cdot t_i,B_i)$ is not in $\g'$, then we add the following node on top of it:
\[
\ww_1 E(s_1\cdot t_1,B_1),\ldots,\ww_m E(s_m\cdot t_m,B_m),\ww_1  E(s_1,A_1\r B_1),\ldots,\ww_m  E(s_m,A_m\r B_m),\] \[\ww_1  E(t_1,A_1),\ldots,\ww_m  E(t_m,A_m),\g'\R\d'
\]
 This corresponds to applying  $m$ times rule $(E\cdot)$.

 Stage of rule $(E!)$: If the labeled system $\J^-$ contains rule $(E!)$, and the top-sequent is of the form:
\[\ww_1  E(t_1,A_1),\ldots,\ww_m  E(t_m,A_m),\g'\R\d'\]
where $w_i E(t_i,A_i)$ are all evidence atoms for which $!t_i\in Sub_{Tm}(\g'\R\d')$  and $w_i E(!s_i,A_i)$ is not in $\g'$, then we add the following node on top of it:
\[
\ww_1  E(!t_1,A_1),\ldots,\ww_m  E(!t_m,A_m),\ww_1  E(t_1,A_1),\ldots,\ww_m  E(t_m,A_m),\g'\R\d'
\]
 This corresponds to applying  $m$ times rule $(E!)$.

 If the labeled system $\J^-$ contains rule $(Mon)$, and the top-sequent is of the form:
\[ \ww_1 R\vv_1,\ldots,\ww_m R\vv_m,\ww_1  E(t_1,A_1),\ldots,\ww_m  E(t_m,A_m),\g'\R\d'\]
where all pairs $\ww_i R\vv_i$ and $\ww_i E(t_i,A_i)$ from the
antecedent of the topmost sequent are listed, then, regarding
condition $(\dagger)$, we write the following node on top of it:
\begin{eqnarray*}
& \vv_1 E(t_1,A_1),\ldots,\vv_m E(t_m,A_m),
 \ww_1 R\vv_1,\ldots,\ww_m R\vv_m,\ww_1  E(t_1,A_1),\ldots,\ww_m  E(t_m,A_m),& \\
 &\ww_1 R\vv_1,\ldots,\ww_m R\vv_m,\g'\R \d'.&
 \end{eqnarray*}
This stage corresponds to applying  $m$ times rule $(Mon)$.

If the labeled system $\J^-$ contains rule $(Ref)$, add to the
antecedent of the top-sequent $\g'\R\d'$ all the relational
atoms $\ww R\ww$, for $\ww$ in $\g'\cup\d'$, that are not in $\g'$ yet (condition
$(\dagger)$).

If the labeled system $\J^-$ contains rule $(Ser)$, add to the
antecedent of the top-sequent $\g'\R\d'$ all the relational
atoms $\ww R\vv$, for $\ww$ in $\g'\cup\d'$ and fresh label $\vv$.

If the labeled system $\J^-$ contains rule $(Trans)$, and the top-sequent is of the form:
\[ \ww_1 R \vv_1,\ldots,\ww_m R \vv_m,\vv_1 R \uu_1,\ldots,\vv_m R \uu_m,\g'\R\d'\]
where all pairs $\ww_iR \vv_i$ and $\vv_i R \uu_i$ from the antecedent
of the topmost sequent are listed, then, regarding condition
$(\dagger)$, we write the following node on top of it:
\[ \ww_1 R \uu_1,\ldots,\ww_m R \uu_m,\ww_1 R \vv_1,\ldots,\ww_m R \vv_m,\ww_1 R \uu_1,\ldots,\ww_m R \uu_m,\g'\R\d'\]

 If all the topmost sequents of the reduction tree are initial sequents, then we terminate the construction of the reduction tree. In this case by transforming each stage of the reduction tree to (possibly more than one application of) the corresponding rule, we can write a derivation for $\g\R\d$. Otherwise, the reduction tree has at least one saturated branch.\\

 {\bf Countermodel}: Suppose the reduction tree has a (finite or infinite) saturated branch,
 say $\g_0\R\d_0,\g_1\R\d_1,\ldots$
where $\g_0\R\d_0$ is the root sequent $\g\R\d$. Let
\[ \overline{\g}=\bigcup_{i\geq 0} \g_i, ~~~~~~~ \overline{\d}=\bigcup_{i\geq 0} \d_i.\]
We shall define a model $\M$ and an $\M$-interpretation $[\cdot]$ in which $[\cdot]$ validates all the formulas in $\overline{\g}$ and no formulas in $\overline{\d}$. The construction of the Fitting countermodel $\M=(\W,\RR,\E_\mathcal{A},\V)$ for
$\g\R\d$ is as follows:
\begin{enumerate}
 \item The set of possible worlds $\W$ is all the labels that occur in $\overline{\g}\cup\overline{\d}$.
  \item The accessibility relation $\RR$ is determined by relational atoms in $\overline{\g}$ as follows: if $\ww R\vv$ is in $\overline{\g}$, then $\ww\RR \vv$ (otherwise $\ww \RR \vv$ does not hold).
   \item For the evidence function $\E$, we first construct a possible evidence function $\mathcal{A}$ on $\W$ for $\JL_\CS$ as follows.
  $\mathcal{A}$ is determined by evidence atoms in $\overline{\g}$: if $\ww E(t,A)$ is in $\overline{\g}$, then $\ww\in\mathcal{A}(t,A)$ (otherwise $\ww\not\in\mathcal{A}(t,A)$).  Since in stage 12 we add evidence atoms $\ww E(c,F)$, for $c:F\in\CS$ and label $\ww$ in the reduction tree, to the antecedent of sequents, we have $\ww\in\mathcal{A}(c,F)$ for every $c:F\in\CS$ and $\ww\in\W$. Therefore, $\mathcal{A}$ is a possible evidence function on $\W$ for $\JL_\CS$. Now let $\E_\mathcal{A}$ be the admissible evidence function based on $\mathcal{A}$ as defined in Definition \ref{def:generated evidence function}.
      \item The valuation $\V$ is determined by labeled formulas in $\overline{\g}$ and $\overline{\d}$ as follows: if $\ww\Vvdash P$ is in $\overline{\g}$, then $\ww\in\V(P)$, and if $\ww\Vvdash P$ is in $\overline{\d}$, then $\ww\not\in\V(P)$ (where $P$ is a propositional variable).
\end{enumerate}

The stages used in the reduction tree ensure that the accessibility relation $\RR$  for $\M$ satisfies those conditions needed for $\JL$-models. Moreover,  by Lemma \ref{lem: generated evidence function is admissible}, $\E_\mathcal{A}$ is an admissible evidence function for $\JL_\CS$. In order to show that $\M=(\W,\RR,\E_\mathcal{A},\V)$ is the desired countermodel we need the following lemmas.

\begin{lemma}\label{lem:E in completeness countermodel}
If $\ww\in\E_\mathcal{A}(r,F)$ and $r\in Sub_{Tm}(\g\R\d)$, then $\ww E(r,F)$ is in $\overline{\g}$.
\end{lemma}
\begin{proof}
Suppose that $\ww\in\E_\mathcal{A}(r,F)$ and $r\in Sub_{Tm}(\g\R\d)$. Then $\ww\in\E_j(r,F)$, for some $j\geq 0$. By induction on $j$ we show that for all $\ww\in\W$

\begin{equation}\label{eq:E in completeness countermodel}
\ww\in\E_j(r,F) \R \ww E(r,F) \in \overline{\g}
\end{equation}

The base case, $j=0$, follows from the definition of $\E_0$. For the induction hypothesis, suppose that (\ref{eq:E in completeness countermodel}) is true for all $\ww\in\W$ and all $0\leq j \leq i$. If $\ww\in\E_{i+1}(r,F)$ for $i\geq 0$, and $\ww\in\E_j(r,F)$ for $j< i+1$, then by the induction hypothesis $\ww E(r,F)$ is in $\overline{\g}$. Assume now that we have $\ww\in\E_{i+1}(r,F)$ for $i\geq 0$, and $\ww\not\in\E_j(r,F)$ for any $j< i+1$. We have the following cases:

 \begin{enumerate}
 \item $r=s+t$, and $\ww\in\E_{i+1}(s+t,F)$. Thus $\ww\in\E_{i}(t,F)$ or $\ww\in\E_{i}(s,F)$. If $\ww\in\E_{i}(t,F)$, then by the induction hypothesis, $\ww E(t,F)$ is in $\overline{\g}$. Therefore, $\ww E(t,F)$ is in $\g_k$, for some $k\geq 0$. Since $r\in Sub_{Tm}(\g\R\d)$, after the stage of rule $(El+)$,  evidence atom $\ww E(s+t,F)$ is added to the antecedent of the sequent. Hence $\ww E(s+t,F)$ is in $\g_l$, for some $l\geq k$. Therefore, $\ww E(s+t,F)$ is in $\overline{\g}$. Proceed similarly if  $\ww\in\E_{i}(s,F)$.

 \item $r=s\cdot t$, and $\ww\in\E_{i+1}(s\cdot t,F)$. Then $\ww\in\E_{i}(s,G\r F)\cap\E_{i}(t,G)$, for some formula $G$. Thus, by the induction hypothesis, $\ww E(s,G\r F)$ and $\ww E(t,G)$ are in $\overline{\g}$. Therefore, $\ww E(s,G\r F)$ and $\ww E(t,G)$ are in $\g_k$, for some $k\geq 0$. Since $r\in Sub_{Tm}(\g\R\d)$, after the stage of rule $(E\cdot)$, evidence atom $\ww E(s\cdot t,F)$ is added to the antecedent of the sequent. Hence $\ww E(s\cdot t,F)$ is in $\g_l$, for some $l\geq k$. Therefore, $\ww E(s\cdot t,F)$ is in $\overline{\g}$.

    \item {\bf j4} is an axiom of $\JL_\CS$, $r=!t$, and $\ww\in\E_{i+1}(!t,F)$. Then $\ww\in\E_{i}(t,G)$, for some formula $G$ such that $F=t:G$. Thus, by the induction hypothesis, $\ww E(t,G)$ is in $\overline{\g}$. Therefore, $\ww E(t,G)$ is in $\g_k$, for some $k\geq 0$. Since $r\in Sub_{Tm}(\g\R\d)$, after the stage of rule $(E!)$, evidence atom $\ww E(!t,F)$ is added to the antecedent of the sequent. Hence $\ww E(!t,F)$ is in $\g_l$, for some $l\geq k$. Therefore, $\ww E(!t,F)$ is in $\overline{\g}$.

  \item  {\bf j4} is an axiom of $\JL_\CS$, $\ww\in\E_{i+1}(r,F)$, $\vv\RR \ww$, $\vv\in\E_{i}(r,F)$.  Thus, by the induction hypothesis, $\vv E(r,F)$ is in $\overline{\g}$. By definition of $\RR$ for countermodel $\M$, $\vv R \ww$ is in $\overline{\g}$. Therefore, $\vv E(r,F)$ and $\vv R \ww$ are in $\g_k$, for some $k\geq 0$. After the stage of rule $(Mon)$, evidence atom $\ww E(r,F)$ is added to the antecedent of the sequent. Hence $\ww E(r,F)$ is in $\g_l$, for some $l\geq k$. Therefore, $\ww E(r,F)$ is in $\overline{\g}$.
 \qe
  \end{enumerate}
 \end{proof}

\begin{lemma}\label{lemma: in completeness countermodel}
Given an arbitrary $\JL$-formula $A$, if $\ww\Vvdash A$ is in
$\overline{\g}$, then $(\M,\ww)\Vdash A$, and if $\ww\Vvdash A$ is in $\overline{\d}$, then $(\M,\ww)\not\Vdash  A$.
\end{lemma}
\begin{proof} By induction on the complexity of $A$. The base case follows from the definition of $\V$. We consider in details
only the case that $A$ is of the form $t:B$ (the proof of other cases
are stated in details in \cite{Negri2009}).

Suppose $\ww\Vvdash t:B$ is in $\overline{\g}$. Then, there is some
$i$ such that $\ww\Vvdash t:B$ appears first in $\g_i$. Therefore,
after the stage of rule $(E)$, $\ww E(t,B)$ is added to $\g_j$ for some $j>i$. Hence, by definition of evidence function $\mathcal{A}$ for
countermodel $\M$, we have $\ww\in\mathcal{A}(t,B)$, and hence $\ww\in\E_\mathcal{A}(t,B)$. Now consider an
arbitrary world $\vv$ in $\W$ such that $\ww\RR \vv$. By definition
of accessibility relation $\RR$ for $\M$, the relational atom $\ww R\vv$ should be in
$\g_l$, for some $l$. Then we can find $k$ such that both $\ww R\vv$
and $\ww\Vvdash t:B$ are in $\g_k$. Hence, in the stage of rule $(L:)$, we add $\vv\Vvdash B$ to $\overline{\g}$. Thus,
by the induction hypothesis, we can conclude that $(\M,\vv)\Vdash
B$. Therefore, $(\M,\ww)\Vdash t:B$.

If $\ww\Vvdash t:B$ is in $\overline{\d}$. Then $\ww\Vvdash t:B$ is in
$\d_i$, for some $i$. We have two cases:
 \begin{description}
  \item[(i)] $\ww E(t,B)$ is in $\overline{\g}$. Then $\ww E(t,B)$ is in $\g_j$, for some $j$. Find minimum index $k$ such
that $\g_k\R\d_k$ is of the form $\ww E(t,B),\g'\R\d',\ww\Vvdash
t:B$. After the stage of rule $(R:)$, we obtain a sequent
$\g_l\R\d_l$ of the form $\ww R\vv,\ww E(t,B),\g''\R\d'',\vv\Vvdash B$, for some
fresh label $\vv$ and $l> k$. Thus, by the induction hypothesis,
$(\M,\vv)\not\Vdash B$ and, by the definition of $\RR$, $\ww\RR \vv$. Hence, $(\M,\ww)\not\Vdash
t:B$.
    \item[(ii)] $\ww E(t,B)$ is not in $\overline{\g}$. Since $\ww\Vvdash t:B$ is in $\overline{\d}$, by the labeled-subformula property for $\J^-_\CS$, we have $t:B\in Sub_{Fm}(\g\R\d)$, and hence $t\in Sub_{Tm}(\g\R\d)$.  By Lemma \ref{lem:E in completeness countermodel}, we have $\ww\not\in\E_\mathcal{A}(t,B)$. Hence $(\M,\ww)\not\Vdash t:B$. \qe
  \end{description}
  \end{proof}

 Finally, let $\M$-interpretation $[\cdot]$ be the
identity function. Thus, $[\ww]=\ww$ for every $\ww\in\W$. We have the following facts:

\begin{enumerate}
\item Lemma \ref{lemma: in completeness countermodel},  together with $\g\subseteq \overline{\g}$ and $\d\subseteq\overline{\d}$ implies that: given an arbitrary $\JL$-formula $A$, if $\ww\Vvdash A$ is in $\g$, then $(\M,[\ww])\Vdash A$, and if $\ww\Vvdash A$ is in $\d$, then $(\M,[\ww])\not\Vdash  A$.

\item By the definition of $\RR$, if $\ww R\vv$ is in $\g$, then $[\ww]\RR [\vv]$. And if $\ww R\vv$ is in $\d$, then $[\ww] \RR [\vv]$ does not hold.

\item  If $\ww E(r,F)$ is in $\g$, then it is in $\overline{\g}$. Therefore, $\ww \in\mathcal{A}(r,F)$, and hence $[\ww] \in\E_\mathcal{A}(r,F)$. And if $\ww E(r,F)$ is in $\d$, then it is in $\overline{\d}$. Therefore, $r\in Sub_{Tm}(\g\R\d)$, since no stages of the reduction tree produces evidence atoms in the succedent of  sequents. Hence, by Lemma \ref{lem:E in completeness countermodel},  $[\ww] \not\in\E_\mathcal{A}(r,F)$.
\end{enumerate}

These facts together imply that the interpretation
$[\cdot]$ validates all the formulas in $\g$ but none of the
formulas in $\d$. Thus $\M$ is a $\JL_\CS$-countermodel for
$\g\R\d$.\qe \end{proof}

Completeness of labeled sequent calculi $\J^-_\CS$ with respect to F-models follows from the above theorem.

\begin{corollary}[Completeness]\label{corolary label compl}
Let $\JL$ be one of the justification logics {\sf J}, {\sf J4}, {\sf JD}, {\sf JD4}, {\sf JT}, {\sf LP}, and $\CS$ be a finite constant specification for \JL. If sequent $\g\R\d$
is valid in every $\JL_\CS$-model, then it is
derivable in $\J_\CS$.
\end{corollary}

For these labeled systems, if the reduction tree of a sequent  has a saturated branch, then the sequent has a countermodel, and hence, by Theorem \ref{Soundness labeled systems}, it is not derivable.

\begin{corollary}\label{cor:saturated branch}
Let $\JL$ be one of the justification logics {\sf J}, {\sf J4}, {\sf JD}, {\sf JD4}, {\sf JT}, {\sf LP}, and $\CS$ be a finite constant specification for \JL. If the reduction tree of a sequent in $\J_\CS$ has a saturated branch then the sequent is not derivable in $\J_\CS$.
 \end{corollary}

 In the next section, using the reduction tree and above corollary, we show how to decide whether a sequent is derivable. Let us now give some examples. The first example shows that the reduction tree of each sequent (except the empty sequent) in ${\sf G3JD}_\CS$ and its extensions (always) contains an infinite saturated branch.

 \begin{example}
Let us try to construct a proof tree for the sequent $\R \ww\Vvdash P$ in ${\sf G3JD}_\emptyset$, where $P$ is a propositional variable.
\begin{eqnarray*}
&\vdots&\\
&\uparrow&\\
\ww R \vv_3,\vv_1 R \vv_3,\vv_2 R \vv_3,\ww R \vv_2,\vv_1 R \vv_2, \ww R \vv_1 &\R& \ww\Vvdash P\\
&\uparrow&\\
\ww R \vv_2,\vv_1 R \vv_2, \ww R \vv_1 &\R& \ww\Vvdash P\\
&\uparrow& \\
 \ww R \vv_1&\R& \ww\Vvdash P\\
 &\uparrow&\\
&\R& \ww\Vvdash P
\end{eqnarray*}
where all sequents (except the root sequent) are obtained by stages of rule $(Ser)$, and $\vv_1, \vv_2,\ldots$ are distinct labels different from $\ww$. Since the branch is infinite, it is saturated. Thus, we can construct a ${\sf JD}_\emptyset$-countermodel $\M=(\W,\RR,\E,\V)$ for the root sequent as follows:
\begin{itemize}
\item $\W=\{\ww,\vv_1, \vv_2,\ldots\}$.
\item $\RR=\{(\ww,\vv_i)~|~i\geq 1\}\cup\{(\vv_i,\vv_j)~|~i<j\}$.
\item $\E(t,A)=\emptyset$, for any term $t$ and any formula $A$.
\item $\V(Q)=\emptyset$, for any propositional variable $Q$.
\end{itemize}
Obviously, $\RR$ is serial and $\E$ satisfies application $(\E1)$ and sum $(\E2)$ conditions. Moreover, by Lemma \ref{lemma: in completeness countermodel}, we have $(\M,\ww)\not\Vdash P$. By letting $[\cdot]$ be the identity $\M$-interpretation, we conclude that $\M$ is a ${\sf JD}_\emptyset$-countermodel for $\R \ww\Vvdash P$.
\end{example}

\begin{example}
Let us try to construct a proof tree for a justification version of
L\"{o}b principle, $x:(y:A\r A)\r z:A$, in ${\sf G3J4}_\emptyset$. Here, to have a
simple reduction tree we replace occurrences of $\b$ in the
L\"{o}b principle $\b(\b A\r A)\r \b A$ by
justification variables $x,y$ and $z$, respectively.
\begin{eqnarray*}
\ww E(x,y:A\r A), \ww\Vvdash x:(y:A\r A)&\R& \ww\Vvdash z:A\\
&\uparrow&\\
 \ww\Vvdash x:(y:A\r A)&\R& \ww\Vvdash z:A\\
 &\uparrow&\\
&\R& \ww\Vvdash x:(y:A\r A)\r z:A
\end{eqnarray*}
The second and third sequent in the reduction tree are obtained
from the root sequent by applying $(R\r)$ and $(E)$, respectively.
Since no other rules can be applied on the third sequent, and it is not an initial sequent, the branch is saturated. Thus, we can construct a ${\sf J4}_\emptyset$-countermodel $\M$ for the root sequent. Let $\W=\{\ww\}$ and $\RR$ be the empty set. For the evidence function $\E$, we have $\ww\in\mathcal{A}(x,y:A\r A)$. Now let $\E_\mathcal{A}$ be the evidence function based on $\mathcal{A}$. Finally let $\V(P)=\emptyset$, for any propositional variable $P$.  By Lemma \ref{lemma: in completeness countermodel} we obtain $(\M,\ww)\Vdash x:(y:A\r A)$ and
$(\M,\ww)\not\Vdash z:A$ (these facts can also be proved directly using the model $\M$). Thus $(\M,\ww)\not\Vdash x:(y:A\r A)\r z:A$.
\end{example}

We close this section with a discussion on the correspondence theory for justification logics. Following the correspondence theory of modal logics, one can expect that justification axioms defines classes of suitable frames. For simplicity, in the rest of this section, we assume that the language of all justification logics contain all term operations $+$, $\cdot$, $!$, $\bar{?}$, $?$. The set of all terms and formulas are denoted by $Tm$ and $Fm$ respectively.

\begin{definition}
\begin{enumerate}
\item A Fitting frame is a triple $\mathcal{F}=(\W,\RR,\E)$ such that $\W$ is a non-empty set, $\RR$ is the accessibility relation on $\W$, and $\E$ is a possible evidence function on $\W$, i.e. $\E:Tm \times Fm\r 2^\W$.
\item A possible Fitting model is an Fitting frame enriched by a valuation $\V$. For a possible Fitting model $\mathcal{M}=(\W,\RR,\mathcal{A},\V)$, where $\mathcal{F}=(\W,\RR,\mathcal{A})$ is a Fitting frame, we say that $\M$ is based on $\mathcal{F}$.
\item A \JL-formula $A$ is true in a possible Fitting model $\mathcal{M}$, denoted $\mathcal{M}\Vdash A$, if $(\mathcal{M},w)\Vdash A$ for every $w\in\W$.
\item A \JL-formula $A$ characterizes a class of possible Fitting models $\mathsf{M}$, if for all possible Fitting models $\mathcal{M}$: $$\mathcal{M}\in\mathsf{M}~\textrm{iff}~ \mathcal{M}\Vdash A.$$
\item A \JL-formula $A$ is valid in a world $w$ in a Fitting frame $\mathcal{F}$, denoted $(\mathcal{F},w)\Vdash A$, if $(\M,w)\Vdash A$ for every possible Fitting model $\M=(\mathcal{F},\V)$ based on $\mathcal{F}$.
\item A \JL-formula $A$ is valid in a Fitting frame $\mathcal{F}$, denoted $\mathcal{F}\Vdash A$, if $(\mathcal{F},w)\Vdash A$ for every $w\in\W$.
\item A \JL-formula $A$ characterizes a class of Fitting frames $\mathsf{F}$, if for all Fitting frames $\mathcal{F}$: $$\mathcal{F}\in\mathsf{F}~\textrm{iff}~ \mathcal{F}\Vdash A.$$
\end{enumerate}
\end{definition}
Soundness and completeness of justification logics with respect to F-models (Theorem \ref{Sound Compl JL}) naturally propose the following \textit{naive characterizations} (cf. also Table \ref{table: frame properties of models}):
\begin{enumerate}
\item $x:P\r P$ characterizes the class of Fitting frames in which the accessibility relation is  reflexive.
\item $x:\bot\r \bot$ characterizes the class of Fitting frames in which the accessibility relation is serial.
\item $x:P\r !x:x:P$ characterizes the class of Fitting frames in which the accessibility relation is transitive and the evidence function satisfies the monotonicity $(\E3)$ and positive introspection  $(\E4)$ conditions.
\item $\neg P\r \bar{?}x:\neg x:P$ characterizes the class of possible Fitting models in which the accessibility relation is symmetric and the evidence function satisfies the weak negative introspection condition $(\E5)$.
\item $x:P\r ?x:\neg x:P$ characterizes the class of possible Fitting models in which the evidence function satisfies the negative introspection $(\E6)$ and strong evidence  $(\E7)$ conditions.
\end{enumerate}

The above characterizations are used implicitly in the literature of justification logics, but no proofs have been found for them so far. In fact we can expect problems in proving of the above naive characterizations. For example see the inconvenient form of items 6 and 7, which are formulated in terms of possible Fitting models instead of Fitting frames. In addition, the behavior of evidence functions in Fitting models is similar to valuation functions. Consider for example the evidence function $\E:Tm \times Fm \r 2^\W$ and the valuation function $\V:\mathcal{P}\r 2^\W$, where $\mathcal{P}$ denotes the set of all propositional variables. $\V(P)$, for $P\in\mathcal{P}$, is the set of all worlds in which $P$ is true. Likewise $\E(t,A)$, for a term $t$ and a formula $A$, is the set of all worlds where $t$ is considered admissible evidence for $A$. Thus, it may be more convenient to not incorporate evidence functions into Fitting \textit{frames}.

In the following theorem we give a model correspondence in the context of labeled systems, which is different from the above  characterizations.

\begin{theorem}[Fitting Model Correspondence]\label{thm:Fitting Model Correspondence}
 Suppose $x$ is a justification variable and $P$ is a propositional variable.
\begin{enumerate}
\item The sequent $\R\ww\Vvdash x:P\r P$ is derivable in ${\sf G3JL}_\CS$ if and only if ${\sf G3JL}_\CS$ contains the rule $(Ref)$.
\item The sequent $\R\ww\Vvdash x:\bot\r \bot$ is derivable in ${\sf G3JL}_\CS$ if and only if ${\sf G3JL}_\CS$ contains the rule $(Ser)$.
\item The sequent $\R\ww\Vvdash x:P\r !x:x:P$ is derivable in ${\sf G3JL}_\CS$ if and only if ${\sf G3JL}_\CS$ contains the rules  $(E!), (Mon), (Trans)$.
\item The sequent $\R\ww\Vvdash \neg P\r \bar{?}x:\neg x:P$ is derivable in ${\sf G3JL}_\CS$ if and only if ${\sf G3JL}_\CS$ contains the rules $(E\bar{?}), (Sym)$.
\item The sequent $\R\ww\Vvdash \neg x:P\r ?x:\neg x:P$ is derivable in ${\sf G3JL}_\CS$ if and only if ${\sf G3JL}_\CS$ contains the rules $(SE), (E?)$.
\end{enumerate}
\end{theorem}
\begin{proof}
The if directions are easy. Let us prove the only if directions, or rather their contrapositives. For item (1), suppose that ${\sf G3JL}_\emptyset$ does not contain rule $(Ref)$. If ${\sf G3JL}_\emptyset$ does not contain rules  $(Ser)$, $(E?)$, and $(E\bar{?})$. Then the reduction tree of $\R\ww\Vvdash x:P\r P$ in ${\sf G3JL}_\emptyset$ is as follows:
\begin{eqnarray*}
\ww E(x,P),\ww\Vvdash x:P &\R& \ww\Vvdash P\\
 &\uparrow&\\
 \ww\Vvdash x:P &\R& \ww\Vvdash P\\
 &\uparrow&\\
&\R& \ww\Vvdash x:P\r P
\end{eqnarray*}
where the second and third sequents are obtained by the stages of rules $(R\r)$ and $(E)$ respectively. Since this branch is saturated, by Corollary \ref{cor:saturated branch}, the sequent $\R\ww\Vvdash x:P\r P$ is not derivable in ${\sf G3JL}_\emptyset$.  For labeled systems that contain $(Ser)$ or $(E?)$, but does not contain $(E\bar{?})$, we obtain an infinite saturated branch, which again shows that the sequent is not derivable. If ${\sf G3JL}_\emptyset$  contains the rules $(E\bar{?})$, then it is easy to verify that  applying any of rules of ${\sf G3JL}_\emptyset$ upwardly, in any step of the previous argument do not yield to an initial sequent. Thus again $\R\ww\Vvdash x:P\r P$ is not derivable in ${\sf G3JL}_\emptyset$. The proof for other clauses, except (4), are similar.

For clause (4), Suppose ${\sf G3JL}_\emptyset$ does not contain rules $(E\bar{?})$ and $(Sym)$. If we try to find a derivation for the sequent $\R\ww\Vvdash \neg P\r \bar{?}x:\neg x:P$ in ${\sf G3JL}_\emptyset$, which does not contain rules $(Ref)$, $(Ser)$, $(E?)$, we see that this sequent only fit into the conclusion of rule $(R\r)$:
\begin{prooftree}
\AXC{$\ww\Vvdash \neg P\R\ww\Vvdash \bar{?}x:\neg x:P$}
\RightLabel{$(R\r)$}
\UIC{$\R\ww\Vvdash \neg P\r \bar{?}x:\neg x:P$}
\end{prooftree}
and the sequent $\ww\Vvdash \neg P\R\ww\Vvdash \bar{?}x:\neg x:P$ only fit into the conclusion of rule $(L\neg)$:
\begin{prooftree}
\AXC{$\R\ww\Vvdash \bar{?}x:\neg x:P,\ww\Vvdash  P$}
\RightLabel{$(L\neg)$}
\UIC{$\ww\Vvdash \neg P\R\ww\Vvdash \bar{?}x:\neg x:P$}
\RightLabel{$(R\r)$}
\UIC{$\R\ww\Vvdash \neg P\r \bar{?}x:\neg x:P$}
\end{prooftree}
Now the the sequent $\R\ww\Vvdash \bar{?}x:\neg x:P,\ww\Vvdash  P$ do not fit into the conclusion of any rules. Hence $\R\ww\Vvdash \neg P\r \bar{?}x:\neg x:P$ is not derivable in ${\sf G3JL}_\emptyset$. If ${\sf G3JL}_\emptyset$  contains the rules $(Ref)$, $(Ser)$, $(E?)$, then it is easy to verify that  applying any of rules $(Ref)$, $(Ser)$, $(E?)$, upwardly, in any step of the previous argument do not yield to an initial sequent. Thus again $\R\ww\Vvdash \neg P\r \bar{?}x:\neg x:P$ is not derivable in ${\sf G3JL}_\emptyset$.

The above results can be extended to the case that $\CS$ is not empty (we leave it to the reader to verify the details of this).\qe
\end{proof}

The kind of correspondences that our labeled sequent calculi  gives us is different from the aforementioned naive characterizations. Firstly, it is indeed formulated using Fitting models rather than Fitting frames (and thus it is better to call it Fitting Model Correspondence instead of Fitting Frame Correspondence, as we did). Let us compare item 1 of the naive characterization with item 1 of the above theorem. The former informally says: a Fitting \textit{frame} validates $x:P\r P$ if{f} it is reflexive. Whereas  item 1 of the above theorem informally says (using Corollary \ref{cor:equivalence JL and G3JL}): $x:P\r P$ is true in all $\JL$-\textit{models} if{f} ${\sf G3JL}$ contains $(Ref)$, meaning that all $\JL$-\textit{models} are reflexive.

Thus, it seems that the only correspondences of Theorem \ref{thm:Fitting Model Correspondence} which describe the naive characterizations in the context of labeled sequent calculus appropriately are items 4 and 5, for axioms {\bf jB} and {\bf j5} respectively. However, there is another drawback in these formulations. All correspondences of Theorem \ref{thm:Fitting Model Correspondence} are stated on all models rather than a particular model. For example, let us compare item 5 of the naive characterization with item 5 of the above theorem. The former informally says: $\neg x:P\r ?x:\neg x:P$ is true in \textit{a given} Fitting model   if{f} the evidence function of the model satisfies the negative introspection $(\E6)$ and strong evidence $(\E7)$ conditions. Whereas  item 5 of the above theorem informally says (again using Corollary \ref{cor:equivalence JL and G3JL}): $\neg x:P\r ?x:\neg x:P$ is true in \textit{all} $\JL$-models if{f} ${\sf G3JL}$ contains $(SE)$ and $(E?)$, meaning that the evidence function of \textit{all} $\JL$-models satisfies the conditions $\E6$ and  $\E7$.

Nonetheless, it seems that Theorem \ref{thm:Fitting Model Correspondence} is the only result (at least to the best of my knowledge) on the correspondence theory in justification logics and the issue have yet to be investigated more.
\section{Termination of proof search}\label{section: termination proof search}
In this section, we establish the termination of proof search for
labeled systems ${\sf G3J}_\CS$, ${\sf G3JT}_\CS$, ${\sf G3LP}_\CS$, for finite $\CS$. In this respect, it is
useful to consider the reduction tree of a sequent, which was constructed in the
proof of Theorem \ref{thm:reduction tree}.\footnote{In \cite{Ghari2012-Thesis} termination of proof search for these systems are proved using \textit{minimal} derivations (that was first defined by Negri in \cite{Negri2005}).} We show that the reduction tree of every sequent in ${\sf G3J}_\CS$, ${\sf G3JT}_\CS$, ${\sf G3LP}_\CS$, for finite $\CS$, is finite. Thus, it is decidable whether the reduction tree has a saturated branch, and hence it is decidable whether the sequent is derivable. In order to show that the reduction tree of a sequent is finite, we find bounds on
the number of applications of rules.

First we define the negative and positive parts of a sequent
$\g\R\d$. Let
\[ \g^f=\{A~|~\textrm{$\ww\Vvdash A$ occurs in $\g$}\}.\]
Now the negative and positive parts of a sequent $\g\R\d$ is defined as the negative and positive parts of the formula $\bigwedge\g^f\r\bigvee\d^f$. We use the
following notations in the remaining of this section. For any
given sequent, let $n(:)$ and $p(:)$ be the number of occurrences
of : in the negative and positive part of the
sequent respectively. For any given derivation
\begin{itemize}
    \item   $l$ denotes the number of labels in the endsequent.
    \item   $r$ denotes the number of relational atoms in the antecedent of the endsequent.
    \item   $e$ denotes the number of evidence atoms in the antecedent of the endsequent.
    \item   For $*\in\{+,\cdot,!\}$, $n(*)$ denotes the number of subterms of terms that contain $*$ in the endsequent.
\end{itemize}
It should be noted that, since we use the reduction tree from the
proof of Theorem \ref{thm:reduction tree}, the
termination of proof search is proved for finite constant
specifications.
\begin{theorem}\label{thm:termination proof search G3J}
Given any finite constant specification $\CS$, and any sequent
$\g\R\d$ in the language of {\sf G3J}, it is decidable whether the sequent is derivable in ${\sf G3J}_\CS$.
\end{theorem}
\begin{proof} Suppose $\g\R\d$ is any sequent to be shown derivable.
Construct the reduction tree with the root $\g\R\d$. In the reduction tree the stages of propositional rules reduce the complexity of
formulas in the sequents (propositional rules have premises in which the active formulas are strictly simpler than the principal formula). In fact, in each branch of the reduction tree the number of applications of propositional rules are bounded by the number of the corresponding connective in the endsequent. Rule $(R:)$ decreases the complexity of
its principal labeled formula but adds a relational atom. The number of
applications of this rule (and therefore the number of eigenlabels
introducing by this rule) is bounded by $p(:)$. Rule $(L:)$ adds a
new labeled formula to the sequent. Regarding condition
$(\dagger)$, rule $(L:)$ in the reduction tree could apply only
once on a pair of formulas $\ww\Vvdash t:A$ and $\ww R\vv$. Thus the
number of applications of $(L:)$ with principal formula $\ww\Vvdash
t:A$ is bounded by the number of relational atoms of the form
$\ww R\vv$, which can be found in the antecedent of the root sequent or
may be introduced by rule $(R:)$ in the antecedent of sequents.
Hence the number of applications of $(L:)$ is bounded by
$n(:)(p(:)+r)$.

Rules $(E)$, $(IAN)$, $(El+)$, $(Er+)$ and $(E\cdot)$ add one new
evidence atom to the sequent, and increase the size of the
sequent. Thus, we have to find bounds on the number of
applications of each one. For $(E)$, since moving upward no rule
omit any evidence atom from the antecedent of sequents, the number
of applications of $(E)$ is bounded by $n(:)$. For $(IAN)$, since constant specification $\CS$ is finite, by the sublabel property, the number of applications of $(IAN)$ is
bounded by $|\CS|(p(:)+l)$ (by $|\CS|$ we mean the number of
elements of $\CS$). By the construction of the reduction tree, we
add an evidence atom $\ww E(t+s,A)$ to the antecedent of the
topsequent provided that $t+s\in Sub_{Tm}(\g\R\d)$.
Therefore, the number of applications of $(El+)$ (and similarly
$(Er+)$) is bounded by the number of evidence atoms that are in
the antecedent of the root sequent or may be introduced by the
rules $(E)$ or $(IAN)$. Hence, the number of applications of
$(El+)$ or $(Er+)$ is bounded by $n(+)(e+n(:)+|\CS|(p(:)+l))$. By
a similar argument, we get a bound on the number of applications
of rule $(E\cdot)$ by $n(\cdot)(e+n(:)+|\CS|(p(:)+l))$.\qe \end{proof}

Note that, as the above proof shows, the number of applications of
$(L:)$ and $(IAN)$ (or $(AN)$) depends on the number of
applications of $(R:)$ and the number of applications of $(El+)$,
$(Er+)$, and $(E\cdot)$ depends on the number of applications of $(E)$ and
$(IAN)$.
\begin{theorem}\label{thm:termination proof search G3JT}
Given any finite constant specification $\CS$, and any sequent
$\g\R\d$ in the language of {\sf G3JT}, it is decidable whether the sequent is derivable in ${\sf G3JT}_\CS$.
\end{theorem}
\begin{proof} Rule $(Ref)$ add a relational atom to the sequent. By the
construction of the reduction tree, $(Ref)$ only introduces atoms
$\ww R\ww$ in which $\ww$ is a label in the root sequent or is an eigenlabel. Thus, the
number of applications of $(Ref)$ is bounded by $l+p(:)$. Since the
atoms $\ww R\ww$ introduced by $(Ref)$, may produce new applications of
the rule $(L:)$, in this case, the number of applications of
$(L:)$ is bounded by $n(:)(2p(:)+r+l)$. The other bounds remain
unchanged.\qe \end{proof}

Rules $(Trans)$ and $(Mon)$ in {\sf G3LP} may
produce infinitely many applications of  rules $(L:)$ and
$(R:)$. For example, suppose we try to find a derivation
for sequent $\vv R\ww,\ww E(t,A)\R \vv\Vvdash \neg s:\neg~ t:A$ in {\sf
G3LP}:
\begin{prooftree}
\def\extraVskip{3pt}
\AXC{$\ww R\ww_1 ,\vv R\ww,\ww E(t,A),\vv\Vvdash s:\neg t:A\R \ww_1 \Vvdash
A$}\RightLabel{$(R:)$}
 \UIC{$\vv R\ww,\ww E(t,A),\vv\Vvdash s:\neg t:A\R \ww\Vvdash t:A$}
 \RightLabel{$(L\neg)$}
  \UIC{$\ww\Vvdash \neg t:A,\vv R\ww,\ww E(t,A),\vv\Vvdash s:\neg t:A\R$}
  \RightLabel{$(L:)$}
  \UIC{$\vv R\ww,\ww E(t,A),\vv\Vvdash s:\neg t:A\R$}
  \end{prooftree}
   Now, by applying $(Trans)$ and $(Mon)$ upwardly on the topsequent  we have:
\begin{equation}\label{sequent 1}
\vv R\ww_1 ,\ww_1  E(t,A),\ww R\ww_1 ,\vv R\ww,\ww E(t,A),\vv\Vvdash s:\neg t:A\R
\ww_1 \Vvdash A.
\end{equation}
  Next, by applying $(L:)$ and then $(L\neg)$ upwardly we obtain:
\begin{equation}\label{sequent 2}
\vv R\ww_1 ,\ww_1  E(t,A),\ww R\ww_1 ,\vv R\ww,\ww E(t,A),\vv\Vvdash s:\neg t:A\R
\ww_1 \Vvdash A,\ww_1 \Vvdash t:A.
\end{equation}
   Again, by applying $(R:)$ upwardly we have:
\begin{equation}\label{sequent 3}
\ww_1 R\ww_2,\vv R\ww_1 ,\ww_1  E(t,A),\ww R\ww_1,\vv R\ww,\ww E(t,A),\vv\Vvdash s:\neg
t:A\R \ww_1 \Vvdash A,\ww_2\Vvdash A.
\end{equation}
We can repeat the steps (\ref{sequent 1})-(\ref{sequent 3}) and
the applications of $(R:)$ infinitely  many times. But these stages are redundant. Indeed, steps
(\ref{sequent 1})-(\ref{sequent 3}) can be shortened in the
following way. By applying the substitution $(\ww_1 /\ww_2)$ on sequent
(\ref{sequent 3}), we obtain  a derivation of the same height of
\[\ww_1 R\ww_1 ,\vv R\ww_1 ,\ww_1  E(t,A),\ww R\ww_1 ,\vv R\ww,\ww E(t,A),\vv\Vvdash s:\neg t:A\R \ww_1 \Vvdash A,\ww_1 \Vvdash A.\]
Then, by height-preserving contraction and $(Ref)$ we obtain a
derivation of
\[\vv R\ww_1 ,\ww_1  E(t,A),wR\ww_1 ,\vv R\ww,\ww E(t,A), s:\neg t:A\R \ww_1 \Vvdash A.\]
But this sequent is the same as sequent (\ref{sequent 1}), which
has been obtained in two steps shorter than (\ref{sequent 1}).
In fact, we can bound the number of applications of $(R:)$ in derivations, as follows:
\begin{proposition}\label{bound R: in LP}
In a derivation of a sequent in ${\sf G3LP}_\CS$ for each formula of the form $t:A$ in the positive part of the sequent, it is enough to have at most $n(:)$ applications of $(R:)$ iterated on
a chain of accessible worlds $\ww R\ww_1 ,\ww_1 R\ww_2,\ldots$, with
principal formula $\ww_i\Vvdash t:A$.
\end{proposition}
\begin{proof} Let $n(:)=m$, and assume that $t:A$ is a formula in the
positive part of the endsequent of the derivation. Consider the
worst case, in which all of the $m$ negative occurrences of $:$
appear in a block in a formula of the form $\ww\Vvdash
t_1:t_2:\ldots:t_m:B$ in the antecedent of the endsequent (in
which $B$ contains no negative occurrence of $:$). After the first
application of $(R:)$ on $t:A$, we have the accessibility atom
$\ww R\ww_1 $ in the antecedent. Then an application of $(L:)$ produces
the formula $\ww_1 \Vvdash t_2:\ldots:t_m:B$. After the second
application of $(R:)$, we have the new accessibility atom
$\ww_1 R\ww_2$, and by $(Trans)$ we get $\ww R \ww_2$. Then applications of
$(L:)$ add to the antecedent the formulas $\ww_2\Vvdash
t_2:\ldots:t_m:B$ and $\ww_2\Vvdash t_3:\ldots:t_m:B$. After $m$
applications of $(R:)$, the antecedent contains in addition:
\[\ww_m \Vvdash t_2:\ldots:t_m:B, \ww_m \Vvdash t_3:\ldots:t_m:B,
\ldots, \ww_m \Vvdash B.\]
 If we apply $(R:)$ one more time, and then apply
 $(L:)$, we have also in the antecedent the formulas
\[\ww_{m+1}\Vvdash t_2:\ldots:t_m:B, \ww_{m+1}\Vvdash t_3:\ldots:t_m:B,
\ldots, \ww_{m+1}\Vvdash B.\]
 Similar to the example discussed before the proposition, by the substitution $(\ww_{m}/\ww_{m+1})$
and then applying height-preserving contraction and $(Ref)$, we
can shorten the derivation. Thus, the last step of $(R:)$
is superfluous.\qe \end{proof}

Thus, if $n(:)>0$ then the number of applications of $(R:)$ is
bounded by $p(:)n(:)$, and if $n(:)=0$ then the number of
applications of $(R:)$ is bounded by $p(:)$. Hence, in general, the number of
applications of $(R:)$ is bounded by $p(:)(n(:)+1)$.

Regarding Proposition \ref{bound R: in LP}, in the construction of the reduction tree in Theorem
\ref{thm:reduction tree}, we need to restrict the
number of applications of $(R:)$ on each subformula
$t:A$ in the positive part of the root sequent. To this end, let $p(:)=p$ and $t_1:A_1,\ldots,t_p:A_p$
be all subformulas of the form $t:A$ in the positive part of the
endsequent $\g\R\d$. In order to calculate the number of
applications of $(R:)$, we employ counters
$k_1,\ldots,k_p$ for the formulas $t_1:A_1,\ldots,t_p:A_p$,
respectively. In stage $n=0$ of the construction of the reduction
tree, let $k_i=0$ for each $i=1,\ldots,p$. Each time we apply rule $(R:)$ on the formula $t_i:A_i$, we update the
value of $k_i$ by one, i.e. $k_i=k_i+1$. Rule $(R:)$
can be applied on $t_i:A_i$ provided that $k_i\leq n(:)$.

\begin{theorem}\label{terminating proof search LP}
Given any finite constant specification $\CS$, and any sequent
$\g\R\d$ in the language of {\sf G3LP}, it is decidable whether the sequent is derivable in ${\sf G3LP}_\CS$.
\end{theorem}
\begin{proof} Since the number of applications of $(R:)$ is bounded by
$p(:)(n(:)+1)$, the number of applications of rules $(L:)$ and
$(AN)$ are bounded by $n(:)(p(:)(n(:)+1)+r+l+p(:))$ and
$|\CS|(p(:)(n(:)+1)+l)$, respectively. Rule $(Mon)$ may produce
new applications of rules $(El+)$, $(Er+)$, $(E\cdot)$ and $(E!)$,
and vise versa. However, the number of applications of these rules
are bounded, since there are a finite number of relational and
evidence atoms in the antecedent of the endsequent and the
reduction tree of a sequent complying condition $(\dagger)$. The
number of applications of $(Trans)$ depends on $r$ and the number
of applications of $(R:)$, and thus is bounded. The number of
applications of $(E)$ is bounded by $n(:)$ as before.\qe \end{proof}

\section{Other labeled systems}\label{sec:Other labeled systems}
In this section we will briefly introduce other variants of labeled sequent calculus for justification logics. We first give some admissible rules.

\begin{lemma}\label{lem:admissible rules in G3J5}
The following rules are \CS-admissible in ${\sf G3J5}_\CS$ and its extensions:
$$\di{\f{\g\R\d,\ww E(t,A)}{\g\R\d,\ww\Vvdash t:A}}(R1) \qquad \di{\f{\ww E(?t,\neg t:A),\g\R\d,\ww\Vvdash t:A}{\g\R\d,\ww\Vvdash t:A}}(R2)$$
$$ \di{\f{\ww\Vvdash t:A,\g\R\d}{\g\R\d,\ww\Vvdash ?t:\neg t:A}}(R3)\qquad \di{\f{\vv\Vvdash A,\ww E(t,A),\ww R\vv,\g\R\d}{\ww E(t,A),\ww R\vv,\g\R\d}}(R4)$$
\end{lemma}
\begin{proof} For $(R1)$, using admissible $Cut$, we have:
\begin{prooftree}
\def\extraVskip{3pt}
\AXC{$\g\R\d,\ww E(t,A)$}
\AXC{${\mathcal D}$} \noLine
  \UIC{$\ww\Vvdash t:A,\ww E(t,A)\R\ww\Vvdash t:A$}
  \RightLabel{$(SE)$}
  \UIC{$\ww E(t,A)\R\ww\Vvdash t:A$}
  \RightLabel{$Cut$}
  \BIC{$\g\R\d,\ww\Vvdash t:A$}
\end{prooftree}
where ${\mathcal D}$ is the standard derivation of $\ww\Vvdash t:A,\ww E(t,A)\R\ww\Vvdash t:A$ from Lemma \ref{lem:Initial axiom for arbitrary A}. F

or $(R2)$  we have:
\begin{prooftree}
\def\extraVskip{3pt}
\AXC{${\mathcal D}$} \noLine
  \UIC{$\ww\Vvdash t:A,\ww E(t,A),\g\R\d,\ww\Vvdash t:A$}
  \RightLabel{$(SE)$}
  \UIC{$\ww E(t,A),\g\R\d,\ww\Vvdash t:A$}
\AXC{$\ww E(?t,\neg t:A),\g\R\d,\ww\Vvdash t:A$}
  \RightLabel{$(E?)$}
  \BIC{$\g\R\d,\ww\Vvdash t:A$}
\end{prooftree}
where ${\mathcal D}$ is the standard derivation of $\ww\Vvdash t:A,\ww E(t,A),\g\R\d,\ww\Vvdash t:A$ from Lemma \ref{lem:Initial axiom for arbitrary A}. For $(R3)$, using admissible $Cut$, we have:
\begin{prooftree}
\def\extraVskip{3pt}
\AXC{${\mathcal D'}$} \noLine
\UIC{$\R\ww\Vvdash ?t:\neg t:A, \ww\Vvdash t:A$}
    \AXC{$\ww\Vvdash t:A,\g\R\d$}
  \RightLabel{$(Cut)$}
  \BIC{$\g\R\d,\ww\Vvdash ?t:\neg t:A$}
\end{prooftree}
where ${\mathcal D'}$ is the standard derivation of $\R\ww\Vvdash ?t:\neg t:A, \ww\Vvdash t:A$ in {\sf G3J5}. For $(R4)$, using admissible $(LW)$, we have:
\begin{prooftree}
\def\extraVskip{3pt}
\AXC{$\vv\Vvdash A,\ww E(t,A),\ww R\vv,\g\R\d$}
  \RightLabel{$(LW)$}
\UIC{$\ww\Vvdash t:A,\vv\Vvdash A,\ww E(t,A),\ww R\vv,\g\R\d$}
  \RightLabel{$(L:)$}
\UIC{$\ww\Vvdash t:A,\ww E(t,A),\ww R\vv,\g\R\d$}
  \RightLabel{$(SE)$}
  \UIC{$\ww E(t,A),\ww R\vv,\g\R\d$}
\end{prooftree}
 \qe \end{proof}
 \begin{lemma}
The following rule is \CS-admissible in ${\sf G3JB}_\CS$ and its extensions:
$$ \di{\f{\ww E(\bar{?}t,\neg t:A),\g\R\d,\ww\Vvdash A}{\g\R\d,\ww\Vvdash A}}(Ew\bar{?})$$
\end{lemma}
\begin{proof}
We have:
\begin{prooftree}
\def\extraVskip{3pt}
\AXC{${\mathcal D}$} \noLine
   \UIC{$\ww\Vvdash A,\g\R\d,\ww\Vvdash A$}
\AXC{$\ww E(\bar{?} t,\neg t:A),\g\R\d,\ww\Vvdash A$}
  \RightLabel{$(E\bar{?})$}
  \BIC{$\g\R\d,\ww\Vvdash A$}
\end{prooftree}
where ${\mathcal D}$ is the standard derivation of $\ww\Vvdash A,\g\R\d,\ww\Vvdash A$ from Lemma \ref{lem:Initial axiom for arbitrary A}.
 \qe \end{proof}

 Labeled sequent calculi {\sf G3JB} and its extensions can be formulated by rule $(Ew\bar{?})$ instead of $(E\bar{?})$. If ${\sf G3JL}_\CS$ contains $(E\bar{?})$, then let ${\sf G3JL}^{w}_\CS$ denote the resulting system where the rule $(E\bar{?})$ is replaced by the rule $(Ew\bar{?})$ (these systems were studied in \cite{Ghari2012-Thesis}). Although these systems enjoy the labeled-subformula property, we fail to show the  admissibility of $Cut$. Here are some counterexamples.
\begin{example}
The sequent $\R \ww\Vvdash A \r \bar{?} t:\neg t:\neg A$ can  only be proved by rule $Cut$  (e.g., a $Cut$ on the formula $\ww\Vvdash \neg\neg A$) in ${\sf G3JB}^{w}_\CS$ and its extensions.
 \end{example}
\begin{example}
The sequent $\R \ww E(x,A),\ww E(\bar{?} s,\neg s:x:A)$, where $x$ is a justification variable, can  only be proved by rule $Cut$ (e.g., a $Cut$ on the formula $\ww\Vvdash x: A$) in ${\sf G3JB}^{w}_\CS$ and its
extensions.
\end{example}
\subsection{Labeled sequent calculus based on Fp-models}
Pacuit in \cite{Pacut2005} show that the negative introspection axiom {\bf j5} can be characterized by the following conditions:
\begin{description}
   \item[$\E 9.$] \textit{Anti-monotonicity}: If $v\in\E(t,A)$ and $w\RR v$, then
  $w\in\E(t,A)$.
    \item[$\E 10.$]  \textit{Negative proof checker}: If $(\M,w)\Vdash\neg t:A$, then $w\in\E(?t,\neg t:A)$.
\end{description}
and Euclideanness of $\RR$.  We call these models \textit{Fp-models} (for more details cf. \cite{Pacut2005}).
\begin{theorem}(\cite{Pacut2005})
Let \JL~be {\sf J5} or one of its extensions, and $\CS$ be a constant specification for \JL. Then justification logics
$\JL_\CS$ are sound and complete with respect to their $\JL_\CS$-Fp-models.
\end{theorem}
 Regarding Fp-models, there is another formulation of labeled systems for {\sf J5} and its extensions by replacing the rules $(E?)$ and $(SE)$ with the following rules:
$$\di{\f{\ww E(?t,\neg t:A),\ww\Vvdash \neg t:A,\g\R\d}{\ww\Vvdash \neg t:A,\g\R\d}}(E?')\qquad \di{\f{\ww E(t,A),\vv E(t,A),\ww R\vv,\g\R\d}{\vv  E(t,A),\ww R\vv,\g\R\d}}(Anti\mbox{-}Mon)$$
$$\di{\f{\vv R\uu,\ww R\vv,\ww R \uu,\g\R\d}{\ww R\vv,\ww R\uu,\g\R\d}}(Eucl)$$
If ${\sf G3JL}_\CS$ contains $(E?)$, $(SE)$, then let ${\sf G3JL}^{Fp}_\CS$ (the labeled sequent calculus based on Fp-models) denote the resulting system where the rules $(E?)$, $(SE)$ are replaced by the rules $(E?')$, $(Anti$-$Mon)$, $(Eucl)$. 
Thus labeled systems ${\sf G3JL}^{Fp}_\CS$ are defined only for justification logics \JL~that contain axiom {\bf j5}. All the definitions of Sections 4-8 can be adopted for labeled systems ${\sf G3JL}^{Fp}_\CS$ (these systems first appeared in \cite{Ghari2012-Thesis}). Main properties of labeled systems ${\sf G3JL}^{Fp}_\CS$ are listed in the following theorem (the proof is similar to that of labeled systems based on F-models).
\begin{theorem}
Suppose \JL~is a justification logic with axiom {\bf j5}, \CS~is a constant specification for \JL, and ${\sf G3JL}^{Fp}_\CS$ is its labeled sequent calculus based on Fp-models.
\begin{enumerate}
\item All sequents of the form $\ww\Vvdash A,\g\R\d,\ww\Vvdash A$, with $A$ an arbitrary $\JL$-formula, are derivable in ${\sf G3JL}^{Fp}_\CS$.
\item  All labeled formulas in a derivation
in ${\sf G3JL}^{Fp}_\CS$  are labeled-subformulas of labeled formulas in
the endsequent.
\item  The rules of substitution $(Subs)$, and weakening are height-preserving \CS-admissible in ${\sf G3JL}^{Fp}_\CS$.
\item If the sequent $\g\R\d$ is derivable in ${\sf G3JL}^{Fp}_\CS$, then it is valid in every $\JL_\CS$-Fp-model.
\end{enumerate}
 \end{theorem}
It seems the rules of ${\sf G3JL}^{Fp}_\CS$ are not invertible, and the rules of contraction are not admissible. Although the systems ${\sf G3JL}^{Fp}_\CS$ enjoys the subformula property, we fail to show the  admissibility of $Cut$. Here are some counterexamples.
\begin{example}
It is not hard to verify that  in the system ${\sf G3JT5}^{Fp}_\CS$ (and its extensions)  the sequent $\R \ww\Vvdash\neg A\r ?t:\neg t: A$ can only be  proved by rule  $Cut$ (e.g., a $Cut$ on the formula $\ww\Vvdash t: A$).
\end{example}
\begin{example}
The sequent $\R \ww E(t,A),\ww E(?t,\neg t:A)$ can  only be proved by rule $Cut$  (e.g., a $Cut$ on the formula $\ww\Vvdash t:A$) in ${\sf G3J5}^{Fp}_\CS$ and its extensions.
\end{example}
With regard to the above facts for the completeness we have:
\begin{theorem}
Suppose \JL~is a justification logic with axiom {\bf j5}, \CS~is a constant specification for \JL, and ${\sf G3JL}^{Fp}_\CS$ is its labeled sequent calculus based on Fp-models. A \JL-formula~$A$ is provable in $\JL_\CS$ iff $\R\ww\Vvdash A$ is provable in ${\sf G3JL}^{Fp}_\CS+(LC)+(RC)+Cut$.
 \end{theorem}

\subsection{Labeled sequent calculus based on Fk-models}
Models for {\sf JD} and its extensions can be also characterized by the class of F-models with the following condition instead of seriality of $\RR$:
\begin{description}
  \item[$\E 8.$] \textit{Consistent evidence}: $\E(t,\bot) =\emptyset$, for all terms $t$.
\end{description}
These models are introduced by Kuznets in \cite{Kuznets2008}, and were
called Fk-models there. Completeness of {\sf JD} (and its extensions) with respect to Fk-models can be proved without the requirement that constant specifications should be axiomatically appropriate.
\begin{theorem}(\cite{Kuznets2008})
Let \JL~be {\sf JD} or one of its extensions, and $\CS$ be a constant specification for \JL. Then justification logics
$\JL_\CS$ are sound and complete with respect to their $\JL_\CS$-Fk-models.
\end{theorem}
Regarding Fk-models, there is another formulation of labeled systems for {\sf JD} and
its extensions, by replacing the rule for seriality $(Ser)$ with the following initial sequent $(AxE\bot)$:
$$\ww E(t,\bot),\g\R\d~~~~~ (AxE \bot).$$
If ${\sf G3JL}_\CS$ contains $(Ser)$, then let ${\sf G3JL}^{Fk}_\CS$ (the labeled sequent calculus based on Fk-models) denote the resulting system where the rule $(Ser)$ is replaced by the initial sequent $(AxE \bot)$. All the definitions of Sections 4-8 can be adopted for labeled systems ${\sf G3JL}^{Fk}_\CS$ (these systems  first appeared in \cite{Ghari2012-Thesis}). Main properties of labeled systems ${\sf G3JL}^{Fk}_\CS$ are listed in the following theorem (the proof is similar to that for labeled systems based on F-models).
\begin{theorem}
Suppose \JL~is a justification logic with axiom {\bf jD}, \CS~is a constant specification for \JL, and ${\sf G3JL}^{Fk}_\CS$ is its labeled sequent calculus based on Fk-models.
\begin{enumerate}
\item All sequents of the form $\ww\Vvdash A,\g\R\d,\ww\Vvdash A$, with $A$ an arbitrary $\JL$-formula, are derivable in ${\sf G3JL}^{Fk}_\CS$.
\item  All rules of ${\sf G3JL}^{Fk}_\CS$ are height-preserving \CS-invertible.
\item  The rules of substitution $(Subs)$, weakening, contraction
and $Cut$ are height-preserving \CS-admissible in ${\sf G3JL}^{Fk}_\CS$.
\item If the sequent $\g\R\d$ is derivable in ${\sf G3JL}^{Fk}_\CS$, then it is valid in every $\JL_\CS$-Fk-model.
\item A \JL-formula $A$ is provable in $\JL_\CS$ iff $\R\ww\Vvdash A$ is provable in ${\sf G3JL}^{Fk}_\CS$.
\end{enumerate}
 \end{theorem}
\begin{theorem}
\begin{enumerate}
\item  All labeled formulas in a derivation
in ${\sf G3JD}^{Fk}_\CS$ (or in ${\sf G3JD4}^{Fk}_\CS$) are labeled-subformulas of labeled formulas in
the endsequent.
\item Every sequent derivable in ${\sf G3JD}^{Fk}_\CS$ (or in ${\sf G3JD4}^{Fk}_\CS$) has a derivation with the sublabel property.
\item  The rule  $(Trans_*)$ is admissible in ${\sf G3JD4}^{Fk}_\CS$ and its extensions.
\end{enumerate}
 \end{theorem}
The subterm property of E-rules (Lemma \ref{lem:subterm property E-rules}), and consequently the subterm property (Proposition \ref{prop:subterm property}), does not hold for these systems. For example, the sequent $\ww E(x,P\r \bot),\ww E(y,P)\R$ is derivable in ${\sf G3JD}^{Fk}_\emptyset$:
\begin{prooftree}
\def\extraVskip{5pt}
  \AXC{$(AxE\bot)$}\noLine
  \UIC{$\ww E(x\cdot y,\bot),\ww E(x,P\r \bot),\ww E(y,P)\R$}
  \RightLabel{$(E\cdot)$}
    \UIC{$\ww E(x,P\r \bot),\ww E(y,P)\R$}
    \end{prooftree}
 but has no derivation in which the rule $(E\cdot)$ has the subterm property. Also the countermodel construction of Theorem \ref{thm:reduction tree} cannot be used to produce Fk-models (in this respect, see  Note 3.3.35 and Example 3.3.36 of \cite{Kuznets2008}).
\subsection{Labeled sequent calculus for modal-justification logics}\label{sec:Labeled modal-justification}
 Modal-justification logics are combinations of modal and justification logics. We combine the modal logic {\sf ML} and justification logic \JL~provided that $\JL^\circ=\ML$, and in this case the respective modal-justification logic is denoted by {\sf MLJL}. The language of \MLJL~is an extension of the language of \JL~and {\sf ML}. Thus formulas of \MLJL~are constructed by the following grammar:
\[ A::= P~|~\bot~|~\neg A~|~A\wedge A~|~A\vee A~|~A\rightarrow A~|~t:A~|~\b A,\]
where $t\in Tm_{\JL}$. Axioms and rules of \MLJL~are the axioms and rules of {\sf ML} and \JL~and the \textit{connection axiom} $``t:A\r \b A"$, i.e.
 \[ \MLJL = \ML + \JL + (t:A\r\b A).\]
 We consider also the logic \N~which has the same language, axioms, and rules as \4~with one additional axiom $``\neg t:A\r\b\neg t:A"$, \textit{implicit-explicit negative introspection axiom}, i.e.
 \[ \N = \4 + (\neg t:A\r\b\neg t:A).\]
 Constant specifications for \MLJL~and \N, and the system $\MLJL_\CS$ and $\N_\CS$ are defined in the usual way.

The first modal-justification logic \4~was introduced by Artemov and Nogina in \cite{AN2004,AN2005b}, where \N~was also introduced there. The system {\sf KJ} was considered by Fitting in \cite{Fitting2012}\footnote{He also introduced the logic {\sf S5LPN} which is obtained from \N~by adding the modal axiom {\bf 5}.}, and by Artemov in \cite{A2012}.  Many of the modal-justification logics introduced here, are already considered in \cite{KuznetsStuder2012}, with the name logics of justifications and belief. However, the following modal-justification logics are new:
\begin{equation}\label{new MLJL logics}
{\sf S5JT5}, {\sf KB5JB4}, {\sf KB5JB45}, {\sf S5JTB5}, {\sf S5JDB5}, {\sf S5JTB45}, {\sf S5JDB45}, {\sf S5JTB4}, {\sf S5JDB4}.
\end{equation}
In fact we use the requirement $\JL^\circ=\ML$ in the definition of \MLJL~which is more general than the definition of  logics of justifications and belief in \cite{KuznetsStuder2012}.\footnote{It is worth noting that the logics {\sf TLP}, {\sf S4LP}, {\sf S5LP}, and the logic {\sf S5JT45} were introduced respectively by Artemov in \cite{A2006} and by Rubtsova in \cite{Rubtsova2006b}, with the difference that they consider multi-agent modal logics for its modal part, and thus these logics cannot be considered neither as  modal-justification logics nor as logics of justifications and belief.} The following lemma from \cite{GoetschiKuznets2012} is our justification to define the modal-justification logics in (\ref{new MLJL logics}).
\begin{lemma}\label{lem:operation replacement}(\cite{GoetschiKuznets2012})
There exist terms $t_!(x)$, $t'_!(x)$, $t_?(x)$ such that for any term $s$ and formula $A$ we have
\begin{enumerate}
\item ${\sf JT5} \vdash s:A\r t_!(s): s: A$.
\item ${\sf JB5} \vdash s:A\r t'_!(s): s: A$.
\item ${\sf JB4} \vdash \neg s:A\r t_?(s):\neg s: A$.
\item ${\sf JDB4} \vdash s:A\r  A$.
\item ${\sf JDB5} \vdash s:A\r  A$.
\end{enumerate}
\end{lemma}
 For example, although the logic {\sf S5JT5} has the modal axiom {\bf 4}, $\b A\r\b\b A$,  and does not have the justification axiom {\bf j4}, by the above lemma $s:A\r t_!(s): s: A$ is a theorem of {\sf S5JT5}.
 \begin{example}
If finite or injective constant specifications are used in proofs, it is easy to show that
\begin{eqnarray*}
  \4 &\rightleftharpoons& \s~+~\LP~+~ t:A\r \b t:A, \\
  \N &\rightleftharpoons& \s~+~\LP~+~ \b t:A\vee \b \neg t:A.
\end{eqnarray*}
\end{example}
In \cite{Ghari2011-JLC} it is shown that theorems of \4~and \N~can be realized respectively in \LP~and ${\sf JT45}$.
F-models for \MLJL~(except those listed in (\ref{new MLJL logics})) were presented in \cite{Fitting2004,KuznetsStuder2012}. We introduce F-models for others, as well as for \N, and show their completeness.
\begin{definition}\label{def: F-model MLJL}
 An $\MLJL_\CS$-model $\M=(\W,\RR,\E,\V)$ is an $\JL_\emptyset$-model, where now $\E$ is a possible evidence function on $\W$ for $\MLJL_\CS$. In addition, if \MLJL~contains the modal axiom {\bf 5} then the accessibility relation $\RR$ should be Euclidean. A $\N_\CS$-model is a $\4_\emptyset$-model such that  $\E$ is a possible evidence function on $\W$ for $\N_\CS$ which meets the strong evidence ($\E 8$) and anti-monotonicity ($\E 9$) conditions. The definition of forcing relation (Definition \ref{def:forcing relation}) for these models has the following additional clause:
\[ (\M,w)\Vdash \b A~~~ \mbox{iff for every $v\in \W$ with $w \RR v$,} ~~~(\M, v)\Vdash A.\]
\end{definition}
We first show soundness and completeness of \N.
\begin{theorem}\label{compl S4LPN-F-model}
 $\N_\CS$ are sound and complete with respect to their $\N_\CS$-models.
\end{theorem}
\begin{proof} For soundness direction of \N, let us show the
validity of implicit-explicit negative introspection axiom $\neg t:A\r\b\neg t:A$. Suppose $(\M,w)\Vdash \neg t:A$ and $w\RR v$. By the strong
evidence condition, we have $w\not\in\E(t,A)$. Hence by the
anti-monotonicity condition, $v\not\in\E(t,A)$. Thus
$(\M,v)\Vdash\neg t:A$, and therefore $(\M,w)\Vdash\b \neg t:A$.

The proof of completeness is similar to that of \4 in \cite{Fitting2004}. Given a constant specification $\CS$ for \N, we
construct a canonical F-model $\M=(\W,\RR,\E,\V)$ for
$\N_\CS$ as follows: Let $\W$ be all
$\N_\CS$-maximally consistent sets, and define the accessibility
relation $\RR$ on $\W$ by, $\g\RR\d$ if{f}
$\g^\sharp\subseteq\d$, where $\g^\sharp=\{A~|~\b A\in\g\}$.
The evidence function $\E$ and the valuation $\V$ are
defined similar to the canonical model of \JL~in the proof of
Theorem \ref{Sound Compl JL}.
 Truth Lemma can be proved easily: for each formula
$F$ and each $\g\in\W$,
\[(\M,\g) \Vdash F \Longleftrightarrow F\in\g.\]
We only verify the case in which $F$ is of the form $t:A$ and $\b A$.

If $(\M,\g) \Vdash t:A$, then $\g\in\E(t,A)$, and therefore $t:A\in\g$. Conversely, suppose $t:A\in\g$. Then $\g\in\E(t,A)$. Since $t:A\r\b A\in\g$, we have $\b A\in\g$. By the definition of $\RR$, $ A\in\d$ for each $\d$ such that $\g\RR\d$. By the induction hypothesis, $(\M,\d) \Vdash A$. Therefore, $(\M,\g) \Vdash t:A$.

If $\b A\in\g$, then $A\in\d$ for each $\d$ such that $\g\RR\d$. By the induction hypothesis, $(\M,\d) \Vdash A$. Therefore, $(\M,\g) \Vdash \b A$. Conversely, suppose $\b A\not\in\g$. Then $\g^\sharp\cup \{\neg A\}$ is a consistent set. If it were not consistent, then $\N_\CS\vdash  B_1\wedge B_2\wedge\ldots\wedge B_n\r A$ for some $\b B_1,\b B_2,\ldots,\b B_n\in\g$. By reasoning in \N~we have $\N_\CS\vdash \b B_1\wedge\b B_2\wedge\ldots\wedge\b B_n\r \b A$, hence $\b A\in\g$ a contradiction. Thus $\g^\sharp\cup \{\neg A\}$ is consistent. Extend it to a maximal consistent $\d$. It is obvious that $\d\in\W$, $\g\RR \d$ and $A\not\in\d$.
By the induction hypothesis, $(\M,\d) \not\Vdash A$. Therefore, $(\M,\g) \not\Vdash \b A$.

Obviously, $\RR$ is reflexive and transitive, and $\E$ satisfies
$\E1$-$\E4$. The proof of the strong evidence condition ($\E 8$) for $\E$ is similar  to that  given in Remark \ref{Remark: Strong Evidence}. We verify the anti-monotonicity  condition ($\E9$).

\textit{$\E$ satisfies $\E 9$:} suppose $\g\RR\d$ and
$\d\in\E(t,A)$. By the definition of $\E$, we have $t:A\in\d$.
Let us suppose $\g\not\in\E(t,A)$, or equivalently
$t:A\not\in\g$, and derive a contradiction. Since $t:A$ is not in
$\g$, and $\g$ is a maximal consistent set, we have $\neg
t:A\in\g$. By implicit-explicit negative introspection axiom, $\b\neg t:A\in\g$, and
so $\neg t:A\in\g^\sharp$. Now $\g\RR\d$ yields $\neg
t:A\in\d$, which is a contradiction.\qe \end{proof}
\begin{theorem}\label{compl MLJL-F-model}
Let \MLJL~be a modal-justification logic, and $\CS$ be a constant specification for \MLJL, with the requirement that if \JL~includes axiom scheme {\bf jD} then \CS~should be also axiomatically appropriate.
$\MLJL_\CS$ are sound and complete with respect to their F-models.
\end{theorem}
\begin{proof} The proof is similar to the proof of Theorem \ref{Sound Compl JL}.
Soundness follows from soundness of modal and justification logics. Soundness of logics listed in (\ref{new MLJL logics}) needs attention. For example,  since reflexivity and Euclideanness of the accessibility relation of ${\sf S5JT5}_\CS$-models implies transitivity, axiom {\bf 4} is valid in ${\sf S5JT5}_\CS$-models.

 For completeness, we construct a canonical model $\M=(\W, \RR,
\E, \V)$ for each $\MLJL_\CS$. Let $\W$ be the set of all maximal
consistent sets in $\MLJL_\CS$, and define the accessibility
relation $\RR$, the evidence function $\E$ and the valuation $\V$ are
defined similar to the canonical model of \N~in the proof of
Theorem \ref{compl S4LPN-F-model}. The forcing relation $\Vdash$ on arbitrary formulas are defined as Definitions \ref{def:forcing relation} and \ref{def: F-model MLJL}. The Truth Lemma can
be shown for the canonical model of $\MLJL_\CS$ similar to the proof of Theorem \ref{compl S4LPN-F-model}. For each modal-justification logic $\MLJL_\CS$, it is easy to see that the canonical model $\M$ of $\MLJL_\CS$ is an $\MLJL_\CS$-model.  \qe \end{proof}

In order to develop labeled systems for \MLJL~and \N, we extend the extended labeled language of justification logics to include labeled
formulas $\ww\Vvdash A$, in which $A$ is a formula in the language
of \MLJL~and \N, respectively. Now, labeled systems for \MLJL~and \N~based on F-models are defined as follows:
\begin{eqnarray*}
  {\sf G3\MLJL} &=& {\sf G3\ML} + {\sf G3JL}. \\
  {\sf G3\N} &=& {\sf G3\4}+ (Anti\mbox{-}Mon) + (SE),
\end{eqnarray*}
Put differently, if {\bf 5} is not an axiom of \ML, then {\sf G3MLJL} is obtained by adding the rules $(L\b)$, $(R\b)$ to \J. Otherwise, if {\bf 5} is an axiom of \ML, {\sf G3MLJL} is obtained by adding the rules $(L\b)$, $(R\b)$, $(Eucl)$, and $(Eucl_*)$ to \J. Labeled systems ${\sf G3MLJL}_\CS$ and ${\sf G3\N}_\CS$ are defined in the usual way.

We now state the main results of the labeled systems of \MLJL~and \N. First note that  the definition of a labeled-subformula of a labeled formula (Definition \ref{def:subformula labeled formula}) can be extended as follows: the labeled-subformulas of
$\ww\Vvdash\b A$ are $\ww\Vvdash\b A$ and all labeled-subformulas of $\vv\Vvdash A$ for arbitrary label $\vv$.
\begin{theorem}\label{thm:properties of G3MLJL}
Let ${\sf G3MLJL}^-$ denote any of the labeled systems {\sf G3KJ}, {\sf G3K4J4}, {\sf G3DJD}, {\sf G3D4JD4}, {\sf G3TJT}, {\sf G3S4LP}.
\begin{enumerate}
\item All sequents of the form $\ww\Vvdash A,\g\R\d,\ww\Vvdash A$, with $A$ an arbitrary $\MLJL$-formula, are derivable in ${\sf G3MLJL}_\CS$.
\item  All labeled formulas in a derivation
in ${\sf G3MLJL}^-_\CS$ are labeled-subformulas of labeled formulas in
the endsequent.
\item Every sequent derivable in ${\sf G3MLJL}^-_\CS$ has an anlytic derivation.
\item  All rules of ${\sf G3MLJL}_\CS$ are height-preserving \CS-invertible.
\item  The rule $(Trans_*)$ is admissible in those systems ${\sf G3MLJL}_\CS$ which contain $(Trans)$.
\item  The rules of substitution $(Subs)$,  weakening, contraction
and $(Cut)$ are \CS-admissible in ${\sf G3MLJL}_\CS$.
\item If the sequent $\g\R\d$ is derivable in ${\sf G3MLJL}_\CS$, then it is
valid in every $\MLJL_\CS$-model.
\item Let $A$ be a formula in the language of \MLJL, and \CS~be a constant specification for \MLJL~with the requirement that if \MLJL~contains axiom scheme {\bf jD} then \CS~should be also axiomatically appropriate. Then $A$ is provable in $\MLJL_\CS$ iff $\R\ww\Vvdash A$ is provable in ${\sf G3\MLJL}_\CS$.
\item Suppose \MLJL~is a modal-justification logic that does not contain axioms {\bf jB} and {\bf j5}, and \CS~is a finite constant specification for ${\sf MLJL}$. Then every sequent
$\g\R\d$ in the language of ${\sf G3MLJL}$ is either derivable in ${\sf G3MLJL}_\CS$ or it has a $\MLJL_\CS$-countermodel.
\end{enumerate}
All the above results, except clauses 2, 3,  hold if we replace \MLJL~and ${\sf G3MLJL}$  with \N~and ${\sf G3\N}$ respectively.
 \end{theorem}
\begin{proof}
The proofs are similar to the respective proofs in Sections 5-9, and the proofs for modal logics in \cite{Negri2005}. In fact, the cases for rules $(L\b)$ and $(R\b)$ are similar to rules $(L:)$ and $(R:)$, respectively.
\begin{enumerate}
\item Obvious.
\item Obvious.
\item We first need to show  that every sequent derivable in ${\sf G3MLJL}^-_\CS$ has a derivation with the subterm and sublabel property. The proof of the subterm  and sublabel property is similar to that given for Propositions  \ref{prop:sublabel property}, \ref{prop:subterm property}.
\item  The proof is similar to that of Proposition \ref{prop:inversion lemma}. We only check the invertibility of rule $(R:)$ where in the induction step $w E(t,A),\g\R\d,w\Vvdash t:A$ is the
conclusion of the rule $(R\b)$:
\[\di{\f{\uu R\uu',\ww E(t,A),\g\R\d',\uu'\Vvdash A,\ww\Vvdash t:A}{\ww E(t,A),\g\R\d',\uu\Vvdash \b A,\ww\Vvdash t:A}\,(R\b)}\]
where the eigenlabel $\uu'$ is not in the conclusion. By the
induction hypothesis we obtain a derivation of height $n-1$ of
\[\ww R\vv,\uu R\uu',\ww E(t,A),\g\R\d',\uu'\Vvdash A,\vv\Vvdash A,\]
 for any fresh
label $\vv$. Then by applying the rule $(R\b)$ we obtain a
derivation of height $n$ of
 \[\ww R\vv,\ww E(t,A),\g\R\d,\uu\Vvdash \b A,\vv\Vvdash A,\]
 as desire. The proof of invertibility of rule $(R\b)$ is similar to that for $(R:)$.
 \item The proof is by induction on the height of the derivation of $\ww R\ww,\ww R\ww,\g\R\d$ in $\MLJL_\CS$. In the induction step, $\ww R\ww$ may be principal in the rules $(Eucl)$, $(Eucl_*)$, $(L\b)$. The proof for cases  $(Eucl)$ and $(Eucl_*)$ is similar to the case of rule $(Trans)$, and for the case $(L\b)$ is similar to the case of rule $(L:)$ in the proof of Lemma \ref{lemma: admissibility of Trans*}. The case where $\ww R\ww$ is principal in $(Anti$-$Mon)$ in ${\sf G3\N}_\CS$ is similar to the case of rule $(Mon)$ in the proof of Lemma \ref{lemma: admissibility of Trans*}.
 \item The proofs are similar to the respective proofs in Sections 5-9, and the proofs for modal logics in \cite{Negri2005}.
 \item Similar to the proof of Theorem \ref{Soundness labeled systems}. The validity-preserving of  rules $(L\b)$ and $(R\b)$ have already been shown in Theorem 5.3 in \cite{Negri2009}, and that of rules $(Eucl)$, $(Eucl_*)$ (and $(Anti$-$Mon)$ in ${\sf G3\N}_\CS$) are obvious.
 \item Similar to the proof of Corollary \ref{cor:equivalence JL and G3JL}. As an example, we prove the connection principle:
\begin{prooftree}
\def\extraVskip{3pt}
\AXC{${\mathcal D}$}\noLine
 \UIC{$\vv\Vvdash A,\ww R\vv,\ww\Vvdash t:A\fCenter \vv\Vvdash A$}
 \RightLabel{$(L:)$}
  \UIC{$\ww R\vv,\ww\Vvdash t:A\R \vv\Vvdash A$}
  \RightLabel{$(R\b)$}
  \UIC{$\ww\Vvdash t:A\R \ww\Vvdash \b A$}
  \RightLabel{$(R\r)$}
  \UIC{$\R \ww\Vvdash t:A\r \b A$}
\end{prooftree}
where the eigenlabel $\vv$ is different from $\ww$ and ${\mathcal D}$
is the derivation of the topmost sequent by clause 1.
 \item Similar to the proof of Theorem \ref{thm:reduction tree}. Additional stages for rules $(L\b)$, $(R\b)$, $(Eucl)$, $(Eucl_*)$ (and $(Anti$-$Mon)$ for ${\sf G3\N}_\CS$) should be added to the construction of the reduction tree.\qe
\end{enumerate}
\end{proof}
We close this section by showing the termination of proof search for ${\sf G3KJ}_\CS$, ${\sf G3TJT}_\CS$, ${\sf G3\4}_\CS$, for finite $\CS$.
For any given sequent, let $n(\b)$ and $p(\b)$ be the number of occurrences of
$\b$ in the negative and positive part of the
sequent.

Termination of proof search for ${\sf G3KJ}_\CS$ and ${\sf G3TJT}_\CS$ follows from Theorems \ref{thm:termination proof search G3J}, \ref{thm:termination proof search G3JT}
\begin{theorem}
Given any finite constant specification $\CS$, and any sequent
$\g\R\d$ in the language of {\sf G3KJ}, it is decidable whether the sequent is derivable in ${\sf G3KJ}_\CS$.
\end{theorem}
\begin{theorem}
Given any finite constant specification $\CS$, and any sequent
$\g\R\d$ in the language of {\sf G3TJT}, it is decidable whether the sequent is derivable in ${\sf G3TJT}_\CS$.
\end{theorem}
For {\sf G3\4}, we need bounds on the number of applications of $(R:)$ and $(R\b)$, similar to that given for $(R:)$ in {\sf G3LP} in Proposition \ref{bound R: in LP}. For {\sf G3\4}, since rule $(L\b)$, as well as rule $(L:)$, can be used in the argument given in the proof of Proposition \ref{bound R: in LP}, we have:
\begin{proposition}\label{bound R: in S4LP}
In a derivation of a sequent in {\sf G3\4} for each
formula of the form $t:A$ in its positive part, it is enough to have at most
$n(:)+n(\b)$ applications of $(R:)$ iterated on a chain of
accessible worlds $\ww R\ww_1 ,\ww_1 R\ww_2,\ldots$, with principal formula
$\ww_i\Vvdash t:A$.
\end{proposition}
Since the proof of the following proposition repeats the proof of
Proposition 6.9 in \cite{Negri2005}, we omit it here.

\begin{proposition}\label{bound Rbox in S4LP}
In a derivation of a sequent in ${\sf G3S4LP}$, for each
formula of the form $\b A$ in its positive part, it is enough to have at most
$n(:)+n(\b)$ applications of $(R\b)$ iterated on a chain of
accessible worlds $\ww R\ww_1 ,\ww_1 R\ww_2,\ldots$, with principal formula
$\ww_i\Vvdash\b A$.
\end{proposition}
The system {\sf G3S4LP} combines {\sf G3LP} and {\sf G3S4}. As
you can see from the proof of connection principle in {\sf G3S4LP} (see the proof of Theorem \ref{thm:properties of G3MLJL}(8)), in the backward
proof search applications of $(R\b)$ introduce relational atoms,
and therefore can increase the number of applications of $(L:)$.
Similarly, applications of $(R:)$ can produce new applications of
$(L\b)$. In spite of this fact, by Propositions \ref{bound R: in
S4LP} and \ref{bound Rbox in S4LP} the number of applications of
$(R:)$ and $(R\b)$, and consequently of $(L:)$, $(L\b)$, and
$(AN)$ are bounded. Thus, by Theorem \ref{terminating proof search
LP} and termination of proof search of {\sf G3S4} (Corollary 6.10
in \cite{Negri2005}), we have
\begin{theorem}
Given any finite constant specification $\CS$, and any sequent
$\g\R\d$ in the language of {\sf G3S4LP}, it is decidable whether the sequent is derivable in ${\sf G3S4LP}_\CS$.
\end{theorem}
\subsection{Labeled sequent calculus based on AF-models}
In this section we recall Artemov-Fitting models (or AF-models)  for \4~and \N, which were first introduced in
\cite{AN2005b}, and then we introduce labeled sequent calculus based on AF-models for \4~and \N.
\begin{definition}\label{Def AF-models}
A structure $\M=(\W, \RR, \RR^e, \E, \V)$ is an $\4_\CS$-AF-model
if $(\W, \RR,  \E, \V)$ is an $\4_\CS$-model and $\RR^e$ is a
reflexive and transitive evidence accessibility relation, such
that $\RR\subseteq \RR^e$. Here, the Monotonicity property should
be read as follows:
\begin{itemize}
 \item Monotonicity: If $w\in\E(t,A)$ and $w\RR^e v$, then $v\in\E(t,A)$.
\end{itemize}
Moreover, the forcing relation on formulas of the form $\b A$ and
$t:A$ are defined as follows:
\begin{itemize}
 \item $(\M,w)\Vdash \b A$ iff for every $v\in \W$ with $w \RR v$, $(\M, v)\Vdash A$,
 \item $(\M, w)\Vdash t:A$ iff $w\in\E(t,A)$ and for every $v\in \W$ with $w \RR^e v$, $(\M, v)\Vdash A$.
\end{itemize}
\end{definition}
$\N_\CS$-AF-models are \4-AF-models where $\E$ is a possible evidence function on $\W$ for $\N_\CS$ and $\RR^e$ is also symmetric.
\begin{theorem}\label{compl S4LP-AF-model}(\cite{AN2005b})
$\4_\CS$ and $\N_\CS$ are
sound and complete with respect to their AF-models.
 \end{theorem}
 By internalizing AF-models of \4~and
\N~into the syntax of labeled
systems, we will obtain a labeled system with the labeled-subformula
property for \N~(as Example \ref{ex:counterexample subformula G3J5} shows the labeled-subformula property does not hold for {\sf G3\N}). In the following, we will define labeled systems
${\sf G3\4}^e$ and ${\sf G3\N}^e$.

Let us first extend the extended labeled language of justification logics
by labeled formulas $\ww\Vvdash A$, in which $A$ is a formula in the
language of \N, and \textit{evidence relational atoms} $\ww R^e \vv$
(which presents the evidence relation $w\RR^e v$ in AF-models).
System ${\sf G3\N}^e$ is an extension of the labeled sequent calculus ${\sf G3\s}$  with initial sequents and rules from Table \ref{table: rules for systems
with AF-models}. Rule $(R^e)$ reflects the condition $R\subseteq R^e$
in AF-models. Again, the initial sequent $(AxR^e)$ is added
to conclude the properties of the evidence accessibility relation.
System ${\sf G3\4}^e$ is obtained from ${\sf G3\N}^e$ by removing the rule $(Sym^e)$. In other words,
\[ {\sf G3\N}^e = {\sf G3\4}^e + (Sym^e).\]
Labeled systems ${\sf G3\4}^e_\CS$ and ${\sf G3\N}^e_\CS$ are defined in the usual way.
\begin{table}[t]
\centering\renewcommand{\arraystretch}{2}
\begin{tabular}{|lc|}
\hline
~\noindent{\bf Initial sequent:}&\\
~$\ww R^e \vv,\g\R\d, \ww R^e \vv$~~~~~$(AxR^e)$ & $\ww E(t,A),\g\R\d, \ww E(t,A)$~~~~~$(AxE)$\\
~\noindent {\bf Rules:}&\\
~$\di{\f{\vv\Vvdash A,\ww\Vvdash t:A,\ww R^e \vv,\g\R\d}{\ww\Vvdash
t:A,\ww R^e \vv,\g\R\d}}(L:^e)$ & $\di{\f{\ww R^e \vv,\ww E(t,A),\g\R\d,\vv\Vvdash A}{\ww E(t,A),\g\R\d,\ww\Vvdash
t:A}}(R:^e)$\\
\multicolumn{2}{|c|}{$\di{\f{\ww E(t,A),\ww\Vvdash
t:A,\g\R\d}{\ww\Vvdash
t:A,\g\R\d}}(E)$}\\
\multicolumn{2}{|c|}{In $(R:^e)$ the eigenlabel $\vv$ must not occur in
the conclusion of rule.}\\
~\noindent {\bf Rules for evidence atoms:}&\\
\multicolumn{2}{|c|}{$\di{\f{\ww E(s\cdot t,B),\ww E(s,A\r B),\ww E(t,A),\g\R\d}{\ww E(s,A\r B),\ww  E(t,A),\g\R\d}}(E\cdot)$}\\
   ~ $\di{\f{\ww E(s+t,A),\ww E(t,A),\g\R\d}{\ww E(t,A),\g\R\d}}(El+)$ & $\di{\f{\ww E(t+s,A),\ww E(t,A),\g\R\d}{\ww E(t,A),\g\R\d}}(Er+)$\\
  ~  $\di{\f{\ww E(!t,t:A),\ww E(t,A),\g\R\d}{\ww   E(t,A),\g\R\d}}(E!)$ &
   $\di{\f{\vv E(t,A),\ww E(t,A),\ww R^e \vv,\g\R\d}{\ww E(t,A),\ww R^e \vv,\g\R\d}}(Mon^e)$~ \\
  ~  \noindent {\bf Axiom necessitation rule:}&\\
\multicolumn{2}{|c|}{$\di{\f{\ww E(c,A),\g\R\d}{\g\R\d}}(AN)$}\\
~\noindent {\bf Rules for evidence relational atoms:}&\\
~$\di{\f{\ww R^e \vv,\ww R\vv,\g\R\d}{\ww R\vv,\g\R\d}\,(R^e)}$ &
$\di{\f{\ww R^e \ww,\g\R\d}{\g\R\d}}(Ref^e)$\\
~ $\di{\f{\vv R^e \ww,\ww R^e \vv,\g\R\d}{\ww R^e \vv,\g\R\d}}(Sym^e)$ &$\di{\f{\ww R^e \uu,\ww R^e \vv,\vv R^e \uu,\g\R\d}{\ww R^e \vv,\vv R^e
\uu,\g\R\d}}(Trans^e)$\\
\hline
\end{tabular}\vspace{0.3cm}
\caption{Initial sequent and rules which should be added to {\sf G3S4} to obtain ${\sf G3S4LPN}^e$.}\label{table: rules for systems with AF-models}
\end{table}
All the results of Sections 5-9 can be extended to the systems ${\sf G3\4}^e_\CS$ and ${\sf G3\N}^e_\CS$. Note that Definition \ref{validity of formula and sequence} is extended to the labeled systems ${\sf G3\4}^e_\CS$ and ${\sf G3\N}^e_\CS$ as follows. For AF-model
$\M=(\W,\RR,\RR^e,\E,\V)$ and $\M$-interpretation $[\cdot]:L\r
\W$, we add the following clause to Definition \ref{validity of formula and sequence}:
 \begin{itemize}
 \item $[\cdot]$ validates the evidence relational atom $\ww R^e \vv$, provided that $[\ww]\RR^e [\vv]$.
 \end{itemize}
 Inductively generated evidence functions for systems $\4_\CS$ and $\N_\CS$ based on a possible evidence function (on Kripke frame $(\W,\RR,\RR^e)$), is defined similar to the one for \LP~with the difference that we replace clause (7) in Definition \ref{def:generated evidence function} by the following one:
\begin{description}
\item[$6.$] $\E_{i+1} (t,F)=\E_{i} (t,F)\cup\{w\in\W~|~v \in\E_i(t,F), v\RR^e w\}$.
\end{description}
We now state the main properties of ${\sf G3\4}^e_\CS$ and ${\sf G3\N}^e_\CS$.
 \begin{theorem}
\begin{enumerate}
\item All sequents of the form $\ww\Vvdash A,\g\R\d,\ww\Vvdash A$, with $A$ an arbitrary $\N$-formula, are derivable in  ${\sf G3\N}^e_\CS$.
\item  All labeled formulas in a derivation in  ${\sf G3\N}^e_\CS$ are labeled-subformulas of labeled formulas in the endsequent.
\item All formulas in a derivation in  ${\sf G3\N}^e_\CS$ are either labeled-subformulas of a labeled formula in the endsequent or atomic formulas of the form $\ww E(t,A)$, $\ww R\vv$, or $\ww R^e \vv$.
\item Every sequent derivable in ${\sf G3\N}^e_\CS$ has an anlytic derivation.
\item  All rules of ${\sf G3\N}^e_\CS$ are height-preserving \CS-invertible.
\item The rule
\[\di{\f{\ww R^e \ww,\ww R^e \ww,\g\R\d}{\ww R^e \ww,\g\R\d}\,(Trans^e_*)}\]
is height-preserving $\CS$-admissible in ${\sf G3\N}^e_\CS$.
\item  The rules of substitution $(Subs)$, $(Trans_*)$, weakening, contraction
and $Cut$ are \CS-admissible in  ${\sf G3\N}^e_\CS$.
\item If the sequent $\g\R\d$ is derivable in ${\sf G3\N}^e_\CS$, then it is
valid in every $\N_\CS$-AF-model.
\item Let $A$ be a formula in the language of \N, and \CS~be a constant specification for \N. Then $A$ is provable in $\N_\CS$ iff $\R\ww\Vvdash A$ is provable in ${\sf G3\N}^e_\CS$.
\item Given any finite constant specification $\CS$ for \N, every sequent
$\g\R\d$ in the language of {\sf G3\N} is either derivable in ${\sf G3\N}^e_\CS$ or it has a $\N_\CS$-AF-countermodel.
\end{enumerate}
All the above results hold if we replace \N~and ${\sf G3\N}^e$ respectively with \4~and ${\sf G3\4}^e$.
 \end{theorem}
\begin{proof} The proof of these clauses are similar to those for \J. Only note that item 6 follows easily from the fact that ${\sf G3\4}^e$ and ${\sf G3\N}^e$ contain the rule $(Ref^e)$. For clause 10, we construct the reduction tree similar to \J~with stages correspond to rules of ${\sf G3\N}^e_\CS$ (and ${\sf G3\4}^e_\CS$). The stages of new rules $(L:^e)$, $(R:^e)$, $(Mon^e)$, $(Trans^e)$, and $(Ref^e)$ are similar to the stages of $(L:)$, $(R:)$, $(Mon)$, $(Trans)$, and $(Ref)$ respectively.

For stage of rule $(R^e)$, if the top-sequent is of the form:
\[ \ww_1 R \vv_1,\ldots,\ww_m R \vv_m,\g'\R\d'\]
where all relational atoms $\ww_i R \vv_i$ from the antecedent of the
topmost sequent are listed. Then, regarding condition $(\dagger)$,
we write the following node on top of it:
\[ \ww_1 R^e \vv_1,\ldots,\ww_m R^e \vv_m,\ww_1 R \vv_1,\ldots,\ww_m R \vv_m,\g'\R\d'\]

For stage of rule $(Sym^e)$, if the top-sequent is of the form:
\[ \ww_1 R^e \vv_1,\ldots,\ww_m R^e \vv_m,\g'\R\d'\]
where all relational atoms $\ww_i R^e \vv_i$ from the antecedent
of the topmost sequent are listed, then, regarding condition
$(\dagger)$, we write the following node on top of it:
\[ \vv_1 R^e \ww_1 ,\ldots,\vv_mR^e \ww_m ,\ww_1 R^e \vv_1,\ldots,\ww_m R^e \vv_m,\g'\R\d'\]

If the reduction tree has a saturated branch, in order to
construct an Artemov-Fitting countermodel
$\M=(\W,\RR,\RR^e,\E_\mathcal{A},\V)$, we add the following clause to the definition of countermodel in the proof of Theorem \ref{thm:reduction tree}:
\begin{description}
\item[$5.$] The evidence accessibility relation $\RR^e$ is determined by
evidence relational atoms in $\overline{\g}$ as follows: if
$\ww R^e \vv$ is in $\overline{\g}$, then $\ww \RR^e \vv$ (otherwise $\ww\RR^e
\vv$ does not hold).
\end{description}
The rest of the proof is similar to that in Theorem \ref{thm:reduction tree}.
\qe \end{proof}

Terminating of proof search of  ${\sf G3\N}^e_\CS$ (and ${\sf G3\4}^e_\CS$) for finite $\CS$ follows from the following propositions.
\begin{proposition}
In a derivation of a sequent in  ${\sf G3S4LPN}^e_\CS$ for each formula of the form $t:A$ in its positive part, it is enough to have at most $n(:)$ applications of $(R:^e)$ iterated on
a chain of accessible worlds $\ww R^e\ww_1 ,\ww_1 R^e \ww_2,\ldots$, with
principal formula $\ww_i\Vvdash t:A$. The same holds for ${\sf G3S4LP}^e_\CS$.
\end{proposition}
\begin{proposition}\label{bound Rbox in S4LP^e}
In a  derivation of a sequent in  ${\sf
G3S4LPN}^e_\CS$, for each formula of the form $\b A$ in its positive
part, it is enough to have at most $n(\b)$ applications of $(R\b)$ iterated
on a chain of accessible worlds $\ww R\ww_1 ,\ww_1 R\ww_2,\ldots$, with
principal formula $\ww_i\Vvdash\b A$. The same holds for ${\sf G3S4LP}^e_\CS$.
\end{proposition}
Thus, similar to the proof of termination of proof search of {\sf G3LP}, restrictions on the number of applications of $(R:^e)$ and $(R\b)$ should be imposed in the
reduction tree construction for the system ${\sf G3\N}^e_\CS$ (and ${\sf G3\4}^e_\CS$).
\begin{theorem}
Given any finite constant specification $\CS$, and any sequent
$\g\R\d$ in the language of ${\sf G3S4LPN}^e_\CS$, it is decidable whether the sequent is derivable in ${\sf G3S4LPN}^e_\CS$. The same holds for ${\sf G3S4LP}^e_\CS$.
\end{theorem}
\begin{proof} The termination of proof search for ${\sf G3S4LP}^e_\CS$ follows from
an argument similar to those of {\sf G3\s} and {\sf G3LP}, and the
fact that the number of applications of rule $(R^e)$ is bounded by a function of
$r$ and of the number of applications of rules $(R\b)$, $(Ref)$, $(Trans)$. Rules $(Ref^e)$ and $(Trans^e)$ are treated similar to
$(Ref)$ and $(Trans)$, respectively, in {\sf G3LP}. For ${\sf
G3S4LPN}^e_\CS$, the rule $(Sym^e)$ increase the number of
applications of $(L:^e)$. The number of applications of $(Sym^e)$
is bounded by a function of the number of evidence relational atoms in the
antecedent of the root sequent and of the number of applications of rules
$(R^e)$, $(Ref^e)$, $(Trans^e)$, $(R:^e)$. \qe \end{proof}
\section{Conclusion}
The main achievement of this paper has been to provide a modular
approach to the proof theory of justification logics. We have
presented contraction- and cut-free labeled sequent calculus for all justification
logics. All of the presented labeled systems enjoy the sublabel property, and some of them enjoys in addition the labeled-subformula property and the subterm property. Thus, analyticity and termination of proof search were proved only
for some of the labeled systems.

Some of the main advantages of giving labeled proof systems for justification logics over pure syntactic proof systems are as follows: they enable us to provide a systematic
approach to the proof theory of justification logics and also of modal-justification logics; they enable us to construct (Fitting) countermodels for non-valid formulas; they enable us to give a form of correspondence theory for justification logics; they enable us to formalize possible world semantic arguments in a sequent calculus.

The method described in
this paper can be used to provide labeled sequent calculus for other justification logics that have Kripke-Fitting-style models, such as multi-agent logics ${\sf T_nLP}$,
${\sf S4_nLP}$ and ${\sf S5_nLP}$ (cf. \cite{A2006}).

It is also
possible to internalize Mkrtychev models
\cite{A2008,Mkrtychev1997}, which are singleton Fitting models,
within the syntax of sequent calculus to produce sequent systems
for justification logics. Moreover, using Mkrtychev models we
may create \textit{label-free} sequent systems. To this end
replace $\ww\Vvdash A$ and $\ww E(t,A)$ with $A$ and
$E(t,A)$, respectively, and omit relational atoms $\ww R\vv$ from the
labeled language, and then change initial sequents and rules of
Tables \ref{table: rules for G3J},\ref{table: rules for G3JL} accordingly. We leave the precise formulation of these systems for another work.

There remain still some questions. How could one extend these
results to find labeled systems based on F-models for {\sf JB} and
{\sf J5} and their extensions such that termination of proof search is proved? In other words, does the subterm property hold for the labeled systems {\sf G3JB}, {\sf G3J5} and their extensions? In this paper, our approach was to internalize the  known Fitting models of the justification logics. Of course one could try to give labeled systems based on other semantics.

Negri presented a cut-free labeled
system for provability logic {\sf GL} (cf. \cite{Negri2005}). Is
it possible to extend it to the \textit{logic of proofs and
provability} {\sf GLA}~(\cite{AN2004,Nogina 2007})?\\

\noindent
{\bf Acknowledgments}\\

Most of the results of this paper originate from \cite{Ghari2012-Thesis}, and part of the results presented here have been obtained while visiting the Logic and Theory Group (LTG) at the University of Bern in 2009. The author would like to thank Gerhard J\"{a}ger and Thomas Strahm for their helps, to Roman Kuznets and Kai Br\"{u}nnler for useful arguments.
I would like to thank Melvin Fitting and Remo Goetschi 
 for useful comments and catching several errors. This research was in part supported by a grant from IPM. (No. 93030416)


\end{document}